\documentclass{article}
\usepackage[a4paper,top=4cm,bottom=4cm,left=2.5cm,right=2.5cm,bindingoffset=5mm]{geometry}
\usepackage{xcolor}

\usepackage[hyphens]{url}\usepackage[breaklinks=true,colorlinks,linkcolor={blue},citecolor={red},urlcolor={purple},]{hyperref}

\usepackage[font=small,format=default,indention=2em,labelsep = period]{caption} 
\usepackage{subfig}
\usepackage{enumitem}
\usepackage{todonotes}
\usepackage{overpic}
\tikzset{>=latex}

\usepackage{algorithm}
\usepackage{algpseudocode}

\usepackage{amssymb,bm,amsmath,amsthm}
\usepackage{graphicx,mathtools}	

\newcommand{\cate}[1]{\textcolor{purple}{#1}}

\newcommand{\disp}{\displaystyle}     
\newcommand{\de}{\partial}        
\newcommand{\norm}[1]{\left\lVert#1\right\rVert}
\def\rhoc{\rho_{c}}
\def\rhof{\rho_{f}}
\def\rhom{\rho^{\mathrm{max}}}
\def\vmax{V^{\mathrm{max}}}
\def\dx{\Delta x}
\def\dy{\Delta y}
\def\dz{\Delta z}
\def\dt{\Delta t}

\def\wl{w_{L}}
\def\wr{w_{R}}
\def\no{\mathrm{NO}}
\def\nodue{\mathrm{NO_{2}}}
\def\nox{\mathrm{NO_{x}}}
\def\ouno{\mathrm{O}}
\def\odue{\mathrm{O_{2}}}
\def\otre{\mathrm{O_{3}}}
\def\cc{\psi}
\def\CC{\Psi}
\newcommand{\myunit}[1]{#1\,}
\def\mysecond{\mathrm{s}}
\def\mymeter{\mathrm{m}}
\def\mykilo{\mathrm{k}}
\def\mycenti{\mathrm{c}}
\def\mymin{\mathrm{min}}
\def\per{/}
\def\km{\mathrm{km}}
\def\myhour{\mathrm{h}}
\def\mygram{\mathrm{g}}
\def\veh{\mathrm{veh}}
\def\vento{\vec{C}}
\def\ventox{c_{x}}
\def\ventoy{c_{y}}
\def\lx{L_{x}}
\def\ly{L_{y}}
\def\vol{\mathcal{V}}

\def\cuno{C^{1}}

\usepackage{titlesec}
\let\oldparagraph=\paragraph
\renewcommand\paragraph[1]{\oldparagraph{#1.}}

\numberwithin{equation}{section}
\newtheorem{remark}{Remark}
\newtheorem{proposition}{Proposition}
\numberwithin{remark}{section}

\makeatletter
\newcommand{\labeltext}[3][]{%
    \@bsphack%
    \csname phantomsection\endcsname
    \def\tst{#1}%
    \def\labelmarkup{}
    \def\refmarkup{}%
    \ifx\tst\empty\def\@currentlabel{\refmarkup{#2}}{\label{#3}}%
    \else\def\@currentlabel{\refmarkup{#1}}{\label{#3}}\fi%
    \@esphack%
    \labelmarkup{#2}
}
\makeatother

\title{\Large\textbf{Evaluation of $\nox$ emissions and ozone production due 
to vehicular traffic via second-order models}}
\author{\normalsize{Caterina Balzotti}\thanks{Dipartimento di Scienze di Base e Applicate per l'Ingegneria, Sapienza Universit\`a di Roma, Rome, Italy (\href{mailto:caterina.balzotti@sbai.uniroma1.it}{caterina.balzotti@sbai.uniroma1.it})
}
\and {\normalsize{Maya Briani}\thanks{Istituto per le Applicazioni del Calcolo ``M.\ Picone'', 
Consiglio Nazionale delle Ricerche, Rome, Italy (\href{mailto:m.briani@iac.cnr.it}{m.briani@iac.cnr.it}, \href{mailto:b.defilippo@iac.cnr.it}{b.defilippo@iac.cnr.it})}
}
\and {\normalsize{Barbara De Filippo}\footnotemark[2]}
\and {\normalsize{Benedetto Piccoli}\thanks{Department of Mathematical Sciences, Rutgers University, Camden, USA (\href{mailto:piccoli@camden.rutgers.edu}{piccoli@camden.rutgers.edu})
}}
}
\date{\vspace{-0.5cm}}

\begin{document}

\maketitle

\begin{abstract}
The societal impact of traffic is a long-standing and complex problem. We focus on the estimation of ozone production due to vehicular traffic. For this, we couple a system of conservation laws for vehicular traffic,
an emission model, and a system of partial differential equations 
for the main reactions leading to ozone production and diffusion. 
The second-order model for traffic is obtained by choosing a special velocity function for a Collapsed Generalized Aw-Rascle-Zhang model and is tuned on NGSIM data. On the other side, the system of partial differential equations describes the main chemical reactions of $\nox$ gases with a source term provided by a general emission model applied to the output of the traffic model.  We analyze the ozone impact of various traffic scenarios and describe the effect of traffic light timing. The numerical tests show the negative effect of vehicles restarts on $\nox$ emissions, suggesting to increase the length of the green phase of traffic lights to reduce them. 
\end{abstract}

\begin{description}
\item[\textbf{Keywords.}] Road traffic modeling; second-order traffic models; emissions; ozone production. 
\item[\textbf{Mathematics Subject Classification.}]  35L65, 90B20, 62P12.
\end{description}

\section{Introduction}\label{sec:introduction}
The impact of road traffic and its inefficiencies on society is well known and was documented with quantitative estimates for more than a decade \cite{trb2019}. 
Moreover, the societal impact is high also in terms of pollution and environmental effects,  with road traffic accounting for nearly one third of carbon dioxide (CO$_2$) emissions \cite{TRBGreenhouse2011}. 
In general, the impact of air quality on public's health is one of the world's worst toxic pollution problems in this century, the current levels of air pollutants in urban areas are associated with large number of health conditions, including respiratory infections, heart disease \cite{air2019} and cancer.
Air pollutants also contribute to the phenomena of greenhouse effect, ozone depletion, deforestation and the acidification of water and soils \cite{ramanathan2009AE} and they can induce certain diseases as well as damages on materials (plastic, metals, stones), including Cultural Heritage's ones \cite{tidblad2017M,fuente2013CH}.
While CO$_2$ is probably one of the most studied molecules,
the effect on health is also related to other pollutants, such as 
particulate matters and nitrogen dioxide (NO$_2$), see \cite{Zhang2013}.
Here we focus on the production of ozone which stems out of chemical
reactions in the atmosphere of the $\nox$ gases \cite{atkinson2000AE,wang2017STE}.

Much attention has been devoted in traffic literature to
quantities such as flow, capacity and travel time. However, advanced modeling of fuel consumption and emission still faces limitations, especially for tools which can be integrated with the increasing flow of data from probe sensors.
One of the main reasons is the high variability of fuel consumption and emissions, which are influenced by many factors as the vehicle type, make, model, year and others.
Even if the estimation of fuel consumption and emission at the level of single vehicle presents such drawbacks, as shown in \cite{PiccoliKeFrieszYao2012}, it is possible to achieve reliable estimates using second-order macroscopic models paired with probe sensor data. 
Despite the modeling difficulties cited above, there is an interesting line of research carried out, for example, in the works \cite{alvarez2017JCAM,alvarez2018MCRF,bayen2014}. In these three papers the traffic modeling relies on the first order Lighthill-Whitham-Richards model \cite{LighthillWhitham1955, Richards1956}. In \cite{alvarez2017JCAM} a reaction-diffusion model describes the spread of carbon monoxide in the air with a source term associated to traffic dynamics. This analysis has been reformulated in \cite{alvarez2018MCRF} to set up an optimization problem aimed at finding the optimal model parameters to reduce pollution. In \cite{bayen2014} the authors propose a new methodology to estimate in real-time the emission rates of pollutants and describe their diffusion in air.
The approach we propose in this work fits in this line of research, but is based on second order traffic models, that is models based on two equations
instead of one. Our choice is motivated by the fact that most models
for emissions use the car acceleration, and the latter are better approximated
by using second order models. The pollutants analysis refers to $\nox$ gases taking into account the chemical reactions that lead to ozone production, and finally considering the spread of ozone in the atmosphere. 
Hence, our approach consists of four components is as follows:
1) Traffic dynamics is given by a second order fluid-dynamic model;
2) Pollutants' production are estimated using well-established emissions models based on the traffic quantities given by 1);
3) Then 2) is used to define a source term in  a system of partial differential equations (briefly PDEs) of reaction type,
representing the complex chemical reactions of $\nox$ gases to produce ozone;
4) Finally diffusion in the atmosphere is obtained coupling
the reaction PDEs of 3) with diffusion ones.
We point out that the chemical reactions from $\nox$ gases to ozone
in the atmosphere are still subject to intensive research. However,
our approach is general and can be used for different chemical reaction models
and other pollutants.
The models 1) and 2) are tuned using the NGSIM data set
\cite{TrafficNGSIM}, while for 3) and 4) we use parameters
from literature.


Let us start by discussing the first step related to the evaluation of traffic quantities.
First notice that most emission models use both the speed $v$ and acceleration $a$ of vehicles \cite{Barth2000}.
Thus a macroscopic model to be paired with an emission estimator must be of second-order, i.e.\ consists of an equation for conservation of mass and one for balance of momentum. In particular, the density-flow relation, also known as fundamental diagram, is typically multi-valued and allows a better fit of traffic data.
General approaches have been proposed for second-order models \cite{GaravelloHanPiccoli2016,LebacqueMammarHajSalem2007}, extending the well-know Aw-Rascle-Zhang model \cite{AwRascle2000,Zhang2002}. 
The recent paper \cite{FanSunPiccoliSeiboldWork2017} proposed to use a generalized second-order model with collapsed fundamental diagram in the free phase, thus allowing phase transitions with a simpler description and fitting well with probe and fixed sensor data. This modeling framework is called Collapsed Generalized Aw-Rascle-Zhang Models (briefly CGARZ) and we specify a model in this class by interpolating the Newell-Daganzo or triangular fundamental diagram with the Greenshield quadratic one.

The second step relies on emission models. Among the different models available in literature we have chosen to use the one in \cite{panis2006elsevier} based on a combination of velocity, acceleration and their powers, with parameters specifically tuned for $\nox$ emissions of a petrol car.
Then we pass to the third step which consists in modeling the chemical reactions at the base of ozone production in the atmosphere caused from $\nox$ emissions due to vehicular traffic. Traffic is estimated to cause around one half of nitrogen oxide production, which in turn is one of the main precursor of ozone. The photodissociation
of NO$_2$ is then responsible for the production of the highly reactive $\ouno$ atom and, finally, of ozone. The model capturing these reactions is comprised of a system of five differential equations. The production of NO$_2$ is tuned to 15\% of the overall $\nox$ production as suggested by the recent work \cite{carslaw2011AE}.

To first analyze the emissions and main reactions at street level, we pair the CGARZ model with a system of ODEs distributed along a one-dimensional parametrization of a road. The CGARZ system is responsible for the source term of the ODEs, representing $\nox$ emissions. The coupled system is then simulated using a Godunov-type scheme \cite{FanSunPiccoliSeiboldWork2017} for the CGARZ paired with an ODE-solver for stiff problems for the reactions differential system. To complete the analysis we also consider the diffusion in air of pollutants. Some example of reaction-diffusion models have been proposed in \cite{alvarez2017JCAM,alvarez2018MCRF,bayen2014,stockie2011SIAM}. To analyze the spread of pollutants in the air, we propose two different approaches to integrate the traffic contribution in the PDEs model: in the first one the pollutant emissions are integrated as boundary condition at the bottom of a two dimensional domain where the diffusion evolves vertically; in the second one, the diffusion evolves horizontally around the road and the emissions of traffic are given as a source term in the equations.

The last but not least part of the paper is devoted to the application of the proposed procedure to various traffic scenarios.
The first numerical test is used to validate the emission model. Indeed, first the second-order traffic model is tuned and tested on NGSIM data \cite{TrafficNGSIM}. 
Then, as in \cite{PiccoliKeFrieszYao2012}, we compare emission predictions using the CGARZ model and a macroscopic emission formula
with \emph{ground-truth} emissions using the whole NGSIM dataset
and a microscopic emission formula. 
The resulting predictions need a correction factor, which is determined alternating the NGSIM data blocks (each of 15 minutes)
as training and verification data. The overall relative error ranges between 5\%
and 23\% with an average value of 14\%.
Notice that the relative error would be on the high end if the ultimate goal of the investigation would be the exact estimates of the emissions. 
Finally we analyze the procedure which leads to the production of ozone.
We first run a simple test: the simulation of an interaction between a shock wave with a rarefaction. The shock represent a backward moving queue while the rarefaction an acceleration wave. The shock has minimal effect on the $\nox$ emissions while
the acceleration wave is the most responsible for the highest values. We then consider a road with a traffic light and green-red cycles. The emissions are compared for different length of the cycle and different proportions of the red-green times. The length of the cycle strongly affects $\nox$ emissions: moving from 2.5 minutes to 7.5 minutes produces an increase of around 10\% of emissions, see Figure \ref{fig:semafori_rfisso}. On the other side, the variation of the red time versus green one does not affect significantly $\nox$ emissions, except for an initial ramp up phase when starting from empty road, see Figure \ref{fig:tc_fisso}. These findings are in line with what observed in the first test, but quite different from the common intuition.
We then focus specifically on ozone production. We use ODEs to simulate the ozone concentration at street level and PDEs for its diffusion in the air. Coherently with the test on $\nox$, the level of ozone concentration are highly influenced by the presence of traffic light.

One of the main conclusions of this work is that the duration of traffic cycles affects $\nox$ emissions and ozone production more than the ratio between green and red phase. Therefore, in order to reduce traffic emissions, a possible solution that emerges from this study is the reduction of vehicle restarts by increasing the green phase duration of traffic lights. 
Furthermore, the ozone production is highly influenced by traffic lights. Indeed, the vertical diffusion in presence of traffic lights during 4 hours of simulation shows a 18\% increase of ozone at 1 meter from the ground compared to the case of no traffic light. Analogously, the horizontal diffusion, which considers also the wind, shows a 11\% increase of ozone at 50 meters from the road in presence of traffic lights during 30 minutes of simulation.

\smallskip
The paper is organized as follows. In Section \ref{sec:CGARZ} we describe the CGARZ model and in Section \ref{sec:Emissions} the emission model. In Section \ref{Air} we introduce a simplified set of chemical reactions which lead to ozone production and in Section \ref{sec:diff} we deal with the diffusion of the chemical species in air. In Section \ref{sec:sistemone} we merge the traffic model with the system of ODEs associated to the chemical reactions and of PDEs for pollutants diffusion. Finally, from Section \ref{sec:test} the proposed procedure is used in several numerical tests to estimate the production of ozone associated to different traffic scenarios.

\section{Traffic model}
\label{sec:CGARZ}
This section is devoted to the first step of our tool: the traffic model. Vehicles dynamics are described by means of a macroscopic second order traffic model, providing the quantities we are interested in, i.e.\ density, speed and acceleration of vehicles.
Specifically,  we introduce the \textit{Collapsed Generalized Aw-Rascle-Zhang} (hereafter CGARZ) model \cite{FanSunPiccoliSeiboldWork2017}, to describe the evolution of traffic flow, proposing new flux and velocity functions.

\subsection{CGARZ Model}
\label{sec:cgarz}
The CGARZ model is one of the \emph{Generic Second Order Models} (GSOM) \cite{LebacqueMammarHajSalem2007}, a family of macroscopic models which satisfy 
\begin{align}
\begin{split}
&\begin{cases}
\rho_t+(\rho v)_x = 0\\
w_t+vw_x = 0\\
\end{cases}\\
&\disp \text{with } v=V(\rho,w),
\end{split}
\label{eq:CGARZ1}
\end{align}
for a specific velocity function $V$. The variables $\rho(x,t)$, $v(x,t)$ and $w(x,t)$ are respectively the traffic density, the velocity and a property of vehicles which is advected by traffic flow. 
The problem can be written in conservative form as:
\begin{align}
\begin{split}
&\begin{cases}
\rho_t+(\rho v)_x=0\\
y_t+(y v)_x=0\\
\end{cases}\\
&\disp \text{with } v = V(\rho,y/\rho),
\end{split}
\label{eq:CGARZ2}
\end{align}
where $y = \rho w$ is the conserved \emph{total property}. 
The variable $w$ correlates different behaviors of drivers to the flow-density curves. Thus, the GSOM posses a family of fundamental diagrams $Q(\rho,w)=\rho V(\rho,w)$, parametrized by $w$. The peculiarity of the CGARZ model is that $w$ does not influence the traffic behavior in the low density regime. This means that vehicles may have different properties, but the velocity and flow in free-flow is not affected by $w$. Thus, CGARZ possesses a single-valued fundamental diagram in free-flow, and a multi-valued function in congestion. The flux function has then the following form
\begin{equation}\label{eq:Qcgarz}
Q(\rho,w) = \begin{cases}
Q_f(\rho) &\quad\text{if $0\leq\rho\leq\rho_f$}\\
Q_c(\rho,w) &\quad\text{if $\rho_f\leq\rho\leq\rhom$},
\end{cases}
\end{equation}
where $\rho_f$ is the \textit{free-flow threshold density} independent on $w$, and $\rhom$ is the maximum density. 
Following \cite{FanSunPiccoliSeiboldWork2017}, the flux function \eqref{eq:Qcgarz} has to satisfy the following properties:
\begin{enumerate}[label=Q\arabic*.,ref=\textup{Q\arabic*}]
\item  \label{q1} $Q(\rho,w)\in \cuno$ for each $w$.
\item  \label{q2} Flux curves have a common $\rhom$ independent of $w$, $Q(\rhom,w) = 0,\,\forall\, w$.
\item  \label{q3} The flux is strictly concave with respect $\rho$, $\frac{\de^2 Q(\rho,w)}{\de\rho^2}<0$ for $\rho\in[0,\rhom)$.
\item  \label{q4} $\frac{\de Q(\rho,w)}{\de w}>0$ if $\rho_f<\rho<\rhom$.
\end{enumerate}
The flux function \eqref{eq:Qcgarz} defines a velocity function $V(\rho,w) = Q(\rho,w)/\rho$. Thus, as a consequence of the properties of $Q$, the velocity function $V$ is in $\cuno$ and is strictly decreasing with respect to $\rho$. Moreover, $V$ satisfies:
\begin{enumerate}[label=V\arabic*.]
\item Vehicles never go backwards, $V(\rho,w) \geq 0$.
\item $\rhom$ is the only density such that $V(\rhom,w)=0$.
\item In the free-flow regime, the traffic velocity is independent of $w$, $\frac{\de V(\rho,w)}{\de w} = 0$ if $0\leq\rho\leq\rhof$.
\item[V4.] In the congestion regime, the traffic velocity is increasing with respect to $w$, $\frac{\de V(\rho,w)}{\de w} > 0$ if $\rhof\leq\rho\leq\rhom$.
\end{enumerate}
In the next section we propose a new family of fundamental diagrams that satisfy the properties listed above.

\subsubsection{Flux and velocity functions}
Here we make a choice for the flux function of the CGARZ family,
thus determining a unique model to be used.
Differently from \cite{FanSunPiccoliSeiboldWork2017}, we choose
the flux function to be an interpolation between a triangular fundamental
diagram, also known as Newell-Daganzo, and a Greenshield fundamental diagram.
The reason for this choice is that those two diagrams are the most known and used in traffic modeling and they present two somehow opposite behavior, with the triangular one presenting a unique characteristic speed in congested regime,
thus allowing contact discontinuities, while the Greenshield one being genuinely
nonlinear in congested regime thus exhibiting rarefaction waves.

The model parameters to be calibrated from data are the following: the maximum speed $\vmax$, the threshold density $\rhof$ from the free-flow to the congested phase, the density $\rho_c$ in which the flux function reaches his maximum value, and a lower and upper bound for $w$, denoted by $w_L$ and $w_R$ respectively. Moreover, we set the maximal density $\rhom$ as a property of the road.\\
As in \cite{FanSunPiccoliSeiboldWork2017}, we assume the Greenshields model in the free-flow regime, i.e.\
\begin{equation*}\label{eq:Qf}
Q_f(\rho) = \rho \vmax\left(1-\frac{\rho}{\rhom}\right),
\end{equation*}
and as a novelty we define the flux function $Q_c(\rho,w)$ in the congested phase, as a convex combination of a lower-bound function $f(\rho)$ and an upper-bound function $g(\rho)$. In particular, we set
\begin{equation}
\label{eq:f}
f(\rho)= \rhof\vmax\left(1-\frac{\rho}{\rhom}\right)
\end{equation}  
as the straight-line which connects $\left(\rho_{f},Q_{f}(\rho_{f})\right)$ with $(\rhom,0)$, and 
\begin{equation}
\label{eq:g}
g(\rho)=\rho \vmax\left(1-\frac{\rho}{\rhom}\right)
\end{equation}
which corresponds to the free-flow phase flux function. Defining 
\begin{equation}
\lambda(w)=\frac{w-\wl}{\wr-\wl},
\label{eq:lambda}
\end{equation} 
then our flux function $Q_{c}(\rho,w)$ is
\begin{equation*}\label{eq:Qc}
Q_{c}(\rho,w) = (1-\lambda(w))f(\rho)+\lambda(w) g(\rho),
\end{equation*}
with $f$ and $g$ given in \eqref{eq:f} and \eqref{eq:g} respectively.
The resulting flux function is
\begin{equation}\label{eq:QcgarzFinale}
Q(\rho,w) = \begin{dcases}
\disp\rho \vmax\left(1-\frac{\rho}{\rhom}\right) &\quad\text{if $0\leq\rho\leq\rho_f$}\smallskip\\
(1-\lambda(w))f(\rho)+\lambda(w) g(\rho) &\quad\text{if $\rho_f\leq\rho\leq\rhom$}.
\end{dcases}
\end{equation}

\begin{proposition}
The flux function \eqref{eq:QcgarzFinale} verifies the properties \ref{q2}-\ref{q4} and the property \ref{q1} for all $\rho\neq\rhof$.
\end{proposition}
\begin{proof}
The function $Q$ is $\cuno$ in $[0,\rhom]\backslash\{\rhof\}$ by construction: the free-flow part $Q_{f}$ is $\cuno$ for all $\rho$, and the congested one is a convex combination of $\cuno$ functions.
Condition \ref{q2} follows directly from the definition of $f$ and $g$ which satisfy $f(\rhom)=g(\rhom)=0$. Condition \ref{q3} is easily verified by the strictly negativity of the second derivative of function in \eqref{eq:QcgarzFinale}.
Finally, condition \ref{q4} follows from the relation
\begin{equation*}
\frac{\partial Q(\rho,w)}{\partial w} = \lambda^{\prime}(w)(g(\rho)-f(\rho))
\end{equation*}
which is strictly positive since $g(\rho)>f(\rho)$ by construction.
\end{proof}
\begin{remark}
To verify condition \ref{q1} for all $\rho\in[0,\rhom]$, it is sufficient to choose a different function $f$ that joins with regularity to free-flow regime. 
\end{remark}

Once the flux function is defined, the velocity function is obtained as 
\begin{equation}
V(\rho,w) = \frac{Q(\rho,w)}{\rho}.
\label{eq:velocityV}
\end{equation}

\subsubsection{Acceleration function}\label{sec:acc}
In time-continuous second-order models, the acceleration equation is a second
partial differential equation of the general form 
\begin{equation*}
\frac{Dv(x,t)}{Dt} = \left(v_t(x,t) + v(x,t) v_x(x,t)\right) = a(\rho(x,t),v(x,t)),
\end{equation*}
where $\frac{D\cdot}{Dt}$ is the total derivative and $v$ is the speed function. 
This equation implies that the rate of change of the local speed $\frac{Dv(x,t)}{Dt} = \left(v_t + v v_x\right)$ in Lagrangian coordinates is equal to an \textit{acceleration function} $a(x,t)=a(\rho(x,t),v(x,t))$.

In CGARZ model we derive the function acceleration by computing the total derivative of $V(\rho,w)$, i.e.\
\begin{equation*}
	a(x,t) = \frac{Dv(x,t)}{Dt} = v_t(x,t) + v(x,t) v_x(x,t),
\end{equation*}
where
$v(x,t)=V(\rho(x,t),w(x,t)),\, v_t = V_\rho \rho_t + V_w w_t, \, v_x = V_\rho\rho_x+V_w w_x.$
Then,
\begin{equation*}
a(x,t) = \left(\rho_t + v \rho_x\right)V_\rho+  \left(w_t + v w_x\right)V_w,
\end{equation*}
and by applying the homogeneous equation in \eqref{eq:CGARZ1} for $w$ we get
\begin{equation}\label{eq:accAnalitica}
a(x,t) = V_\rho\left(\rho_t + v \rho_x\right) = -V_\rho \rho v_x.
\end{equation}

\section{Estimating emissions by traffic quantities}\label{sec:Emissions}
In this section we analyze the second step of our tool: the emission model. Specifically, we describe the emission model proposed in \cite{panis2006elsevier} appropriate for several air pollutants.
Emitted by different sources, primary and secondary air pollutants mainly include: sulphur oxides, nitrogen oxides ($\nox$), volatile organic compounds (VOC), particulates, free radicals, toxic metals, etc.  \cite{wayne1991CP,seinfeld2016JWS}.
In areas with heavy street traffic and high amounts of UV radiation, 
ozone ($\otre$), $\nox$ and hydrocarbons are  of particular interest.

The existence of high concentration of ozone in the urban atmosphere suggests to have an effective control of some other pollutants such as carbon monoxide and sulphur dioxide (SO$_{2}$), ozone is a secondary pollutant formed in the ambient air through a complex set of sunlight initiated reactions of its precursor, primary emission of VOC, catalyzed by hydrogen oxide radicals, and of $\nox$ \cite{jacob2000AE,song2011APR}.
For the complexity of the phenomena involved, in this paper we focus on  emission models for only $\nox$.

\subsection{Emission Model}\label{sec:emissionModel}
We use the microscopic emission model proposed in \cite{panis2006elsevier}. This model gives the instantaneous emission rate of four pollutant types: carbon dioxide, nitrogen oxides, volatile organic compounds and particulate matter. The emission rate $E_{i}$ of vehicle $i$ at time $t$ is computed using vehicle's instantaneous speed $v_i(t)$ and acceleration $a_i(t)$
\begin{equation}\label{eq:EmissionRate}
E_i(t) = \max \{ E_0 ,  f_1 + f_2 v_i(t) + f_3 v_i(t)^2 + f_4 a_i(t) + f_5 a_i(t)^2 + f_6 v_i(t) a_i(t)\},
\end{equation}
where $E_0$ is a lower-bound of emission and $f_1$ to $f_6$ are emission constants. The parameters are experimentally calibrated using non-linear multiple regression techniques as explained in \cite{panis2006elsevier}. Both the emission lower-bound and coefficients differ according to the type of pollutant and of vehicle (i.e.\ petrol car, diesel car, truck, etc.).  We are particularly interested in the $\nox$ emission rate, whose coefficients depend on whether the vehicle is in acceleration (defined as $a_i(t)\geq \myunit{-0.5}{\  \mymeter\per\mysecond^{2}}$) or deceleration (with $a_i(t)< \myunit{-0.5}  {\mymeter\per\mysecond^{2}}$) mode, where $\mymeter$ denotes meter and $\mysecond$ second. In Table \ref{table:EmissionRateParam}  we report the $\nox$ emission coefficients for a petrol car, for which $E_{0}=0$. See \cite[Table 2]{panis2006elsevier} for the coefficients related to the other pollutants and vehicles type.
\begin{table}[h]
\renewcommand{\arraystretch}{1.5}
\small
\centering
\begin{tabular}{|c|c|c|c|c|c|c|}\hline
Vehicle mode & $f_1$ $\left[\disp\frac{\mygram}{\mysecond}\right]$ & $f_2$ $\left[\disp\frac{\mygram}{\mymeter}\right]$ & $f_3$ $\left[\disp\frac{\mygram\, \mysecond}{\mymeter^{2}}\right]$ & $f_4$ $\left[\disp\frac{\mygram\, \mysecond}{\mymeter}\right]$& $f_5$ $\left[\disp\frac{\mygram\, \mysecond^{3}}{\mymeter^{2}}\right]$& $f_6$ $\left[\disp\frac{\mygram \, \mysecond^{2}}{\mymeter^{2}}\right]$ \\
\hline
If $a_i (t) \geq -0.5\,\mymeter\per\mysecond^2$ &  6.19e-04  & 8e-05  & -4.03e-06  & -4.13e-04  & 3.80e-04  & 1.77e-04\\
If $a_i (t) <-0.5\,\mymeter\per\mysecond^2$ &  2.17e-04  & 0  & 0  & 0  & 0  & 0\\\hline
\end{tabular}
\caption{$\nox$ parameters in emission rate formula \eqref{eq:EmissionRate} for a petrol car, where $\mygram$ denotes gram, $\mymeter$ meter and $\mysecond$ second.}
\label{table:EmissionRateParam}
\end{table}
\begin{remark}
In this work we assume to have a unique typology of vehicles, i.e.\ petrol cars. The integration with other types of vehicles and comparison of the corresponding emission rates is an interesting subject of study, it would require the use of traffic models for multi-class vehicles and this goes beyond the scopes of this work.
\end{remark}

Assuming to have $N$ vehicles in a stretch of road going all at the same speed $\bar v$, with the same acceleration $\bar a$, the emission rate is given by the $N$ contributes of the vehicles, such that
\begin{equation}
E(t)=\sum_{i=1}^N E_i(t) = N \max \{ E_0 ,  f_1 + f_2 \bar v(t) + f_3 \bar v(t)^2 + f_4 \bar a(t) + f_5 \bar a(t)^2 + f_6 \bar v(t) \bar a(t)\}.
\label{eq:emissioni}
\end{equation}
In particular this equation can be used in conjunction with quantities provided by a numerical solution to a macroscopic model such as the CGARZ one.

\begin{remark}
We make use of a particular emission model. However, the large majority
of microscopic emissions models are based on a combination of
polynomial expression in the velocity and acceleration, see for instance \textup{\cite{Barth2000,Smit2010}} and references therein. Thus our analysis
can be easily adapted to other models.
\end{remark}

\section{Chemical reactions}
\label{Air}
This section is devoted to the third step of our tool: the chemical reactions associated to the pollutants under investigation. In particular, in this work we focus on $\nox$ gases and the reactions
which lead to the $\otre$ formation. 
$\nox$ gases are usually produced from the reaction among nitrogen and oxygen ($\odue$) during combustion of fuels, such as hydrocarbons, in air, especially at high temperatures, such as occurs in car engines \cite{Omid2015}.
They include nitrogen oxide ($\no$) and nitrogen dioxide ($\nodue$); the latter is classified as a secondary pollutant. NO is produced according to the following reaction with $\odue$ and nitrogen (N$_{2}$) \cite{manahan2017CRCP}, 
\begin{equation*}
	\mathrm{N}_2 + \odue \longrightarrow 2\no,
\end{equation*}
where the rate of the chemical reaction can be increased by raising the temperature.
In the combustion mechanism, $\no$ can react with $\odue$ thus forming $\nodue$, 
\begin{equation*}
2\no + \odue \longrightarrow 2\nodue .
\end{equation*}
$\nodue$  is a very reactive compound that can be photo-dissociated into atomic oxygen ($\ouno$), this mechanism is considered one of key steps in the formation of tropospheric ozone \cite{atkinson1984CR}.
Nitrogen oxides and volatile organic compounds are considered ozone precursors, where  traffic is the main source (more than $50\%$ of anthropogenic source). 
The photolysis of $\nodue$ is speeded up in warmer conditions and with more UV-light. In the troposphere with strong solar irradiation, $\nodue$ is a relevant precursor substance for the ozone in photochemical smog and it is due to the following reactions: 
\begin{align}
\label{eq:reazione1}
&\nodue+ h\nu \overset{k_{1}}\longrightarrow \ouno+\no \\
\label{eq:reazione2}
&\ouno + \odue + \mathrm{M} \overset{k_{2}}\longrightarrow \otre+\mathrm{M},
\end{align}
where $h$ is Planck's constant, 
$\nu$ its frequency and $k_1$, $k_2$ are the reaction rate constants.
M is a chemical species, such as $\odue$ or N$_2$, that adsorbs the excess of energy generated in reaction \eqref{eq:reazione2} \cite{manahan2017CRCP}. 
Moreover, in presence of $\no$, $\otre$ reacts with it and this reaction destroys the ozone and reproduces the $\nodue$, with kinetic constant $k_3$:
\begin{equation}
\otre + \no \overset{k_{3}}\longrightarrow \odue + \nodue.
\label{eq:reazione3}
\end{equation}
This means that the previous reactions do not result net ozone production, because the reactions only recycle $\otre$ and $\nox$. 
Net ozone production occurs when other precursors, such as carbon monoxide, methane, non-methane hydrocarbons or certain other organic compounds (volatile organic compounds) are present in the atmosphere and fuel the general pathways to tropospheric $\otre$ formation. Although it would be interesting to consider the whole ground-level ozone production, here we focus only on the photochemical smog reactions \eqref{eq:reazione1}, \eqref{eq:reazione2} and \eqref{eq:reazione3}.

For vehicle's emissions, the maximum $\nodue$ concentration is recorded at medium engine load and low engine speed. At high speed, the $\nodue$ emissions are reduced to a minimum (in most cases less than 4\%) \cite{rossler2017MTZ}. 
According to a recent study using British data \cite{carslaw2011AE}, the fraction of $\nodue$ in vehicle $\nox$ emissions (all fuels) increased from around 5-7\% in 1996 to 15-16\% in 2009. 
For this reason we will consider in our simulation a $\nodue$ concentration equal to 15\% of $\nox$.

\medskip
Now, we set up the system of ordinary differential equations associated to the chemical reactions \eqref{eq:reazione1}, \eqref{eq:reazione2} and \eqref{eq:reazione3}.
We assume that the reactions take place in a volume 
of dimension $\vol=\dx\dy\dz$, 
during the daily hours and that the chemical specie M in \eqref{eq:reazione2} is $\odue$. 
Moreover, we add the traffic emissions contribution as a source term for the concentration of $\no$ and $\nodue$. Hence, we denote the chemical species concentration by $[\cdot]=[\frac{\text{weight unit}}{\text{volume unit}}]$ and we define the variation of the concentration of $\nox$ in $\vol$, at each time $t$ as  
\begin{equation}\label{eq:nox_source}
S_{\nox}(t) = \disp\frac{E_{\nox}(t)}{\vol},
\end{equation}
where the emission rate $E_{\nox}(t)$ is given by \eqref{eq:emissioni}.

Let us denote by $\CC=(\cc_{1},\cc_{2},\cc_{3},\cc_{4},\cc_{5})$ the vector of the five chemical species concentration, i.e.\ $\cc_1(t) = [\ouno]$, $\cc_2(t) = [\odue]$, $\cc_3(t) = [\otre]$, $\cc_4(t) = [\no]$ and $\cc_5(t) = [\nodue]$. The final system of equations, given by coupling the three reactions \eqref{eq:reazione1}-\eqref{eq:reazione3} and the source term \eqref{eq:nox_source}, becomes 
\begin{equation}\label{eq:sistemone}
\begin{cases}
\disp\frac{d\cc_1}{dt}= k_1\, \cc_5-k_2\, \cc_1\,\cc_2^2\smallskip\\
\disp\frac{d\cc_2}{dt} =  k_3\,\cc_3\,\cc_4- k_2\, \cc_1\,\cc_2^2 \smallskip\\
\disp\frac{d\cc_3}{dt} = k_2\,\cc_1\,\cc_2^2 - k_3\,\cc_3\,\cc_4  \smallskip\\
\disp\frac{d\cc_4}{dt} = k_1\,\cc_5  - k_3\,\cc_3\,\cc_4 + (1-p)\,s(t) \smallskip\\
\disp\frac{d\cc_5}{dt} = k_3\,\cc_3\,\cc_4- k_1\,\cc_5  + p\, s(t),
\end{cases}
\end{equation}
where $p=0.15$ corresponding to 15\% of $\nodue$ derived from the emission rate of $\nox$, $s(t)$ is the source term defined in \eqref{eq:nox_source} and the parameters $k_{1}$, $k_{2}$ and $k_{3}$, shown in Table \ref{tab:parametriK}, are estimated according to \cite{jacobson2005CUP}. System \eqref{eq:sistemone} can be rewritten in vectorial form as
\begin{equation}\label{eq:sistemaVett}
	\frac{d\CC(t)}{dt}=G(\CC(t))+S(t),
\end{equation}
where $G$ represents the chemical reactions and $S$ the source term.

\begin{table}[h!]
\centering
\small
\renewcommand{\arraystretch}{1.2}
\begin{tabular}{|c|c|c|}\hline
$k_{1}$ &  $k_{2}$ & $k_{3}$ \\\hline
$\myunit{0.02}{\mysecond^{-1}}$ & $\myunit{6.09\times 10^{-34}}{\mycenti\mymeter^6}\,\mathrm{ molecule}^{-2}\,\mysecond^{-1}$ & $\myunit{1.81\times10^{-14}}{\mycenti\mymeter^3}\,\mathrm{molecule}^{-1}\,\mysecond^{-1}$ \\\hline
\end{tabular}
\caption{Parameters $k_{1}$, $k_{2}$, and $k_{3}$ of system \eqref{eq:sistemone}, where $\mycenti\mymeter$ denotes centimeter, $\mysecond$ second and $\mathrm{molecule}$ the number of molecules.}
\label{tab:parametriK}
\end{table}

\section{Diffusion of chemical species in air}\label{sec:diff}
In this section we deal with the diffusion of the chemical species in air. We refer to \cite{alvarez2017JCAM,alvarez2018MCRF, bayen2014,stockie2011SIAM} for some examples of study of pollutants diffusion through PDEs. Here we propose two different approaches to integrate the traffic contribution into a reaction-diffusion model to analyze the spread of pollutants in the atmosphere. 

\subsection{Vertical diffusion}
Let us consider a domain $\Omega=[0,L]\times[0,H]$, where $L$ is the length of the road and $H$ is the height from the source of $\nox$ emissions, during a time interval $[0,T]$. We assume that $\CC$ is constant along the third direction, therefore our model is two-dimensional and $\CC$ still represents the concentration of pollutants in unit of weight per unit of volume. The first approach we propose integrates the traffic contribution into the boundary conditions of the following reaction-diffusion problem
\begin{equation}\label{eq:diffVert}
	\begin{cases}
		\disp\frac{\de\CC}{\de t}(x,y,t)-\mu\Delta\CC(x,y,t) = G(\CC(x,y,t)) &\quad\text{in $\Omega\times(0,T]$}\smallskip\\
		\CC(x,y,0) = \CC_{0}(x,y) &\quad\text{in $\Omega$}		
	\end{cases}
\end{equation}
where $\mu$ is the diffusion coefficient, $G(\CC)$ introduced in \eqref{eq:sistemaVett} represents the chemical reactions and $\CC_{0}$ is the initial datum. We assume that $\mu$ is the same for all the five chemical species under analysis, where $\mu$ is a typical value ($10^{-8}\,\km^{2}\per\myhour$) for aerosols \cite{Sportisse2010}, which also include our pollutants.
We fix homogeneous Neumann boundary condition on the left, upper and right boundary of $\Omega$ for all the chemical species, i.e.\ for $t\in[0,T]$ we have
\begin{align*}
	\disp\frac{\de\CC}{\de y}(x,H,t) = 0 \quad\text{for $x\in[0,L]$,}\qquad\frac{\de\CC}{\de x}(0,y,t) = \disp\frac{\de\CC}{\de x}(L,y,t) = 0 \quad\text{for $y\in[0,H]$.}
\end{align*}
The lower boundary of $\Omega$ has homogeneous Neumann boundary condition for the first three chemical species and Dirichlet condition for the other two, i.e.\ for $x\in[0,L]$ and $t\in[0,T]$ we have
\begin{align}
	\nonumber\disp\frac{\de\cc_{1}}{\de y}(x,0,t) &= \frac{\de\cc_{2}}{\de y}(x,0,t) =\frac{\de\cc_{3}}{\de y}(x,0,t) = 0 \\
	\cc_{4}(x,0,t) &= (1-p)e(x,t)\label{eq:BC1}\\
	\cc_{5}(x,0,t) & = p\,e(x,t),\label{eq:BC2}
\end{align}
where $e(x,t)$ is given by the emission rate $E_{\nox}(x,t)$ per unit of time over unit of volume.

\subsection{Horizontal diffusion}
We now consider a horizontal domain $\Omega=[0,\lx]\times[0,\ly]$, where $\lx$ is the length of the road and $\ly$ is the length of the area transversal to the road where the pollutants spread during a time interval $[0,T]$. Again, we assume that the concentration of pollutants $\CC$ is constant along the third direction, reducing to the following two-dimensional reaction-diffusion problem
\begin{equation}\label{eq:diffOriz}
	\begin{cases}
		\disp\frac{\de\CC}{\de t}(x,y,t)+\vento\nabla\CC(x,y,t)-\mu\Delta\CC(x,y,t) = G(\CC(x,y,t)) +S(x,y,t) &\quad\text{in $\Omega\times(0,T]$}\smallskip\\
		\CC(x,y,0) = \CC_{0}(x,y) &\quad\text{in $\Omega$}		
	\end{cases}
\end{equation}
where $\vento=(\ventox,\ventoy)$ is the wind, $\mu$ is the diffusion coefficient \cate{(the same for all the chemical species)}, $G(\CC)$ represents the chemical reactions, $S(x,y,t)$ is the source term of $\nox$ emission rates and $\CC_{0}$ is the initial datum. Unlike the previous case, where the contribution of the road is given through Dirichlet boundary conditions, in this case we add a source term of $\nox$ emissions in the reaction-diffusion equation in correspondence of the road, placed in the middle of the $y$-axis. The source term $S(x,y,t)$ is given by a vector always null except for $y=\ly/2$, where we have 
\begin{equation*}
	S(x,\ly/2,t) = (0,0,0,(1-p)s(x,t),p\,s(x,t))
\end{equation*}
with $s(x,t)$ given by \eqref{eq:nox_source}.
The boundary of $\Omega$ is treated through homogeneous Neumann boundary conditions for the all the five chemical species.

\section{From traffic quantities to the production and diffusion of the ozone}\label{sec:sistemone}
In this section we merge the traffic model with air pollutants dynamics, summarizing the four steps of the proposed tool and introducing the numerical methodology. The procedure is the following:
\begin{enumerate}
\item Estimate the traffic quantities, i.e.\ the density and the speed of vehicles with the CGARZ model \eqref{eq:CGARZ2} and the analytical acceleration with \eqref{eq:accAnalitica}.
\item Estimate the emission rate with \eqref{eq:emissioni} and the corresponding source term in the chemical reactions per unit of volume given by \eqref{eq:nox_source}.
\item Solve system \eqref{eq:sistemone} to estimate the concentration of the chemical species at street level.
\item Solve system \eqref{eq:diffVert} or \eqref{eq:diffOriz} to estimate the diffusion of the chemical species concentration in air.
\end{enumerate}
We now describe the numerical implementation. Let us consider the road $[0,L]$ during the time interval $[0,T]$ discretized via a grid of $N_{x}\times N_{t}$ cells of length $\dx\times\dt$.
For each cell centered at $x_{i}$ and time $t^{n}$ of the numerical grid our aim is then to estimate the traffic quantities $\rho^{n}_{i}$, $v^{n}_{i}$, $a^{n}_{i}$, the emission rates $E^{n}_{i}$ and the source term $s^{n}_{i}$. The resolution of system \eqref{eq:sistemone} gives us the concentration of the five chemical species $\CC^{n}_{i}$ produced on the road. Systems \eqref{eq:diffVert} and \eqref{eq:diffOriz} involve a two-dimensional domain $\Omega$  discretized via a grid of steps $\dx\times\dy$ and describe the diffusion of the concentration $\CC^{n}_{ij}$ in the air.

\subsection{Numerical method for the CGARZ model}\label{sec:traffNum}
The CGARZ model \eqref{eq:CGARZ2} is numerically solved using the 2CTM scheme  described in \cite{FanSunPiccoliSeiboldWork2017}, which is a Godunov type scheme and can be used for any GSOM. Here we describe the general scheme and how to apply it to the CGARZ model. 

Let us consider the numerical grid introduced above and set $v^{n}_{i}=V(\rho^{n}_{i},w^{n}_{i})$. The 2CTM scheme is described by the system
\begin{align*}
\rho^{n+1}_{i}&=\rho^{n}_{i}-\frac{\dt}{\dx}(F^{\rho,n}_{i+1/2}-F^{\rho,n}_{i-1/2})\\
y^{n+1}_{i} &=y^{n}_{i}-\frac{\dt}{\dx}(F^{y,n}_{i+1/2}-F^{y,n}_{i-1/2}),
\end{align*}
where $F^{\rho,n}_{i\pm1/2}$ and $F^{y,n}_{i\pm1/2}$ are the numerical fluxes. 
In order to define $F^{\rho,n}_{i-1/2}$ and $F^{y,n}_{i-1/2}$, consider the two constant left and right states $(\rho^{-},w^{-})=(\rho^n_{i-1},w^n_{i-1})$ and $(\rho^{+},w^{+})=(\rho^n_{i},w^n_{i})$ respectively, and compute the solution of the Riemann problem between the two consecutive cells centered in $x_{i-1}$ and $x_{i}$,
\begin{equation*}
	\begin{array}{lcr}
	  \left\{\begin{array}{c}
	   \rho_t+(\rho v)_x=0\\
		y_t+(y v)_x=0\\
	  \end{array}\right.
	  & \mbox{ with } &
	  (\rho_{0},y_{0}) =\begin{array}{cc}
	   (\rho^{-},\rho^{-}w^{-}) &\quad\text{if $x<x_{i-1/2}$}\\
			(\rho^{+},\rho^{+}w^{+}) &\quad\text{if $x\geq x_{i-1/2}$}.
	  \end{array}
	\end{array}
\end{equation*}
The solution of the Riemann problem is defined by an intermediate state $(\rho^{*},w^{*})$ separated from the left and right state by 
a 1-shock or rarefaction wave and a 2-contact discontinuity respectively. The Riemann invariants \cite{FanHertySeibold2013} $w = const.$ and $V(\rho,w)=const.$, imply that $w^{*}=w^{-}$ and $V(\rho^{*},w^{*})=\min\{v^{+},V(0,w^{-})\}$ with $v^{+}=V(\rho^{+},w^{+})$. Note that the minimum between the two velocities is required since vehicles from the left try to adapt their velocity to $v^{+}$, but if $v^{+}>V(0,w^{-})$ they cannot exceed their maximum  speed $V(0,w^{-})$.
Let us introduce now the supply and demand functions $S$ and $D$ defined as
\begin{equation*}
	S(\rho,w)=\begin{cases}
		Q^{\mathrm{max}}(w) &\quad\text{if $\rho\leq\rho^{cr}(w)$}\\
		Q(\rho,w) &\quad\text{if $\rho>\rho^{cr}(w)$}\\
	\end{cases}\qquad
	D(\rho,w)=\begin{cases}
		Q(\rho,w) &\quad\text{if $\rho\leq\rho^{cr}(w)$}\\
		Q^{\mathrm{max}}(w) &\quad\text{if $\rho>\rho^{cr}(w)$},\\
	\end{cases}
\end{equation*}
with $\rho^{cr}(w)$ critical density, i.e.\ the value where the flux curve identified by $w$ attends its maximum $Q^{\mathrm{max}}(w)$. The numerical flux is then defined as
\begin{equation}\label{eq:numflux}
	F^{\rho,n}_{i-1/2}=\min\{D(\rho^{n}_{i-1},w^{n}_{i-1}),S(\rho^{n}_{i-1/2},w^{n}_{i-1/2})\}
\end{equation}
where $(\rho^{n}_{i-1/2},w^{n}_{i-1/2})$ is the value of the intermediate state described above. Moreover, since $y=\rho w$ the numerical fluxes $F^{y,n}_{i\pm1/2}$ are such that
\begin{equation*}
	F^{y,n}_{i-1/2}=w^{n}_{i-1/2}F^{\rho,n}_{i-1/2}\qquad\text{and}\qquad F^{y,n}_{i+1/2}=w^{n}_{i}F^{\rho,n}_{i+1/2}.
\end{equation*}
By construction of the flux function for the CGARZ model, the condition $v^{+}>V(0,w^{-})$ never holds, since $V(0,w)=\vmax$ for any $w$. Hence the intermediate state $(\rho^{*},w^{*})$ is such that $w^{*}=w^{-}$ and $V(\rho^{*},w^{*})=v^{+}$. 
In \eqref{eq:numflux} we then get $w^n_{i-1/2}=w^n_{i-1}$ and $\rho^{n}_{i-1/2}$ such that $V(\rho^n_{i-1/2},w^n_{i-1})=V(\rho^n_{i},w^n_{i})$. 

The stability of the scheme is guaranteed by the CFL condition 
\begin{equation}\label{eq:CFL}
	\dt\leq \dx/(2\Lambda)
\end{equation}
with $\Lambda = \max_{j=1,2}|\lambda_{j}|$ and $\lambda_{j}$ eigenvalues of the Jacobian matrix associated to \eqref{eq:CGARZ2}. In our case  $\Lambda$ coincides with the maximum velocity $\vmax$.

\subsubsection{Evaluating the acceleration}
As described in Section \ref{sec:acc}, we can approximate the acceleration directly derived from the theoretical model by \eqref{eq:accAnalitica} as
\begin{equation}\label{eq:accNum}
	a^{n}_{i}=-V_\rho(\rho^{n}_{i},w^{n}_{i}) \rho^{n}_{i} \frac{v^{n}_{i+1}-v^{n}_{i-1}}{2\dx}.
\end{equation}
Here we also describe a discrete formulation for the acceleration recovered by average quantities, as an alternative to \eqref{eq:accNum}.
We follow the approach proposed in \cite{luspay2010IEEE,zegeye2013elsevier} for the particular case of a single road with $n_{\ell}$ lanes. 
To define the average acceleration of a cell, we distinguish between the \emph{temporal acceleration} and the \emph{spatial-temporal acceleration}. 
The temporal acceleration refers to the change of the average speed for the vehicles which remain in the same cell $i$ between time $t^{n}$ and $t^{n+1}$,
\begin{equation}
a^{tmp}_{i}(n) = \frac{ v_{i}^{n+1}- v_{i}^{n}}{\dt}.
\label{eq:accTemp}
\end{equation}
Let $ q_{i}^{n}$ be the flux of vehicles which cross the cell $i$ between time $t^{n}$ and $t^{n+1}$. The total number of vehicles which remain in the cell and therefore which are subjected to the temporal acceleration is
$c^{tmp}_{i}(n) = n_{\ell}\dx\rho_{i}^{n}-\dt q_{i}^{n}$.
The spatial-temporal acceleration refers to the change of the average speed for the vehicles which move from a cell to the following one. It is defined as
\begin{equation}
a^{spt}_{i}(n) = \frac{ v_{i+1}^{n+1}- v_{i}^{n}}{\dt},
\label{eq:accSpat}
\end{equation}
and the total number of vehicles subjected to this acceleration is 
$c^{spt}_{i}(n) = \dt q_{i}^{n}$.
Combining the definitions of temporal \eqref{eq:accTemp} and spatial-temporal \eqref{eq:accSpat} acceleration, we can introduce the average acceleration of vehicles in cell $i$ at time $t^{n}$ as
\begin{equation*}
a_{i}^{n}=\frac{a^{tmp}_{i}(n) c^{tmp}_{i}(n) +a^{spt}_{i}(n) c^{spt}_{i}(n) }{c^{tmp}_{i}(n) +c^{spt}_{i}(n) },
\end{equation*}
which, after some computations, can be rewritten as
\begin{equation}
a_{i}^{n}=\frac{ v_{i}^{n+1}-v_{i}^{n}}{\dt}+ v_{i}^{n}\frac{ v_{i+1}^{n+1}- v_{i}^{n+1}}{\dx}.
\label{eq:acc}
\end{equation}
Hereafter we refer to this formulation as discrete acceleration.

\subsection{Numerical tests}\label{sec:test}
In this section we show some examples illustrating the several steps which lead to the estimate of the production of ozone. First of all we validate the emission model to estimate the $\nox$ emission rates with a numerical test using the NGSIM dataset \cite{TrafficNGSIM}. Then we provide some tests related to the complete procedure, focusing on the production of ozone. In particular, we investigate the impact of traffic lights on pollutants production, looking for strategies to reduce it.

\subsubsection{Validation of the emission model}
\label{sec:emissioni}
In this section we compare the $\nox$ emission rates given by \eqref{eq:EmissionRate} computed using the NGSIM dataset \cite{TrafficNGSIM} with that given by \eqref{eq:emissioni} computed along numerical solutions to the CGARZ model. In other words, the macroscopic CGARZ model is fed by real data only at initial time, then the emission rate is computed along the numerical solution to CGARZ and compared with that resulting from the NGSIM complete dataset, considered as a \emph{ground truth.}

The NGSIM database contains detailed vehicle trajectory data on the interstate I-80 in California, on April 13, 2005. The area under analysis is approximately 500 meters in length and consists of six freeway lanes. Several video cameras recorded vehicles moving through the monitored area, while a specific software has transcribed the vehicle trajectory data from video. The data include the precise location, velocity and acceleration of each vehicle within the study area every 0.1 seconds. The period analyzed in this work refers to three time slots: 4:00 pm - 4:15 pm, 5:00 pm - 5:15 pm and 5:15 pm - 5:30 pm.

First of all we estimate the flow-density and velocity-density relationships from the dataset. We divide the study area into space-time cells $C^n_i = [x_{i}, x_{i+1}]\times [t^{n}, t^{n+1}]$ of length $\myunit{120}{\mymeter}\times \myunit{4}{\mysecond}$. 
The density in $C^n_i$ is equal to the number of vehicles (denoted by veh) which cross the cell during the time interval $[t^{n},t^{n+1}]$. The velocity in $C^n_i$ is the mean of all the velocities measured in the cell, and the flux is the product between density and velocity. The relationships between flow and density and between velocity and density are shown in the top panels of Figure \ref{fig:QV}. In the two graphs we clearly see two ``clouds''  in which data are concentrated (except a small number of outliers accounting for less than 3\% of points). From the analysis of these data we have estimated a possible set of model parameters: $\vmax=\myunit{65}{\mykilo\mymeter\per\myhour}$, $\rho_f=\myunit{110}{\mathrm{veh}\per\mykilo\mymeter}$, $\rhom=\myunit{800}{\mathrm{veh}\per\mykilo\mymeter}$, $\rho_{c}=\rhom/2$, $\wl= 5687$ and  $\wr=13000$, where $\mykilo\mymeter$ denotes kilometer, $\myhour$ hour and $\mathrm{veh}$ the number of vehicles. Specifically, the parameters $\vmax$ and $\rhof$ are chosen such that the area enclosed between the curves $f$ and $g$, in \eqref{eq:f} and \eqref{eq:g} respectively, covers more than the 97\% of data points of the real data clouds; $\rhom$ is a property of the road, defined by
\begin{equation*}
\rhom = \frac{\text{Number of lanes}}{\text{Lenght of vehicles}}=\frac{6}{7.5\times10^{-3}\,\mykilo\mymeter},
\end{equation*}
and we set the two extreme $w_{L}$ and $w_{R}$ as
$w_{L} = g(\rhof) \text{ and } w_{R} = g(\rhoc)$.
%
The family of curves generated by the data set given above are shown in the bottom panels of Figure \ref{fig:QV}.
\begin{figure}[h!]
\centering
\subfloat
{\includegraphics[width=.406\linewidth]{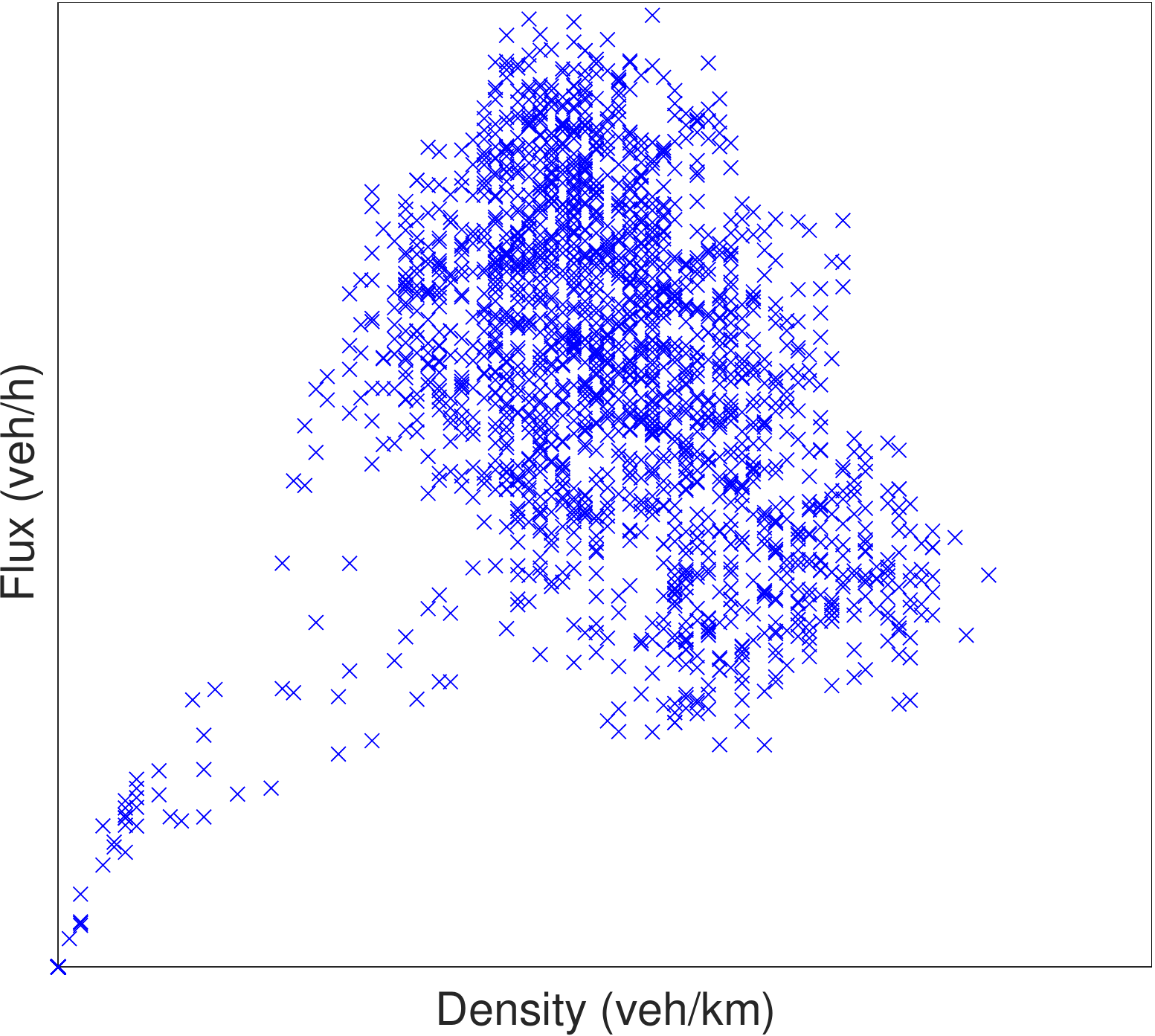}} \quad\,
\subfloat
{\includegraphics[width=.405\linewidth]{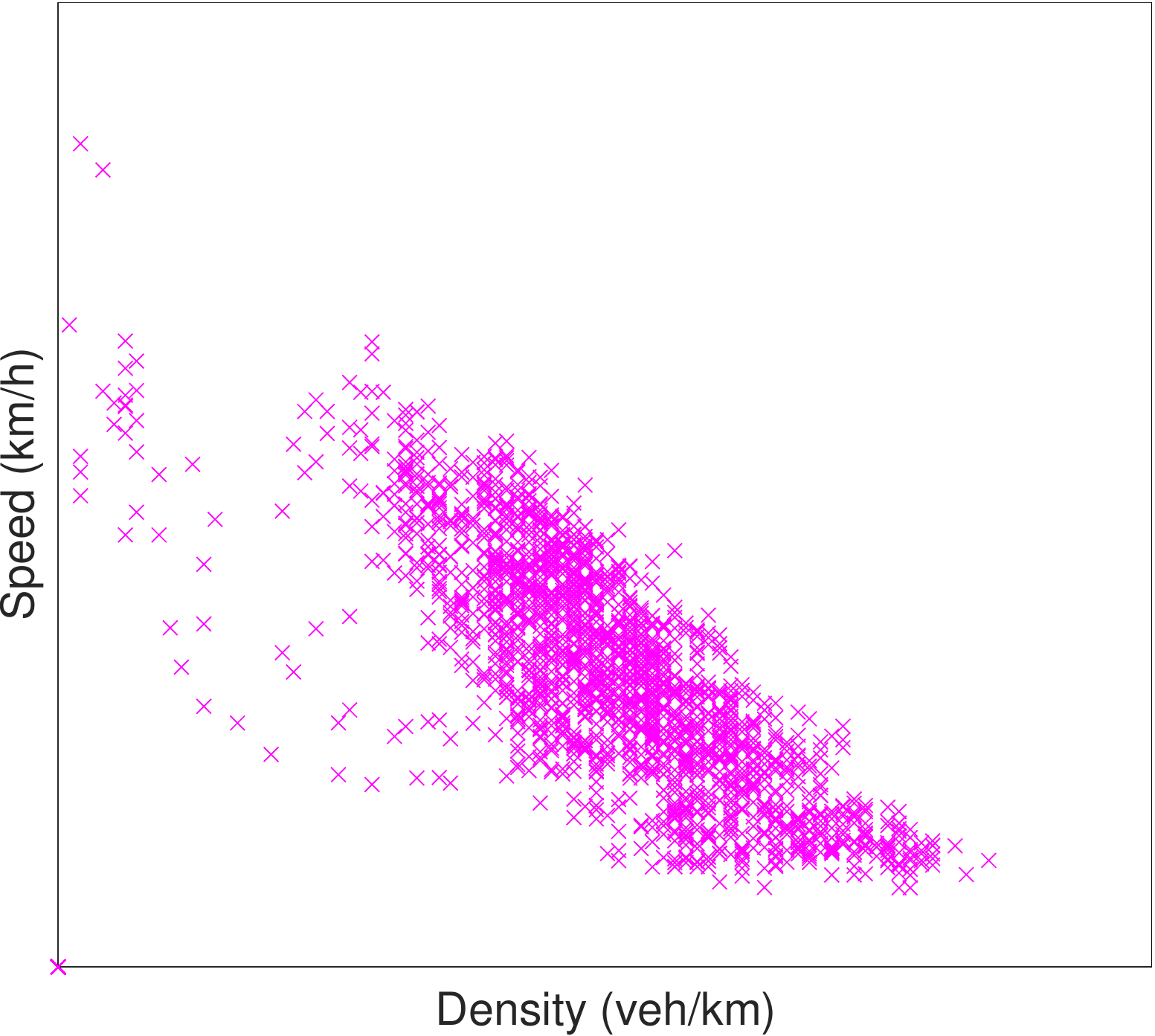}}\\
\subfloat
{\includegraphics[width=.42\linewidth]{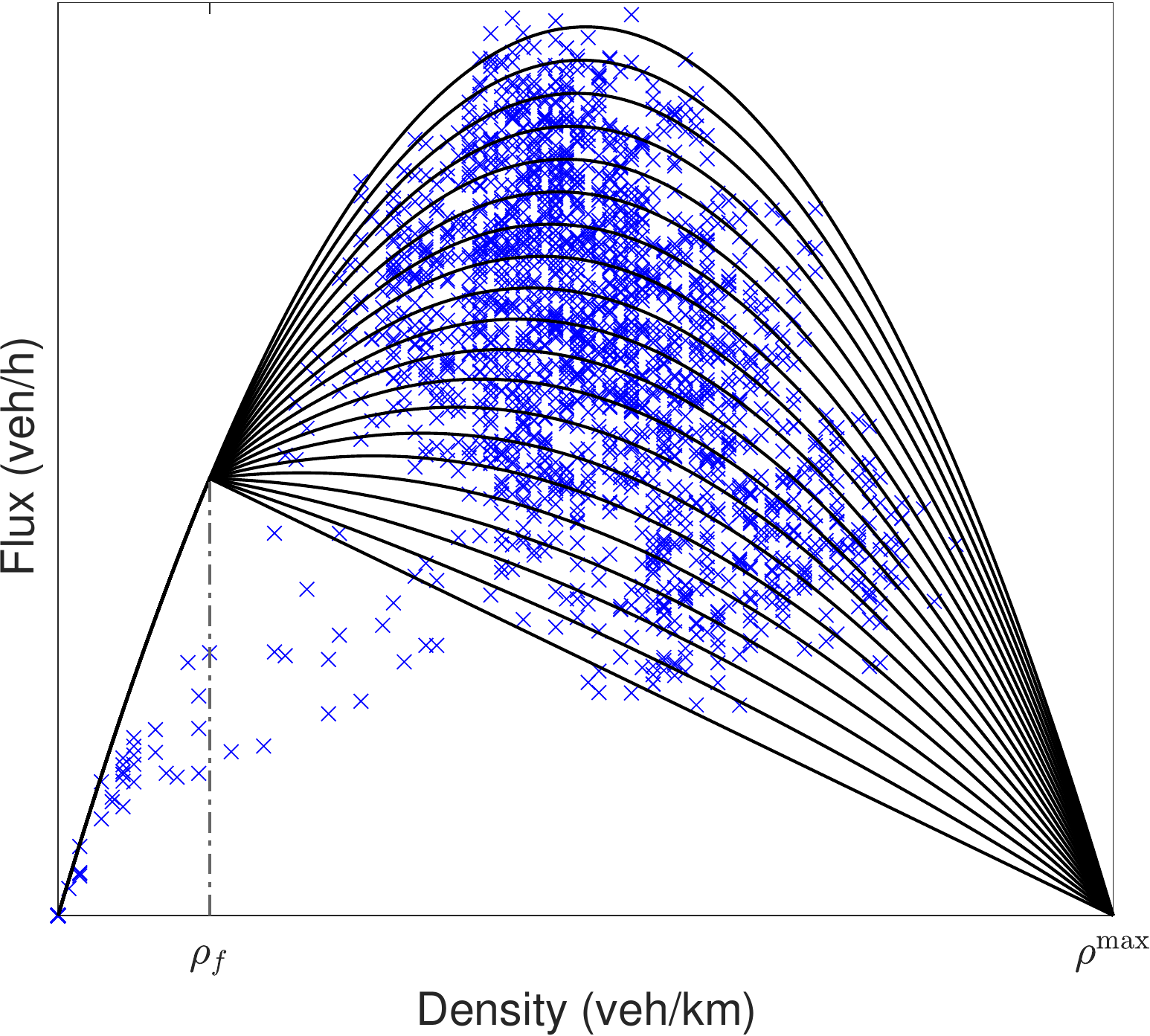}} \quad
\subfloat
{\begin{overpic}[width=0.42\linewidth]{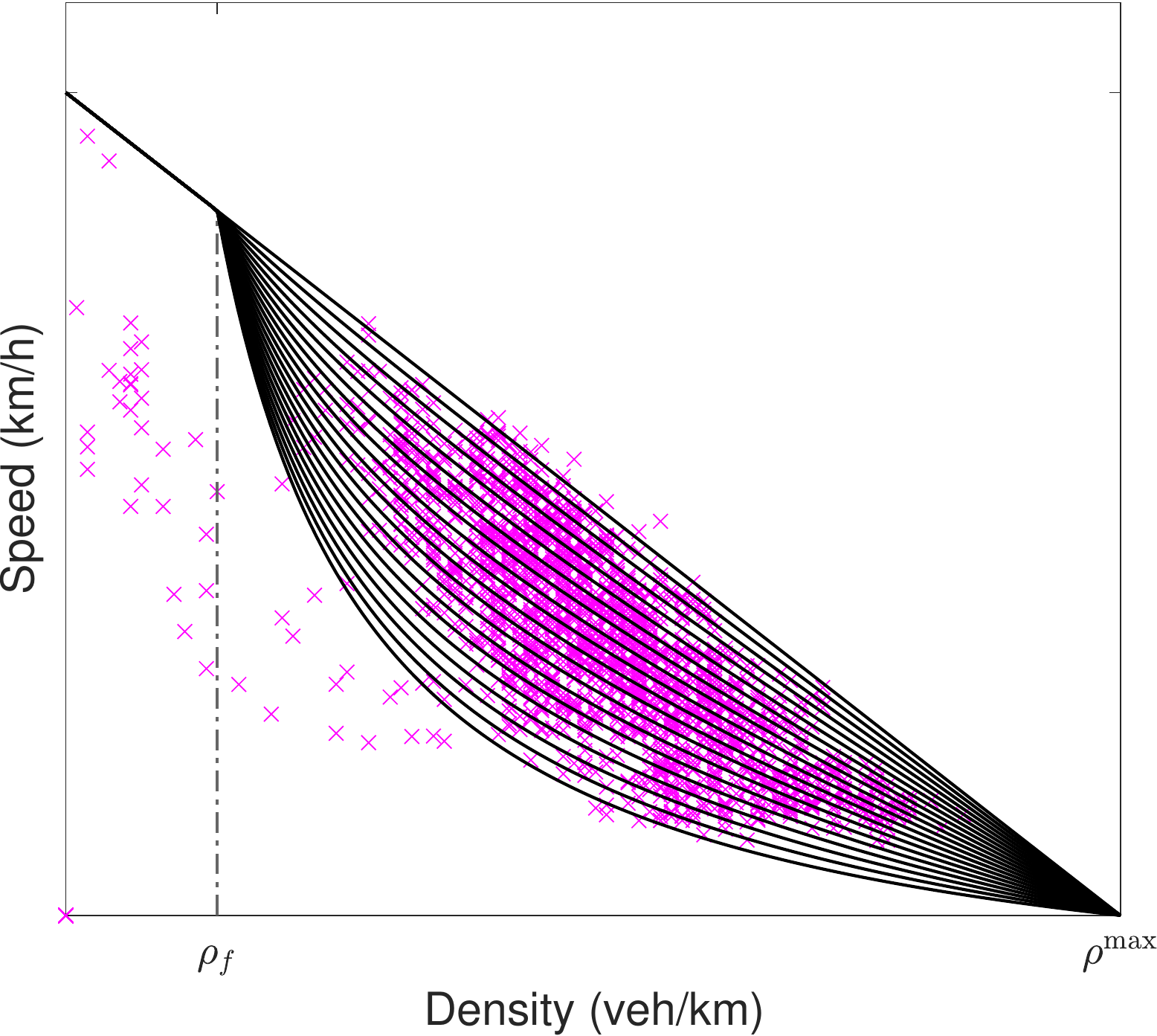}
\put(-5.7,82){\tiny$\vmax$}
\end{overpic}}
\caption{Top: Flow-density relationship (left) and velocity-density relationship (right) from the NGSIM dataset. Bottom: Family of flux functions \eqref{eq:QcgarzFinale} (left) and family of velocity functions \eqref{eq:velocityV} (right) for the calibrated parameters.}
\label{fig:QV}
\end{figure}

We now focus on $\nox$ emissions. The microscopic speed and acceleration included in the NGSIM dataset can be fed directly in \eqref{eq:EmissionRate} providing microscopic $\nox$ emissions produced by each vehicle.  Then, we sum the emissions of vehicles along the entire road
\begin{equation}
E^{\mathrm{true}}(t^{n}) = \sum_{i=1}^{N_{car}(t^{n})}E_{i}(t^{n}),
\label{eq:emissioniTotMicro}
\end{equation}
where $N_{car}(t^{n})$ is the number of vehicles crossing the road at time $t^{n}$ and $E_{i}(t^{n})$ is the emission rate of vehicle $i$ at time $t^{n}$. 

The CGARZ model \eqref{eq:CGARZ2}, calibrated with the NSGIM dataset, is used here to estimate the average density and speed of vehicles along the road. 
The initial density $\rho_0$ and velocity $v_0$ are obtained with a kernel density estimation of the ground-truth data,  specifically the Parzen-Rosenblatt window method.
Given a vehicle location $x_{i}(t)$ and velocity $v_{i}(t)$, density and flow rate functions are obtained as superpositions of Gaussian profiles,
\begin{equation}\label{eq:kernel}
\rho(x,t) = \frac{1}{h}\sum_{i=1}^nK(x,x_i),\quad
v(x,t) = \frac{\sum_{i=1}^n v_i K(x,x_i)}{\sum_{i=1}^n K(x,x_i)},
\end{equation}
where
$K(x,x_i) = \phi\left((x-x_i)/h\right) +\phi\left((x-(2a-x_i))/h\right)+\phi\left((x-(2b-x_i))/h\right)$, $\phi(x) = \exp{(-x^2/2)}/\sqrt{2\pi}\ $,
$h$ is a distance parameter, $a$ and $b$ are the extremes of the road. In this work $h=\myunit{25}{\mymeter}$.

The initial $w_0$ is defined such that
$V(\rho_0(x_{i}),w_0(x_{i}))=v_0(x_{i})$, for $j=1,\dots,N_x$ 
and then $y_0(x_{i}) = \rho_0(x_{i}) w_0(x_{i})$. Following the numerical procedure described in Sections \ref{sec:traffNum} we compute the average emission rate $E^{n}_{i}$ of the cell $x_{i}$ at time $t^{n}$, for all $i$ and $n$, by means of \eqref{eq:emissioni}. 
Similarly to the microscopic case \eqref{eq:emissioniTotMicro}, we sum the emission rates all over the cells
\begin{equation}
E^{\mathrm{mod}}(t^{n}) =  \sum_{i=1}^{N_{x}}E^{n}_{i},
\label{eq:emissioniTotMacro}
\end{equation}
where $t^{n} = n\delta t$, with $\delta t=\myunit{0.1}{\mysecond}$ is the time frame of the NGSIM dataset. 

Two formulas to compute the acceleration were proposed in \eqref{eq:accAnalitica} and \eqref{eq:acc}. The former is analytical and adapted for macroscopic models, while the latter is discrete and can be used to any type of data.
In Figure \ref{fig:acc} we compare the numerical results using the two formulations.  
The red-solid line of the left plot represents the $\nox$ emission rate computed using the discrete acceleration on average density and speed values obtained via kernel density estimation \eqref{eq:kernel} from NGSIM trajectory data. The blue-circles line, instead, represents the ground-truth emission rate  \eqref{eq:emissioniTotMicro}. The results are quite similar, suggesting the accuracy of the discrete acceleration \eqref{eq:acc}. Finally, on the right plot of Figure \ref{fig:acc} we compare the emission rate of $\nox$ computed with equation \eqref{eq:emissioni}, using the two different definitions of the acceleration function \eqref{eq:accAnalitica} and \eqref{eq:acc}. The results are almost identical and have same computational cost, and this further certifies the efficiency of the CGARZ model \eqref{eq:CGARZ2} and suggests the use of the analytical formula \eqref{eq:accAnalitica} to estimate emissions. 
\begin{figure}[h!]
\centering
\includegraphics[width=0.405\linewidth]{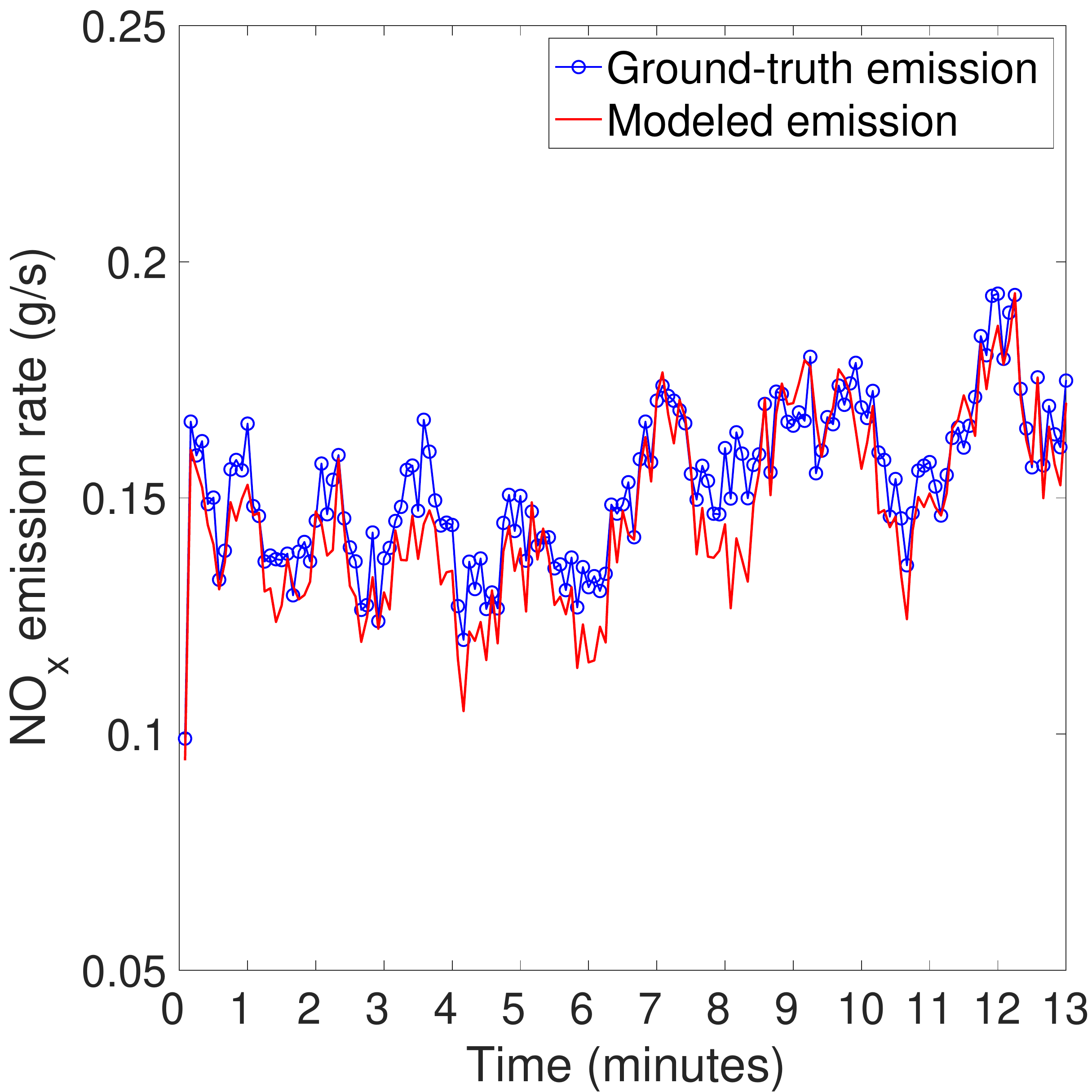}\qquad
\includegraphics[width=0.4\linewidth]{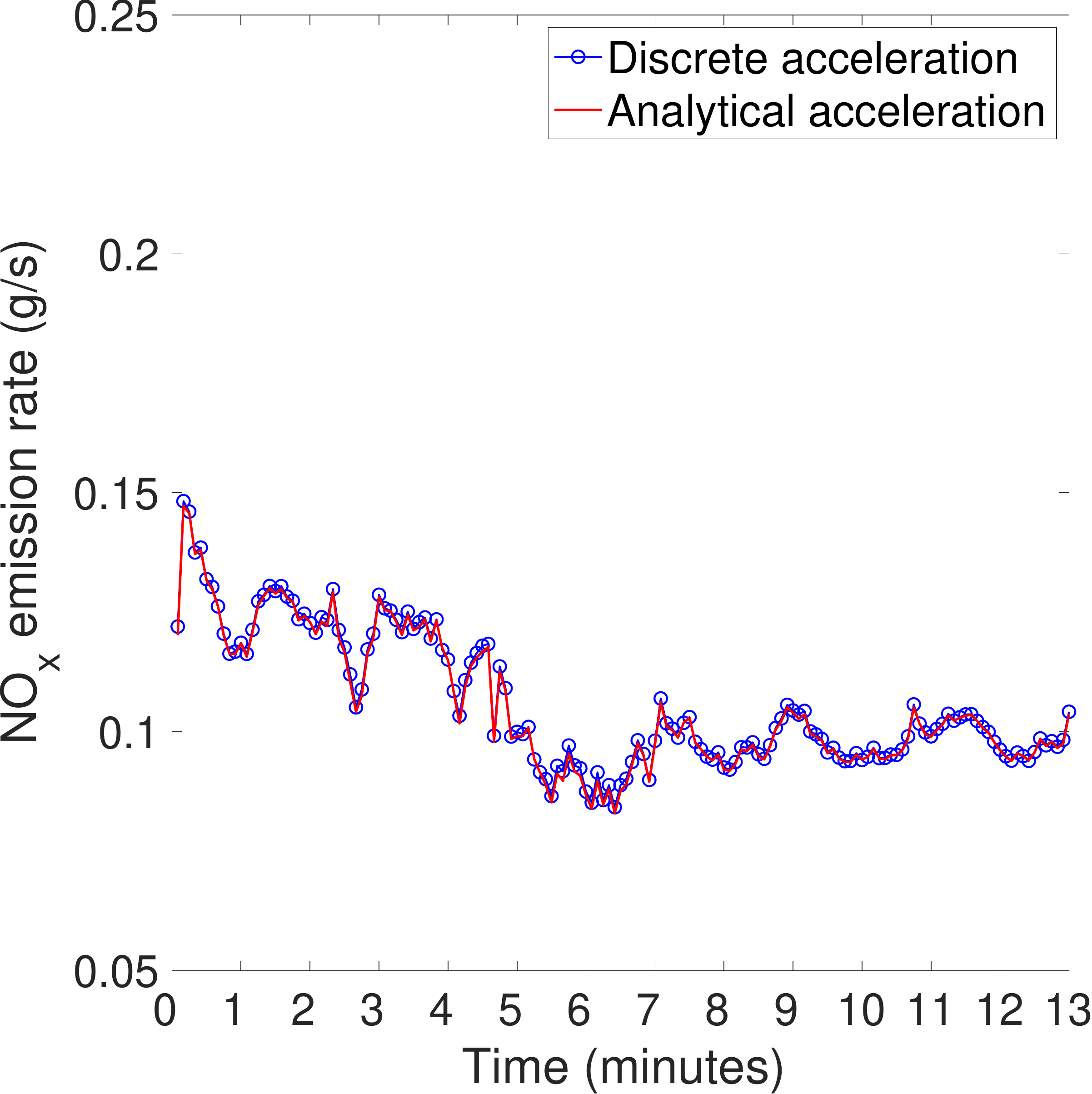}
\caption{Comparison between ground-truth emission rate and modeled emission rate computed using discrete acceleration \eqref{eq:acc} on density and speed via kernel density estimation (left). Comparison of emission rate  computed with the discrete \eqref{eq:acc} and analytical \eqref{eq:accAnalitica} acceleration (right). Both the results refer to 500 meters of road and 13 minutes of simulation (data from 4:01 pm - 4:14 pm of NGSIM dataset).}
\label{fig:acc}
\end{figure}

We compare now the emission rate along the entire road obtained with \eqref{eq:emissioniTotMicro} and \eqref{eq:emissioniTotMacro} respectively, for each period of the NGSIM dataset. The results are computed with 13-minute simulations, in which we exclude the first and the last minute of recorded trajectories for corruption of data.
In Figure \ref{fig:emissioni1} we observe that the emission rate obtained by the CGARZ model \eqref{eq:emissioniTotMacro} (black-dotted) is lower than the ground-truth emission \eqref{eq:emissioniTotMicro} (blue-solid). Improved results are obtained by multiplying the modeled emissions by a proper correction factor (red-circles). Specifically, for each data period $j$, we have computed a correction factor $r_{j}$ via linear regression between the ground-truth emission and the modeled one. Moreover, we define the following error 
\begin{equation}
\mathrm{Error}(r_{j}) = \frac{\norm{E^{\mathrm{true}}-r_{j}E^{\mathrm{mod}}}_{L^{1}}}{\norm{E^{\mathrm{true}}}_{L^{1}}},  \qquad j = 1,2,3,
\label{eq:erroreEm}
\end{equation}
where $E^{\mathrm{true}}$ and $E^{\mathrm{mod}}$ are vectors whose $k$-th components are given by \eqref{eq:emissioniTotMicro} and \eqref{eq:emissioniTotMacro} respectively. Table \ref{tab:erroriEm} shows the errors \eqref{eq:erroreEm} obtained using the three different correction factors for all the time periods of the NSGIM dataset, where $r_{1}=1.42$, $r_{2}=1.35$ and $r_{3}=1.15$. We observe that the correction factors $r_{1}$, $r_2$ and $r_{3}$ give similar results.
\begin{table}[hpbt!]
\centering
\small
\begin{tabular}{|c|c|c|c|}\hline
Period & $\mathrm{Error}(r_{1})$ & $\mathrm{Error}(r_{2})$ & $\mathrm{Error}(r_{3})$\\\hline
4:01 pm - 4:14 pm & 0.1604 & 0.1666 & 0.2204\\\hline
5:01 pm - 5:14 pm & 0.0819 & 0.0842 & 0.1625\\\hline
5:16 pm - 5:29 pm & 0.2304 & 0.1773 & 0.0586\\\hline
\end{tabular}
\caption{Errors given by \eqref{eq:erroreEm} for the three slots of the NGSIM dataset and different correction factor $r_{1}=1.42$, $r_{2}=1.35$ and $r_{3}=1.15$.}
\label{tab:erroriEm}
\end{table}
 \begin{figure}[hpbt!]
\centering
\subfloat[][{Data from 4:01 pm - 4:14 pm and correction factor $r_{1}$.}]
{\includegraphics[scale=0.168]{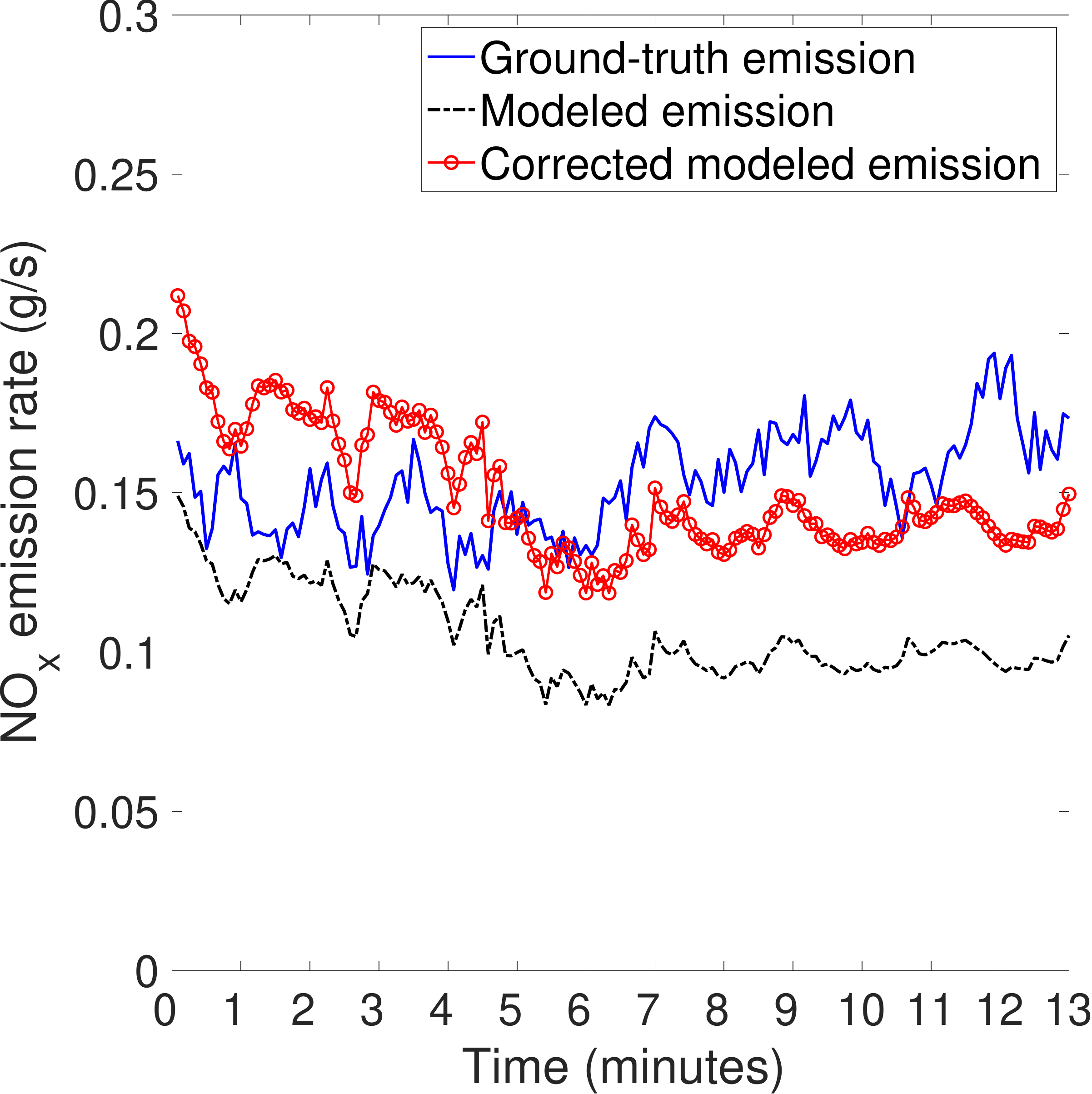}}\quad
\subfloat[][{Data from 5:01 pm - 5:14 pm and correction factor $r_{1}$.}]
{\includegraphics[scale=0.168]{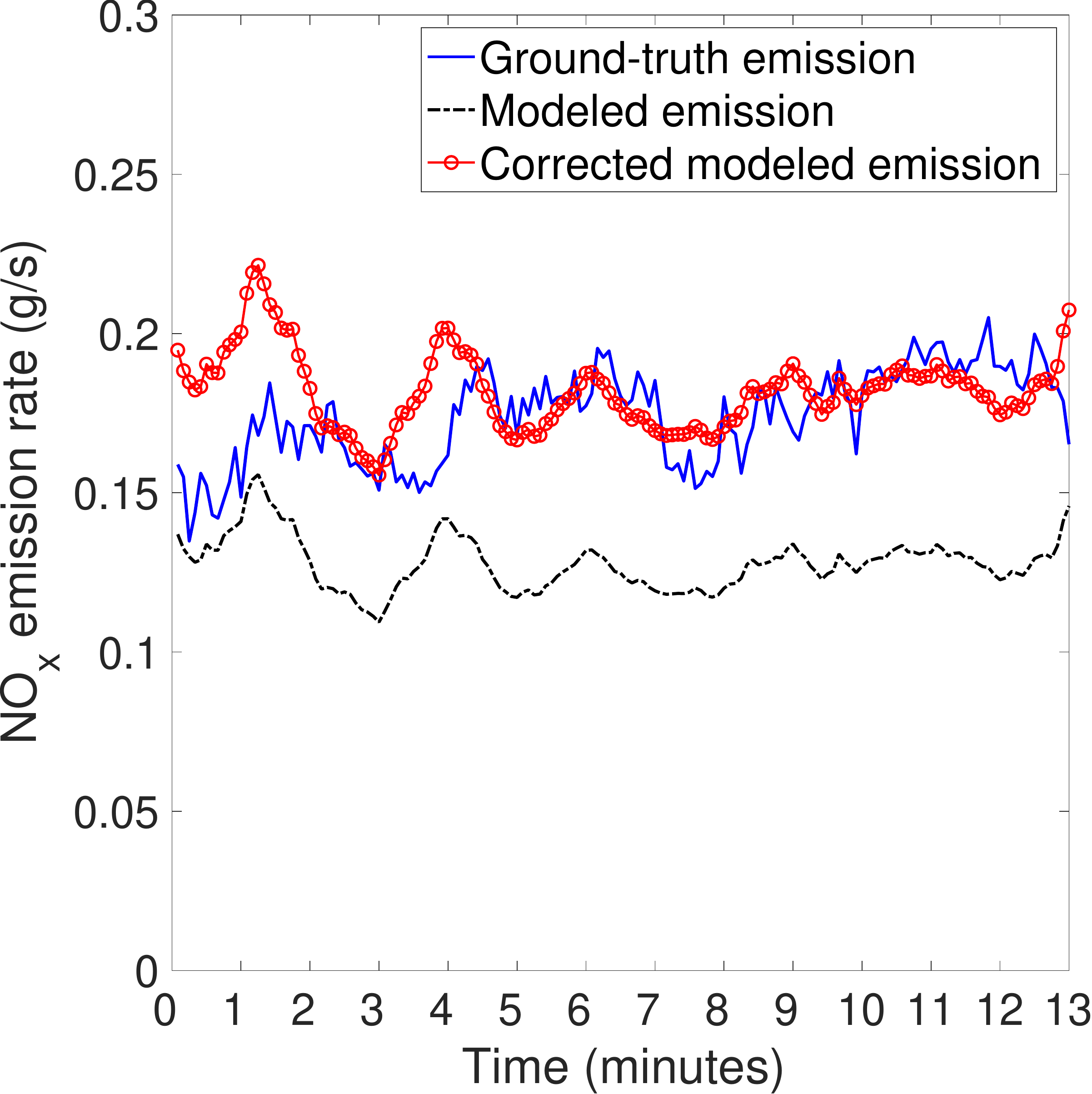}}\quad
\subfloat[][{Data from 5:16 pm - 5:29 pm and correction factor $r_{1}$.}]
{\includegraphics[scale=0.168]{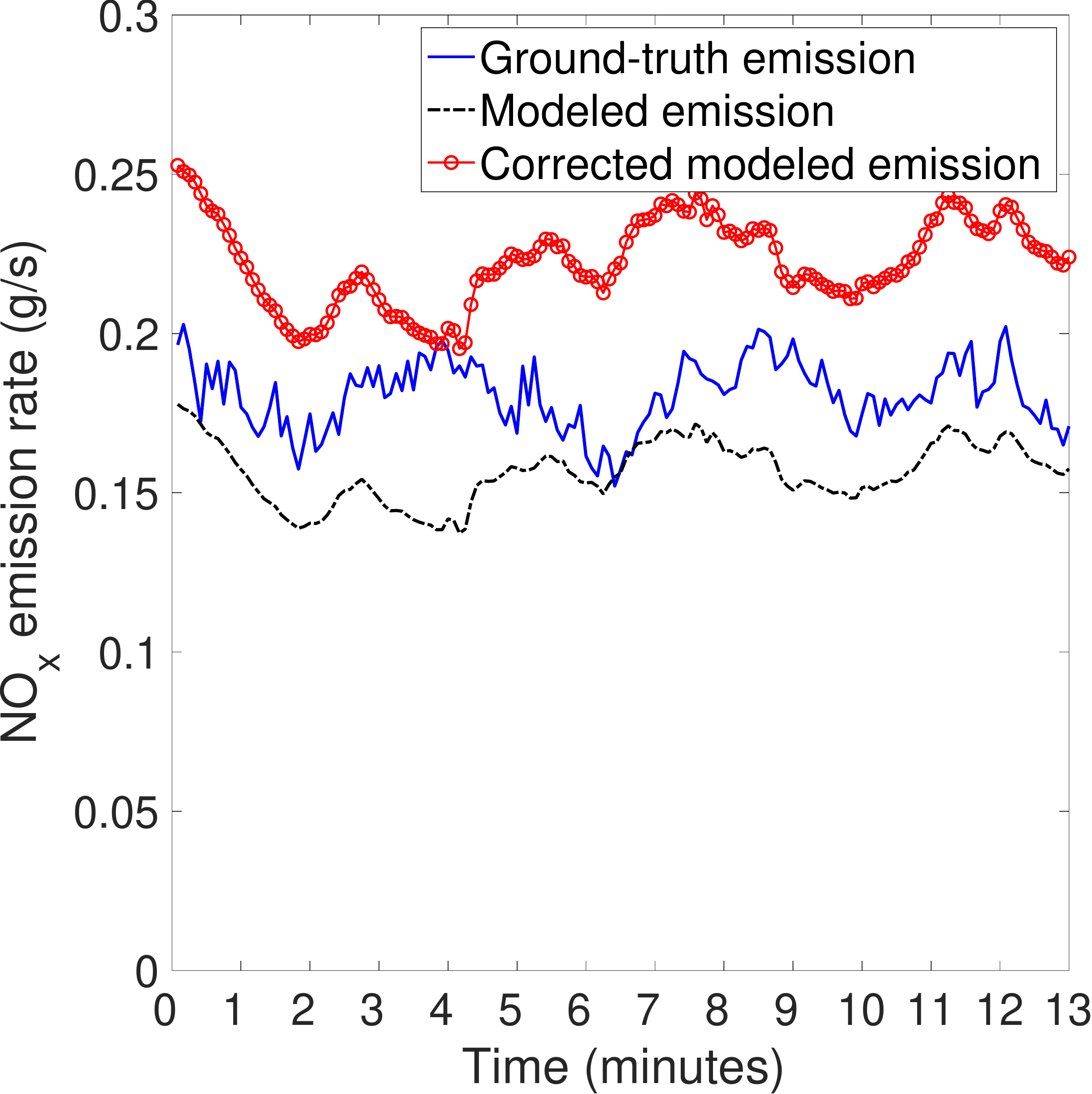}}
\\
\subfloat[][{Data from 4:01 pm - 4:14 pm and correction factor $r_{2}$.}]
{\includegraphics[scale=0.168]{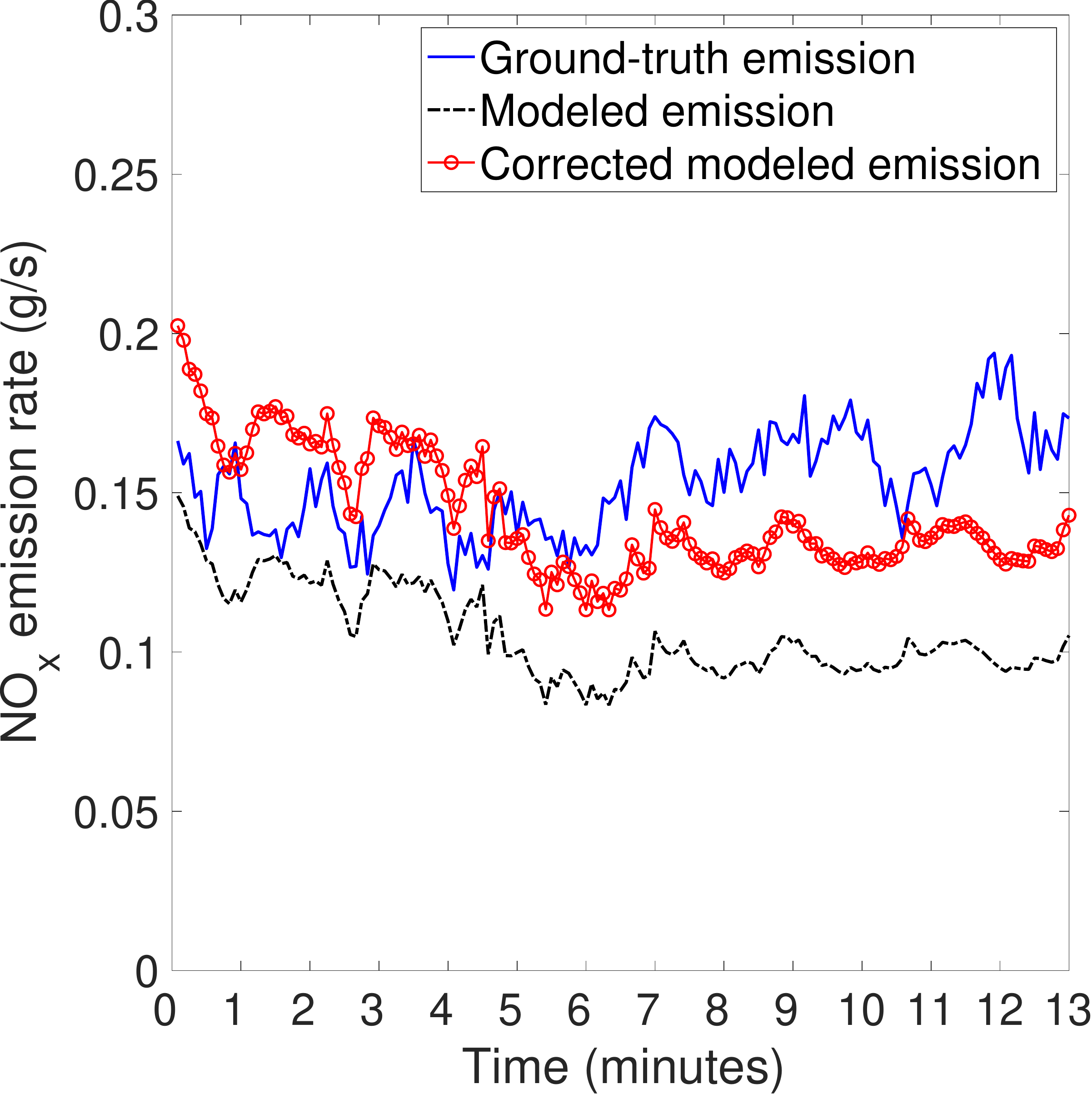}}\quad
\subfloat[][{Data from 5:01 pm - 5:14 pm and correction factor $r_{2}$.}]
{\includegraphics[scale=0.168]{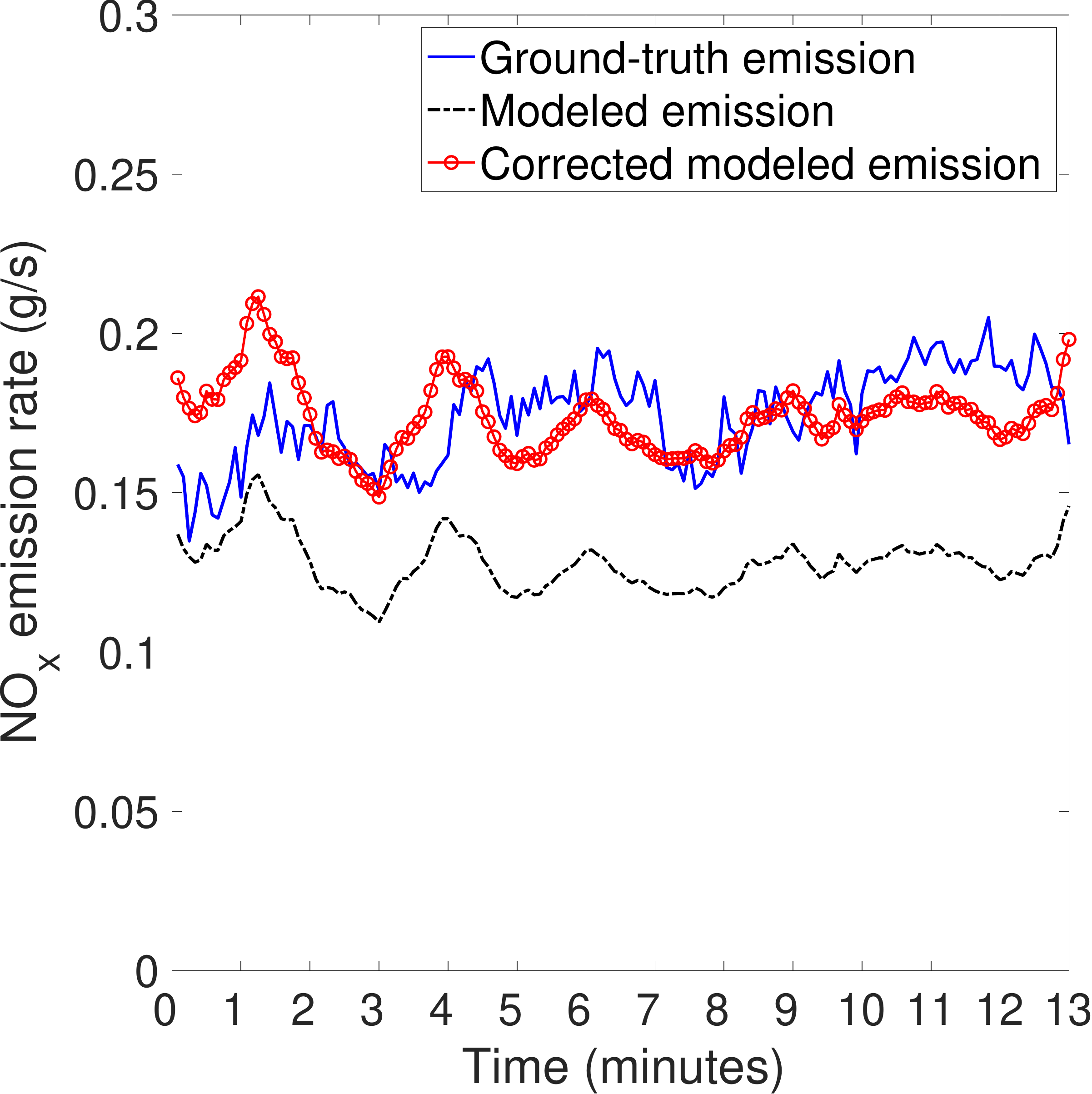}}\quad
\subfloat[][{Data from 5:16 pm - 5:29 pm and correction factor $r_{2}$.}]
{\includegraphics[scale=0.168]{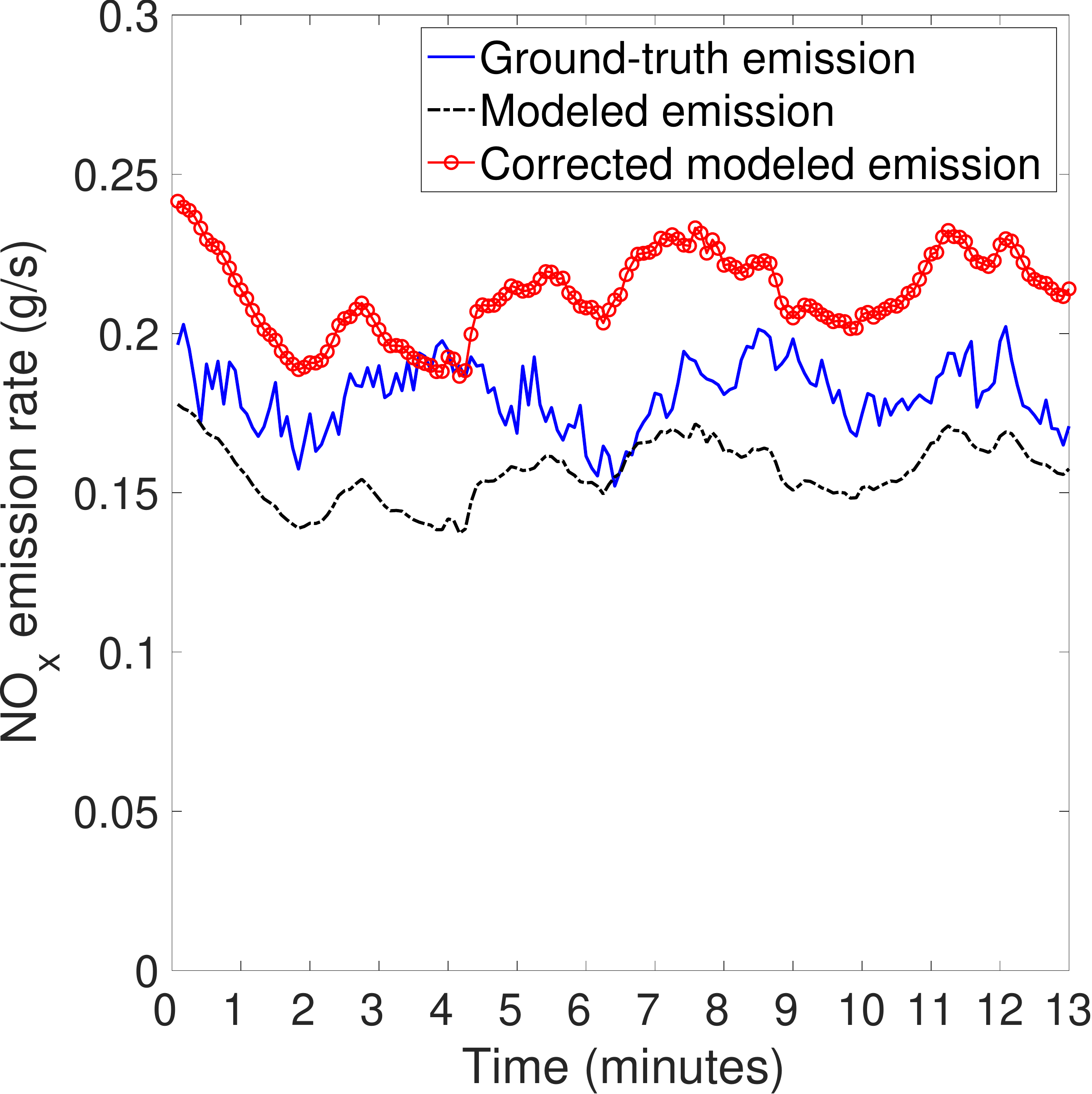}}
\\
\subfloat[][{Data from 4:01 pm - 4:14 pm and correction factor $r_{3}$.}]
{\includegraphics[scale=0.168]{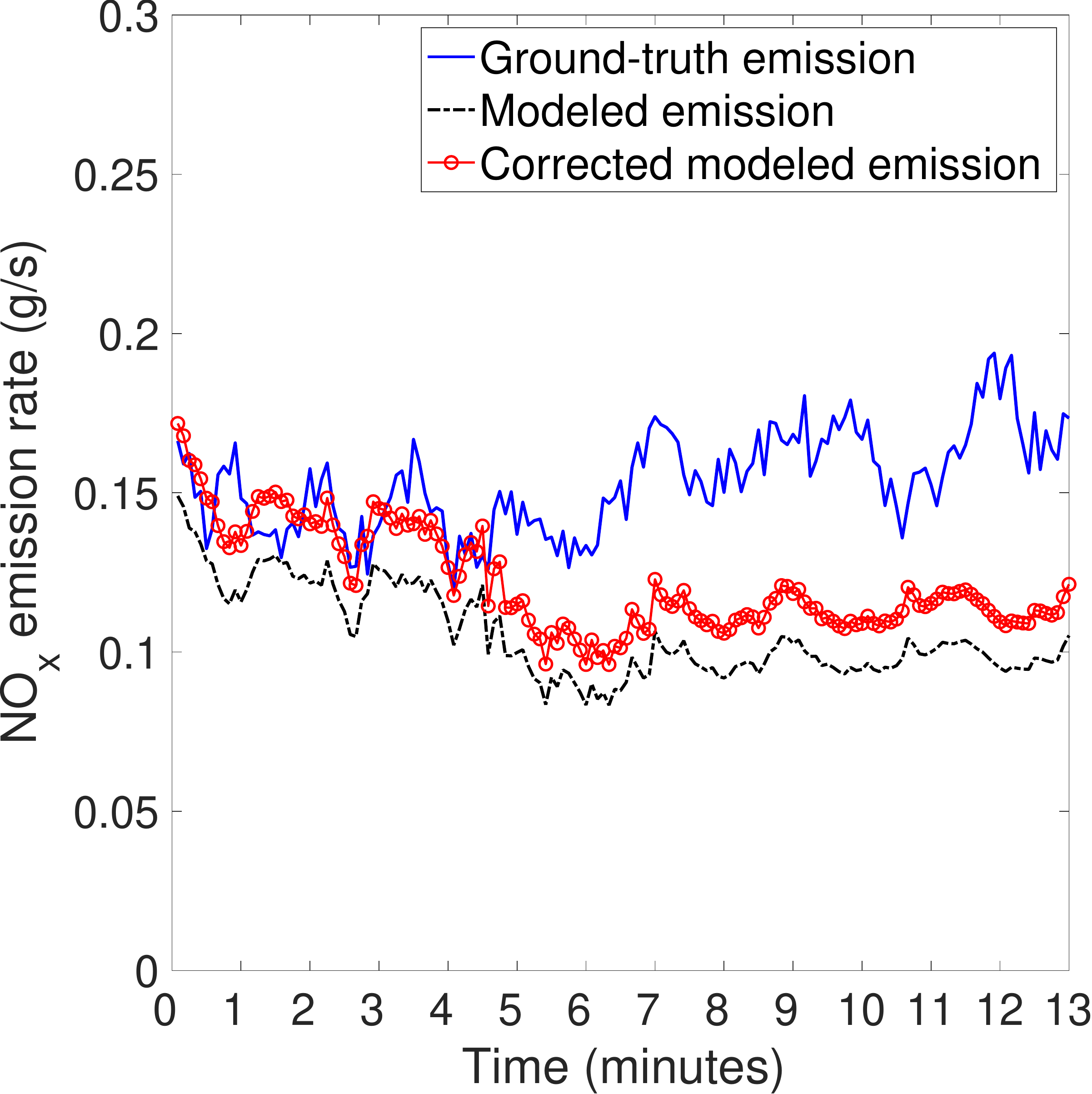}}\quad
\subfloat[][{Data from 5:01 pm - 5:14 pm and correction factor $r_{3}$.}]
{\includegraphics[scale=0.168]{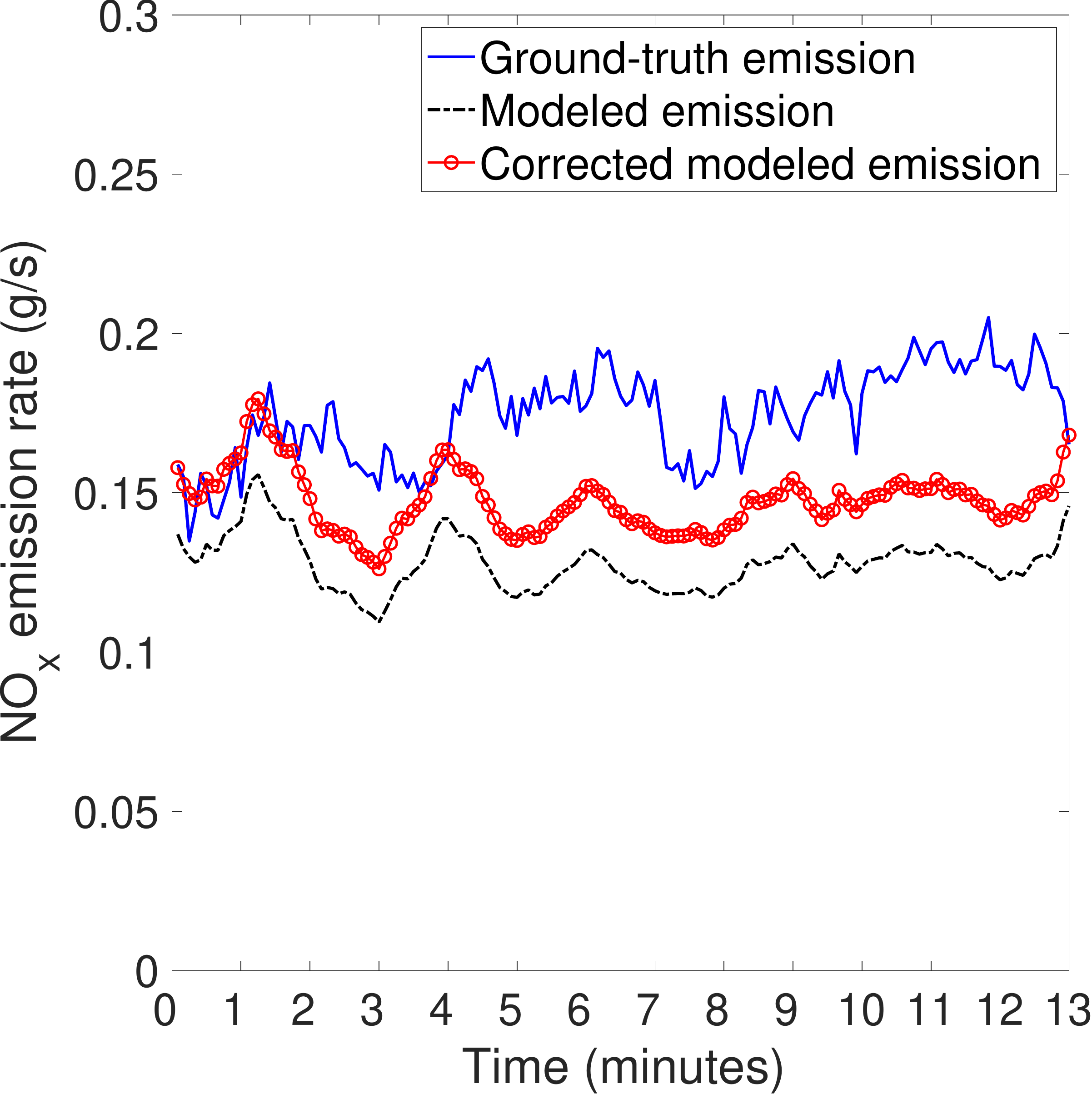}}\quad
\subfloat[][{Data from 5:16 pm - 5:29 pm and correction factor $r_{3}$.}]
{\includegraphics[scale=0.168]{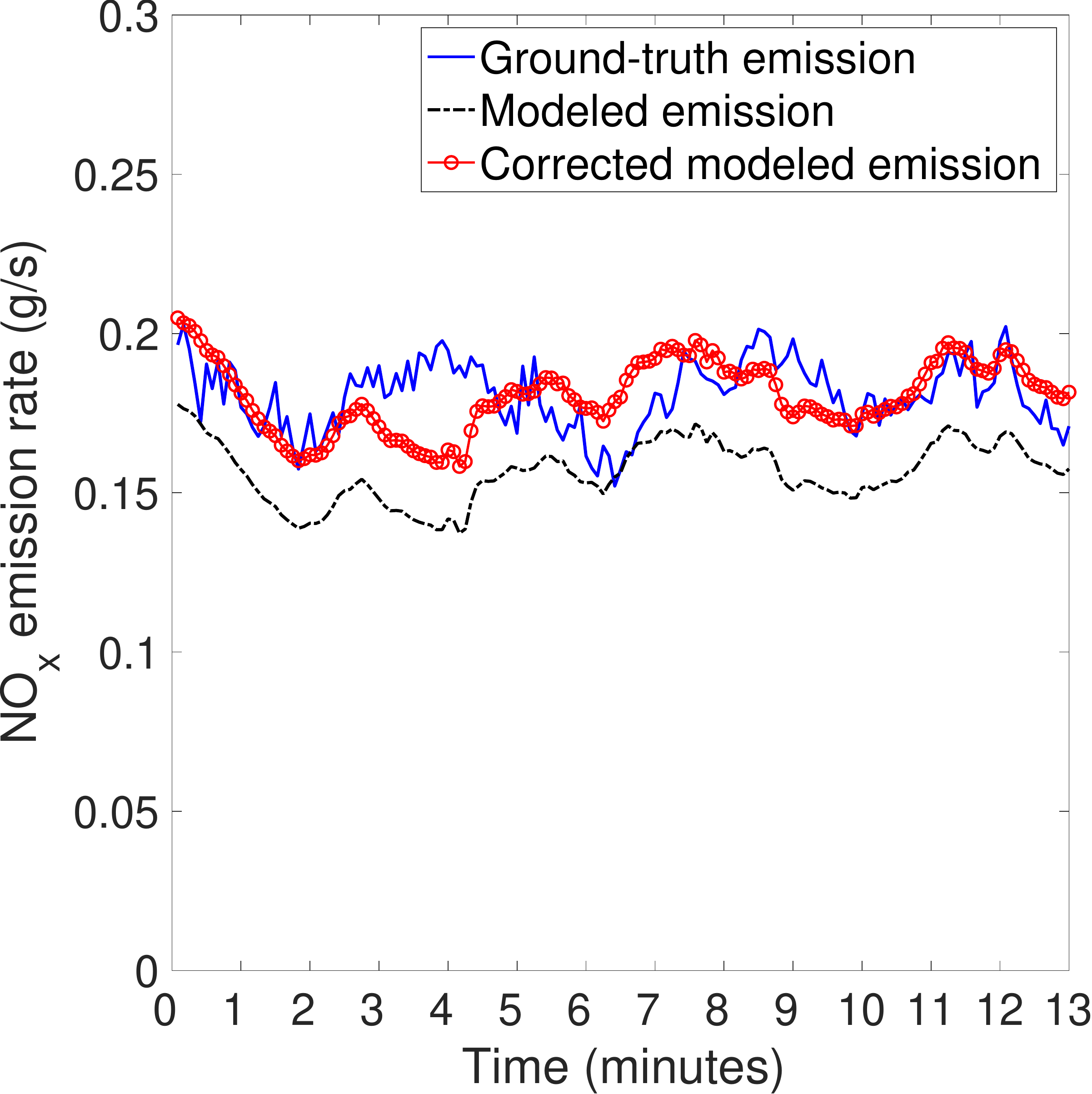}}
\caption{Comparison of modeled (black-dotted), modeled with correction factors $r_{j}$ (red-circles) and ground-truth (blue-solid) emission rates along 500 meters of road during 13 minutes of simulation for the three time periods of the NSGIM dataset. The top row is computed for $r_{1}=1.42$, the central row for $r_{2}=1.35$ and the bottom row for $r_{3} =1.15$.}
\label{fig:emissioni1}
\end{figure}

\subsubsection{Traffic dynamics}
Let us consider the CGARZ traffic model \eqref{eq:CGARZ2} on a time horizon
$[0,T]$ and on a road with one lane parametrized by $[0,L]$. 
We fix the maximum speed $\vmax=\myunit{70}{\mykilo\mymeter\per\myhour}$, maximum density $\rhom=\myunit{133}{\mathrm{veh}\per\mykilo\mymeter}$ and
left boundary condition $\rho(0,t)=\myunit{52}{\mathrm{veh}\per\mykilo\mymeter}$ $\forall t$, while we use Neumann boundary
condition on the right, which corresponds to allowing all vehicles to leave the road.
The other parameters used in all simulations are $T=\myunit{30}{\mymin}$, $L=\myunit{3}{\km}$, $\dx=\myunit{0.03}{\km}$, $\dt=\myunit{1.5}{\mysecond}$ and 
the initial density $\rho_{0}=\myunit{52}{\veh\per\km}$ and the initial property $w_{0}(x)=\wr$ for $x \leq 2L/3$ and $w_{0}(x) = \wl$ otherwise, where $\wl = f(\rhof ) = 1140$ and $\wr = g(\rhom/2) = 2327$, with $f$ in \eqref{eq:f} and $g$ in \eqref{eq:g}.

In the following we show different traffic scenarios to evaluate the production of ozone.

\paragraph{\labeltext{Traffic dynamic 1}{test1}: road without traffic lights}
The dynamic is described by an initial shock wave around the middle of the road and a rarefaction wave stemming from the right end of the road.
The shock wave propagates backward for the first half of the simulation, when the interaction with the rarefaction wave, and the consequent cancellation, changes the shock speed to positive.
In Figure \ref{fig:test1_traffico} we compare the 3D plots of density, speed, acceleration and $\nox$ emission rates. The four graphs have the same shape, since they depend on the density of vehicles. 
The acceleration reaches the minimum value along the blue curve shown in the graph, while the maximum value is reached at the beginning of the simulation at the end of the road, when the vehicles leave the road with maximum flux. 
Finally, the $\nox$ emission rate has a peak in correspondence of the highest values of acceleration and it is equal to 0 along the curve with the darkest blue.  
\begin{figure}[h!]
\centering
\subfloat[][Density ($\veh/\km$).]{
\includegraphics[width=0.35\linewidth]{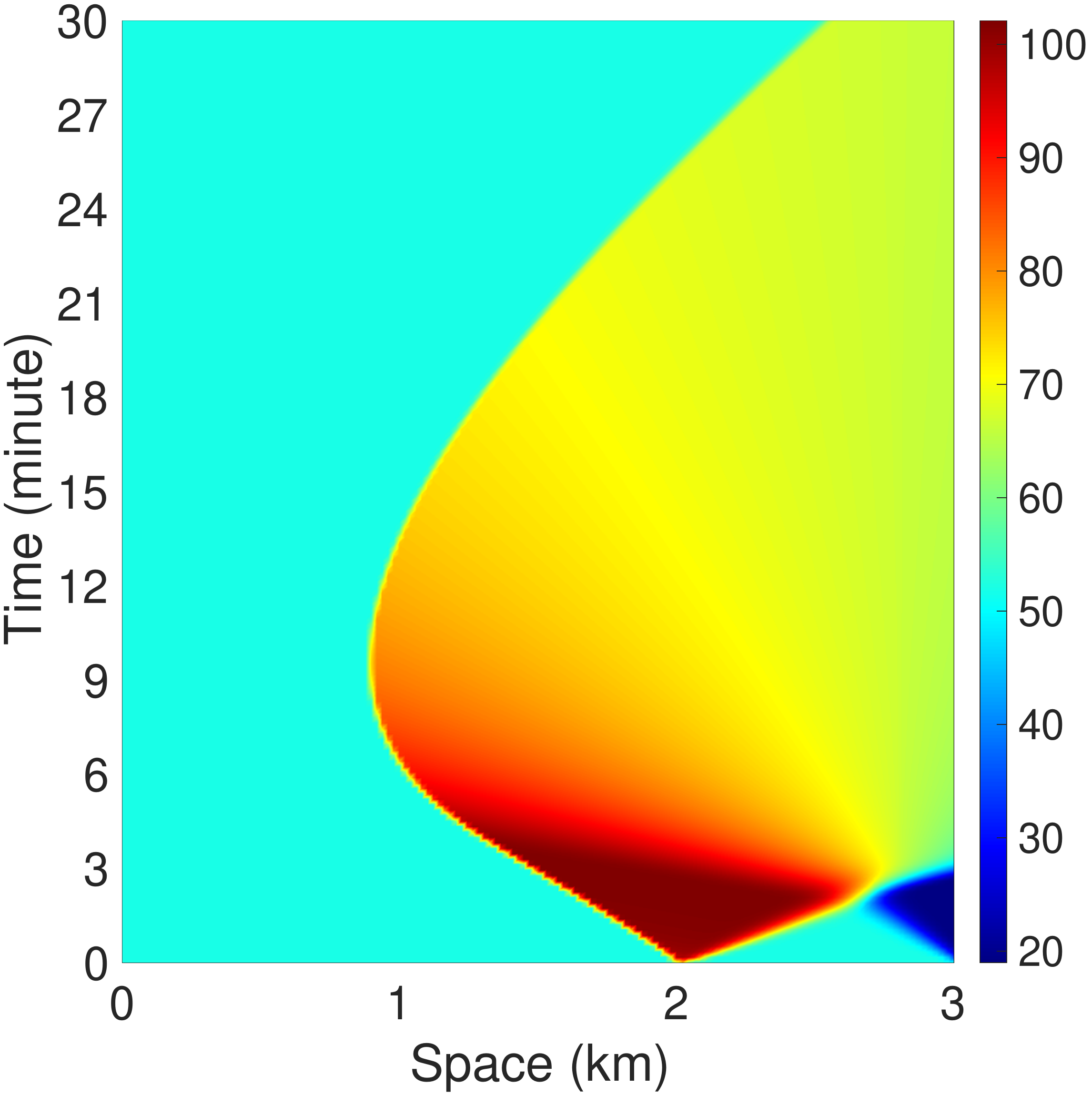}
}\quad
\subfloat[][Speed ($\km\per\myhour$).]{
\includegraphics[width=0.345\linewidth]{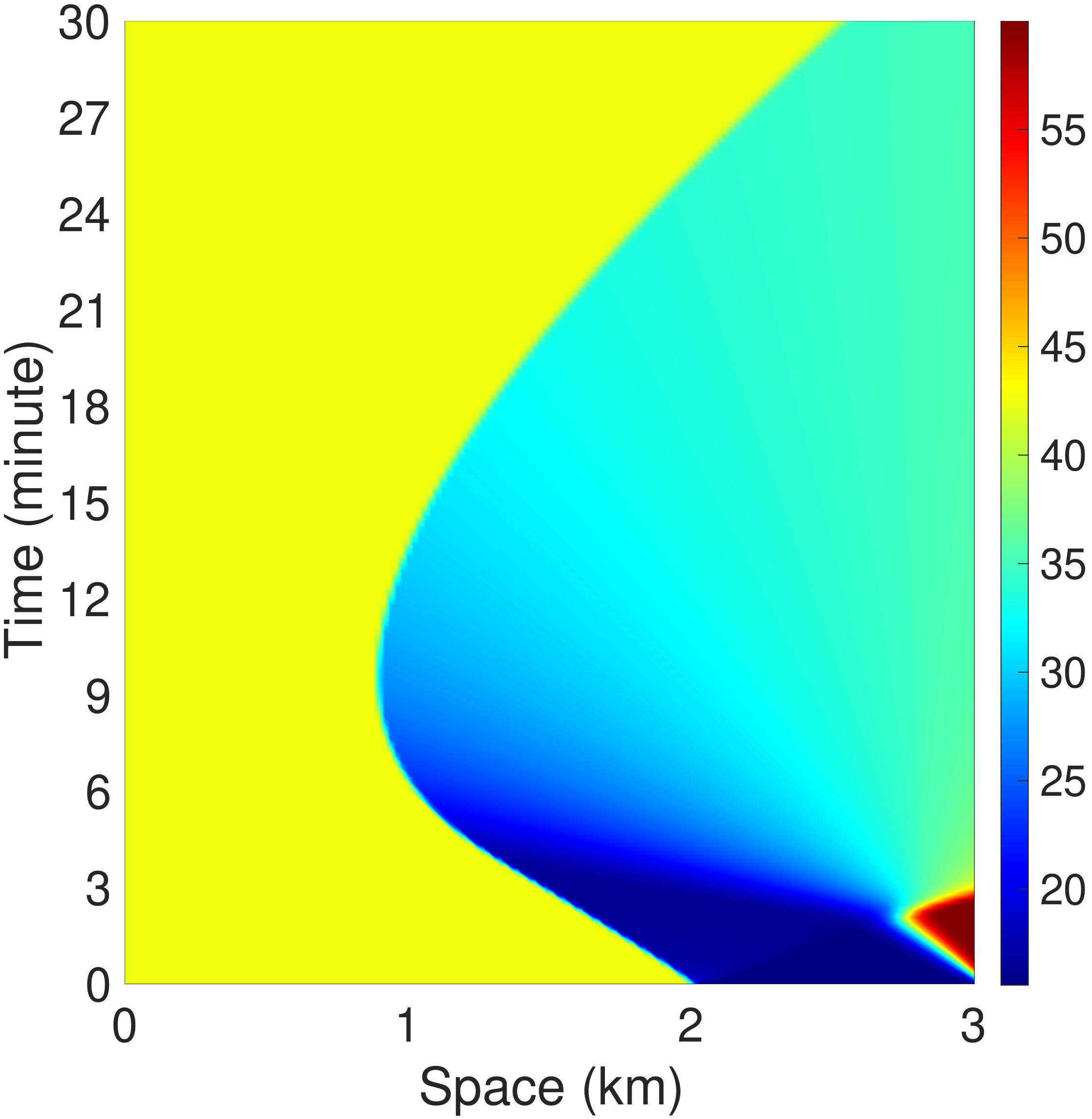}}
\\
\subfloat[][Acceleration ($\mymeter\per\mysecond^{2}$).]{
\includegraphics[width=0.35\linewidth]{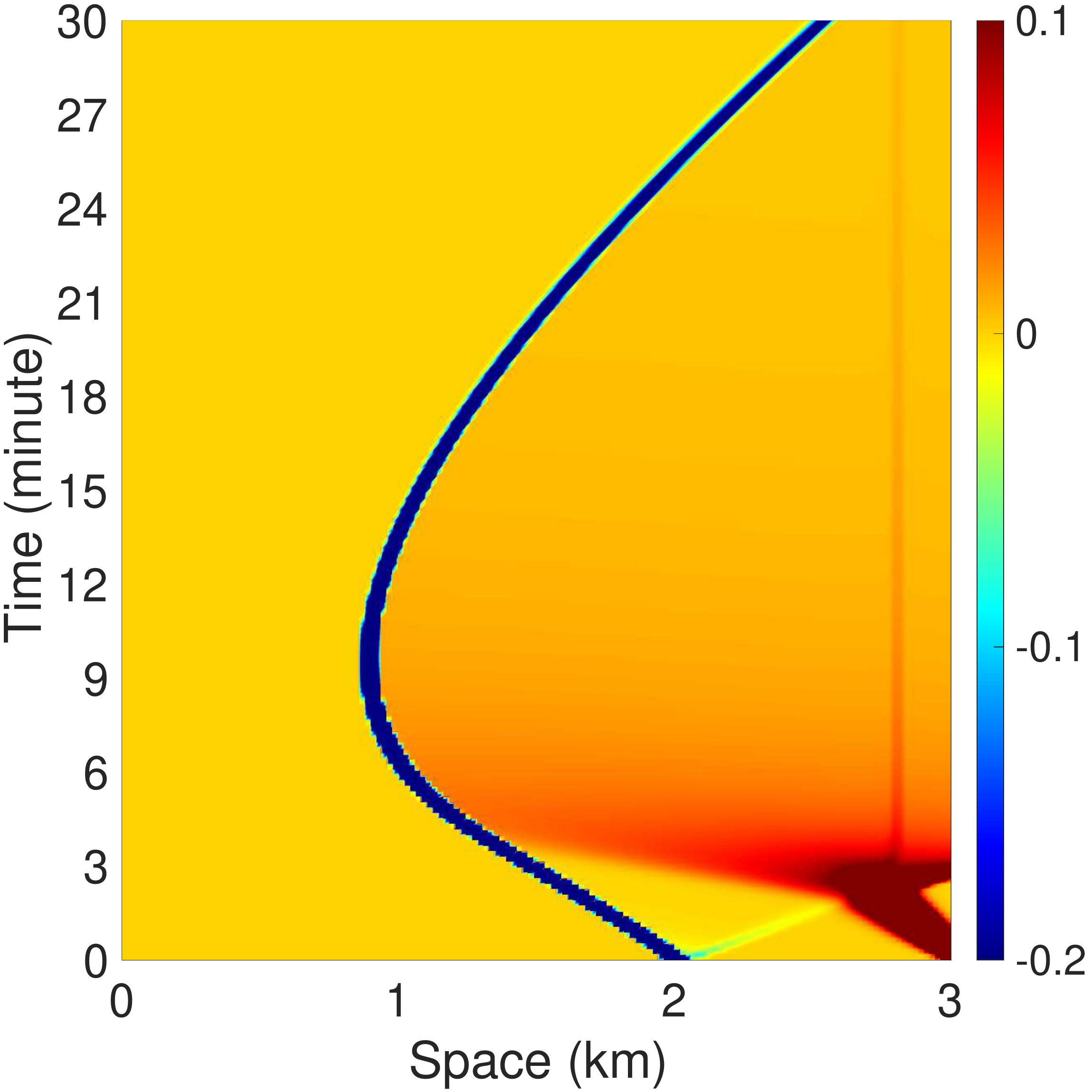}}
\quad
\subfloat[][Emission rates ($\mygram\per\myhour$).]{
\includegraphics[width=0.345\linewidth]{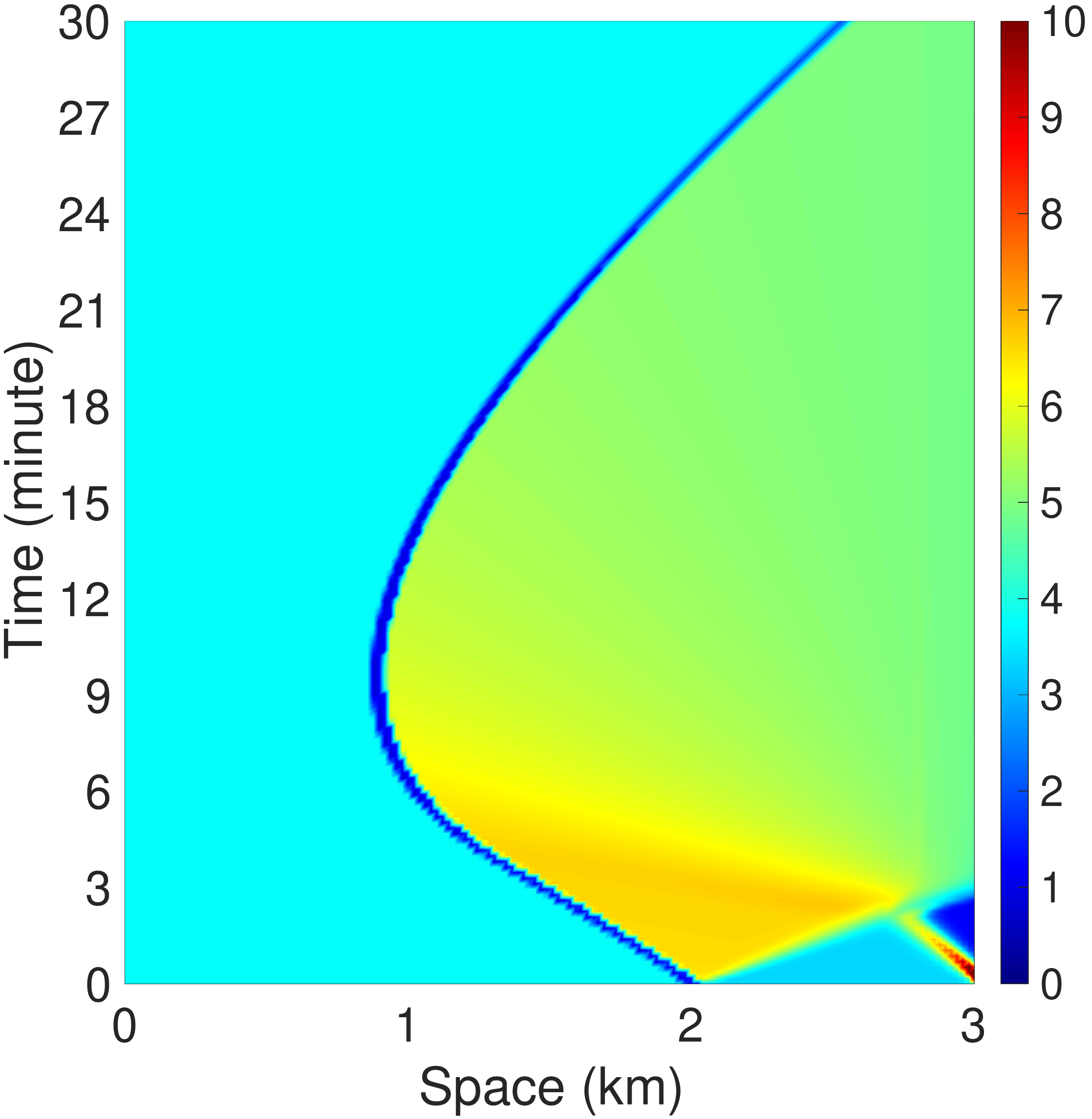}}
\caption{\ref{test1}: Variation of density (a), speed (b), analytical acceleration (c) and $\nox$ emissions (d) in space and time.}
\label{fig:test1_traffico}
\end{figure}

On the left plot of Figure \ref{fig:test1_emissioni} we show data points
of speed, acceleration and emission obtained along the numerical test. More precisely, the horizontal and vertical axes denote speed and acceleration, respectively, while the color gives the $\nox$ emission value. We observe that the $\nox$ emission is higher for positive values of the acceleration and at low speed with values of acceleration near to $-0.5\,\mymeter\per\mysecond^{2}$, and it decreases with negative acceleration. On the right plot of Figure \ref{fig:test1_emissioni} we show the variation in time of the total emission, defined as the sum on the cells of the emission rates, at any time. For this test, the total emission increases until the dynamics is described by the shock wave, and then it starts to decrease.
\begin{figure}[hpbt!]
\centering
\includegraphics[width=0.34\linewidth]{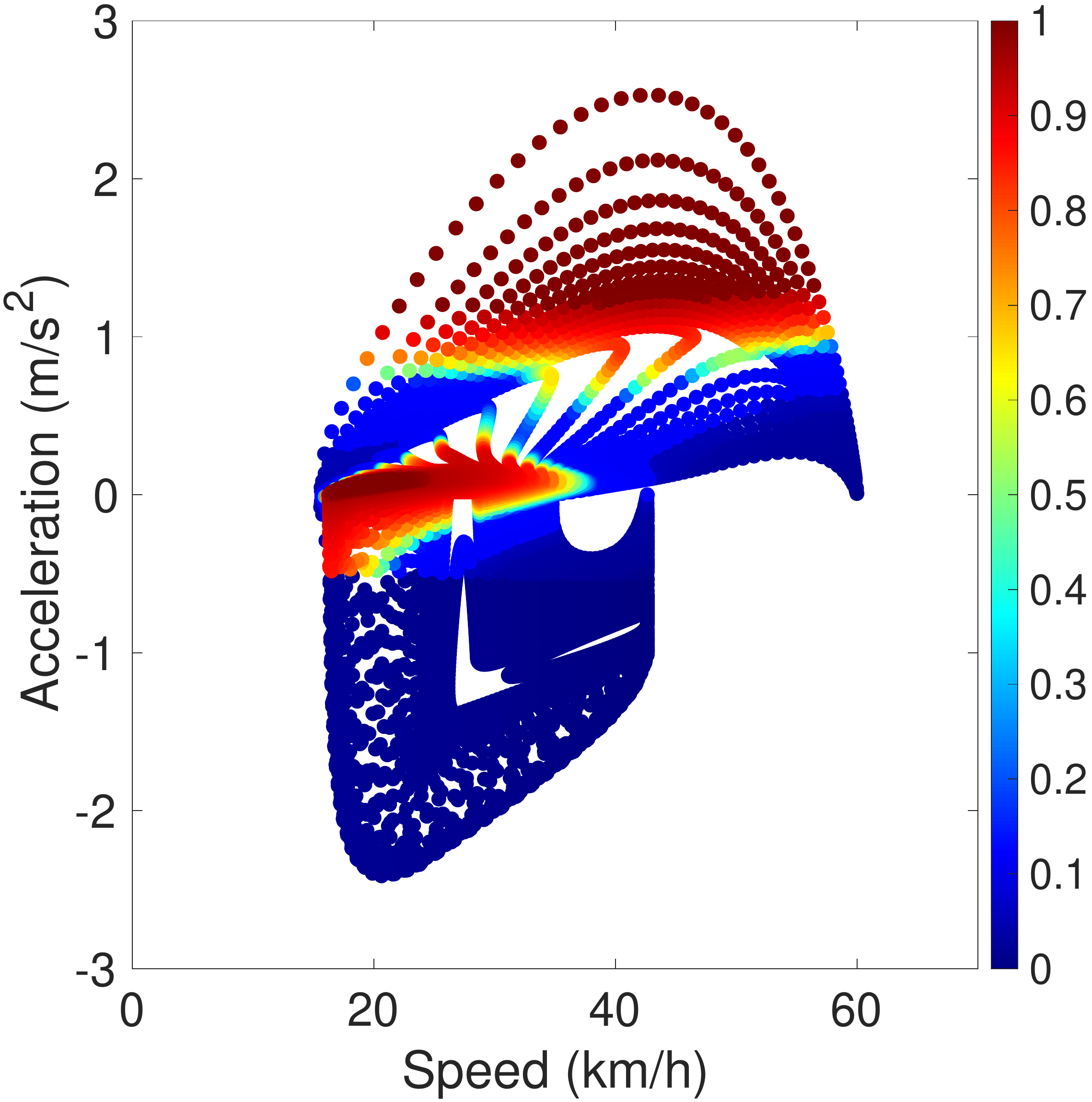}\quad
\includegraphics[width=0.35\linewidth]{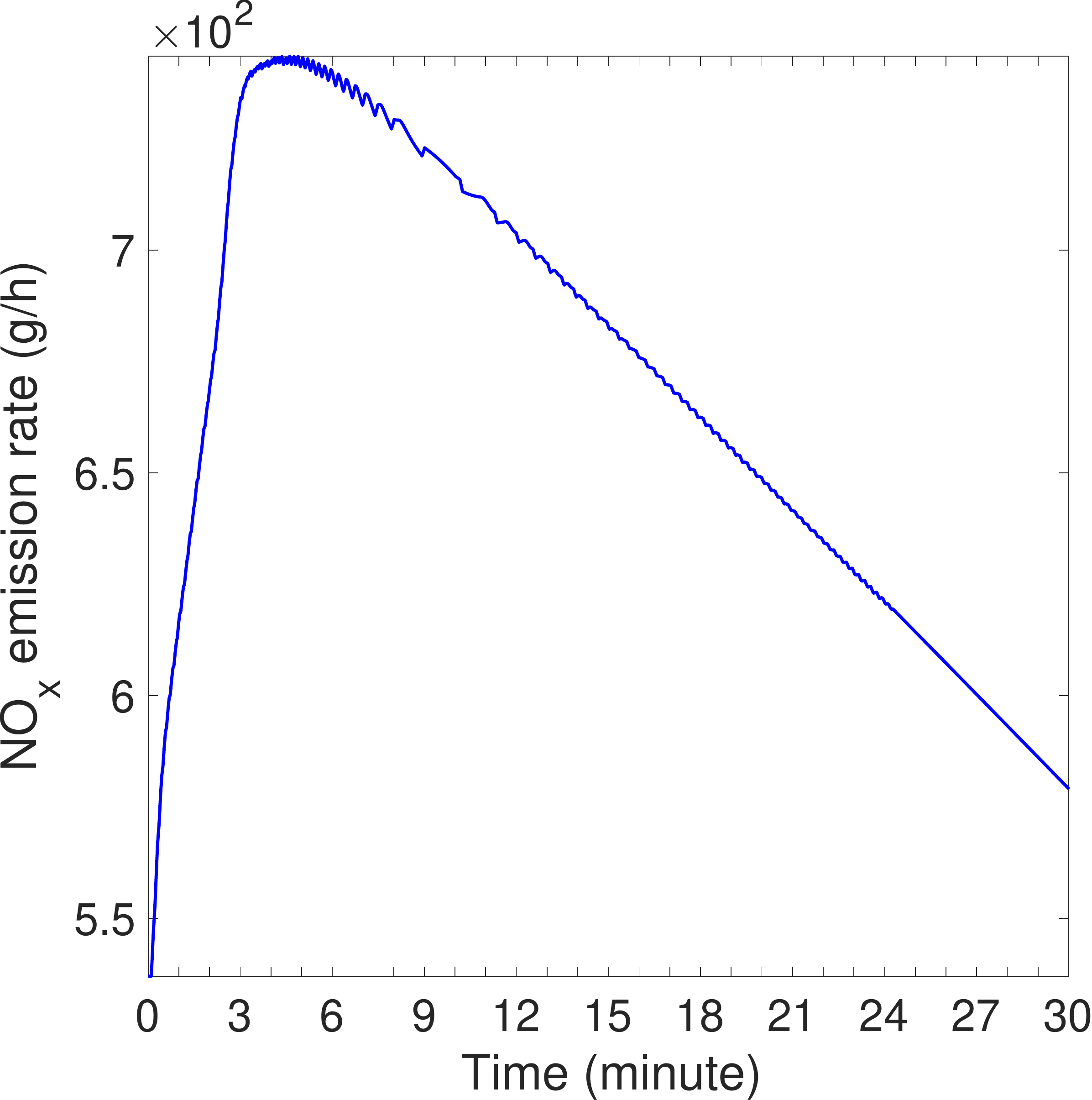}\\
\caption{\ref{test1}: $\nox$ emission rate ($\mygram\per\myhour$) as a function of speed and acceleration (left); variation in time of the total emission rate ($\mygram\per\myhour$) along the entire road (right).}
\label{fig:test1_emissioni}
\end{figure}

\paragraph{\labeltext{Traffic dynamic 2}{test2}: road with traffic lights}
Here we test the effect of different traffic light cycles varying the time frame of the red phase. The latter corresponds to a Neumann boundary
condition imposing vanishing outflow, while the green phase correspond
to Neumann boundary condition allowing all cars to leave the road.
We start by showing the solution obtained with a traffic light cycle of $5$ minutes with a $2$ minutes red phase. 
In Figure \ref{fig:test2_traffico} we show density, speed, acceleration and $\nox$ emission rate in space and time. The wave with high density created by the red traffic lights takes about 9 minutes to reach the left boundary of the road. Once it reaches the left boundary of the road we see a periodic behavior in all the graphs, determined by the traffic lights. The graphs related to density and speed have opposite behavior: when the density increases the speed decreases and vice versa. Similar to test \ref{test1}, the acceleration reaches the maximum values when the traffic light turns green and the vehicles leave the road. Again, the peaks of $\nox$ emission rates correspond to the highest acceleration values.

In Figure \ref{fig:test2_emissioni} we show on the left the emission rate as a function of speed and acceleration, and on the right the total emission along the road in time. Similar to Figure \ref{fig:test1_emissioni}, the left graph shows higher emission levels at positive acceleration and at low speed and values of acceleration near to $\myunit{-0.5}{\mymeter\per\mysecond^{2}}$. In the graph we can see two phases, horizontally divided at height $-0.5$. We observe that $\myunit{-0.5}{\mymeter\per\mysecond^{2}}$ is the acceleration value which distinguishes the two possible choices of the parameters in \eqref{eq:EmissionRate}, see Table \ref{table:EmissionRateParam}. 
The right graph of Figure \ref{fig:test2_emissioni}  shows the total emission in time, where the red and green lines represent the relative traffic light. We observe that, during the first 10 minutes, the emission rate increases faster when the traffic light is green and slower when it is red. Then, it reaches a maximum value after which it assumes a periodic behavior which depends on the traffic light.

\begin{figure}[hpbt!]
\centering
\subfloat[][Density ($\veh/\km$).]{
\includegraphics[width=0.35\linewidth]{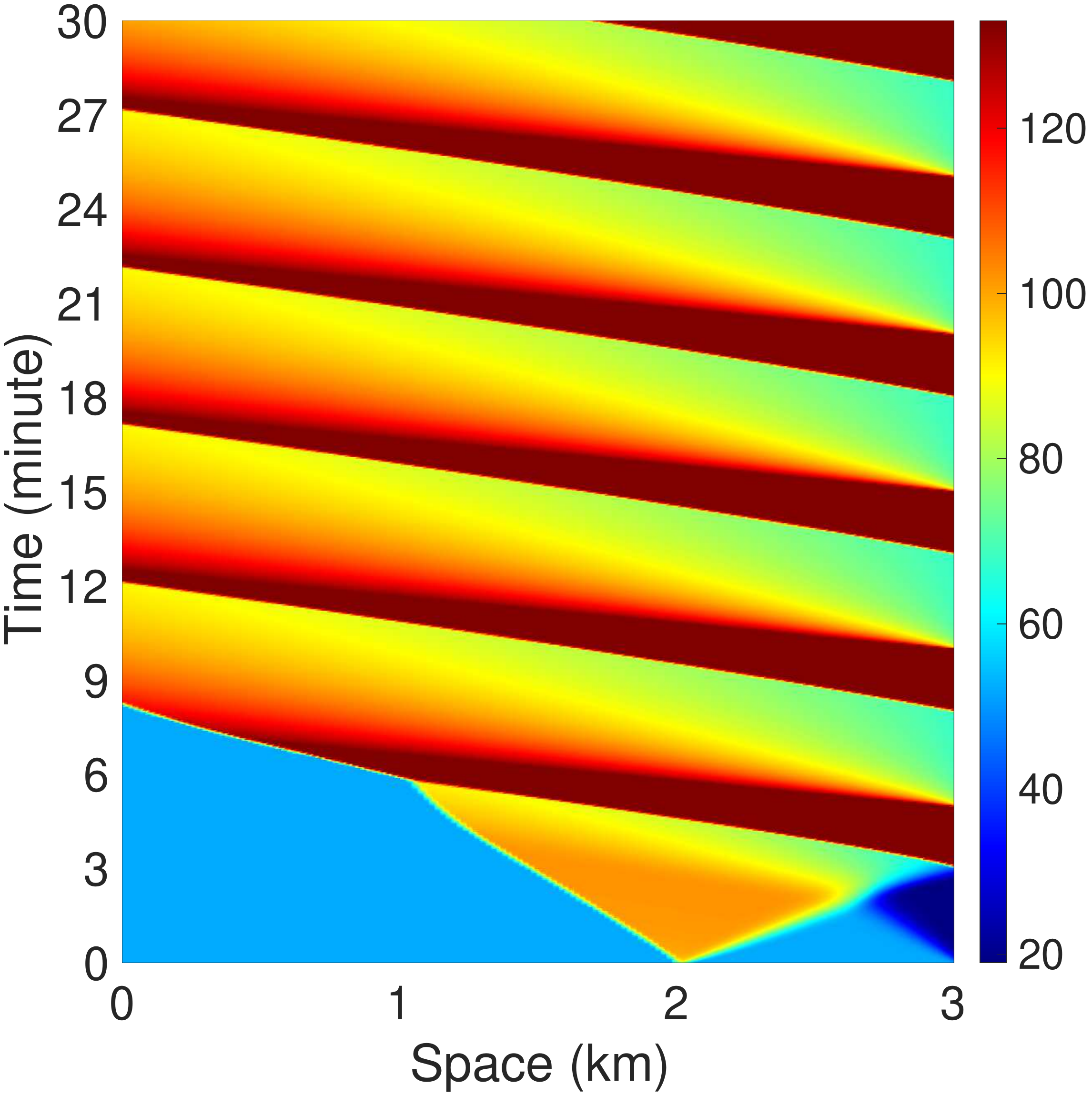}}
\quad
\subfloat[][Speed ($\km\per\myhour$).]{
\includegraphics[width=0.345\linewidth]{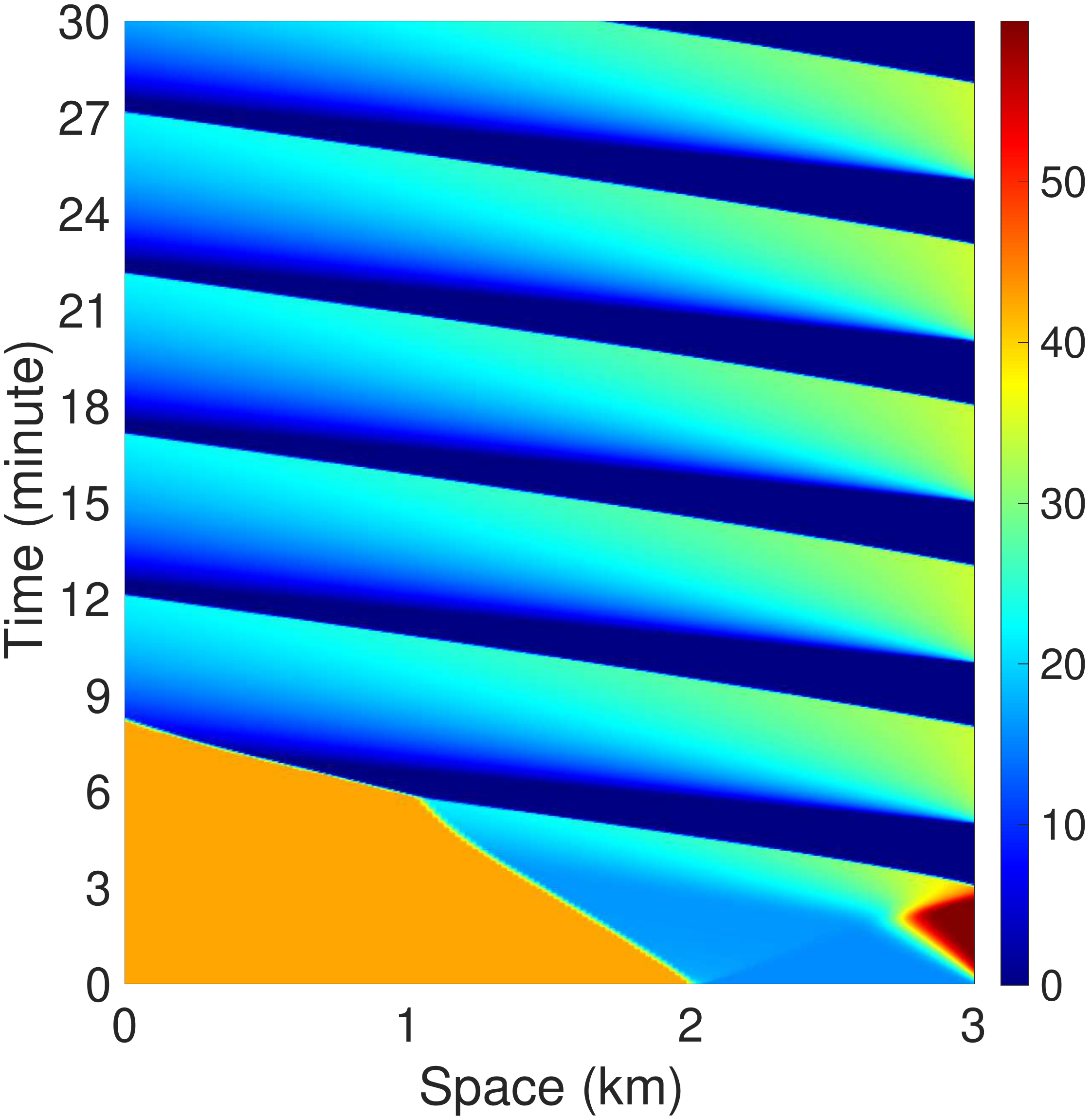}}
\\
\subfloat[][Acceleration ($\mymeter\per\mysecond^{2}$).]{
\includegraphics[width=0.35\linewidth]{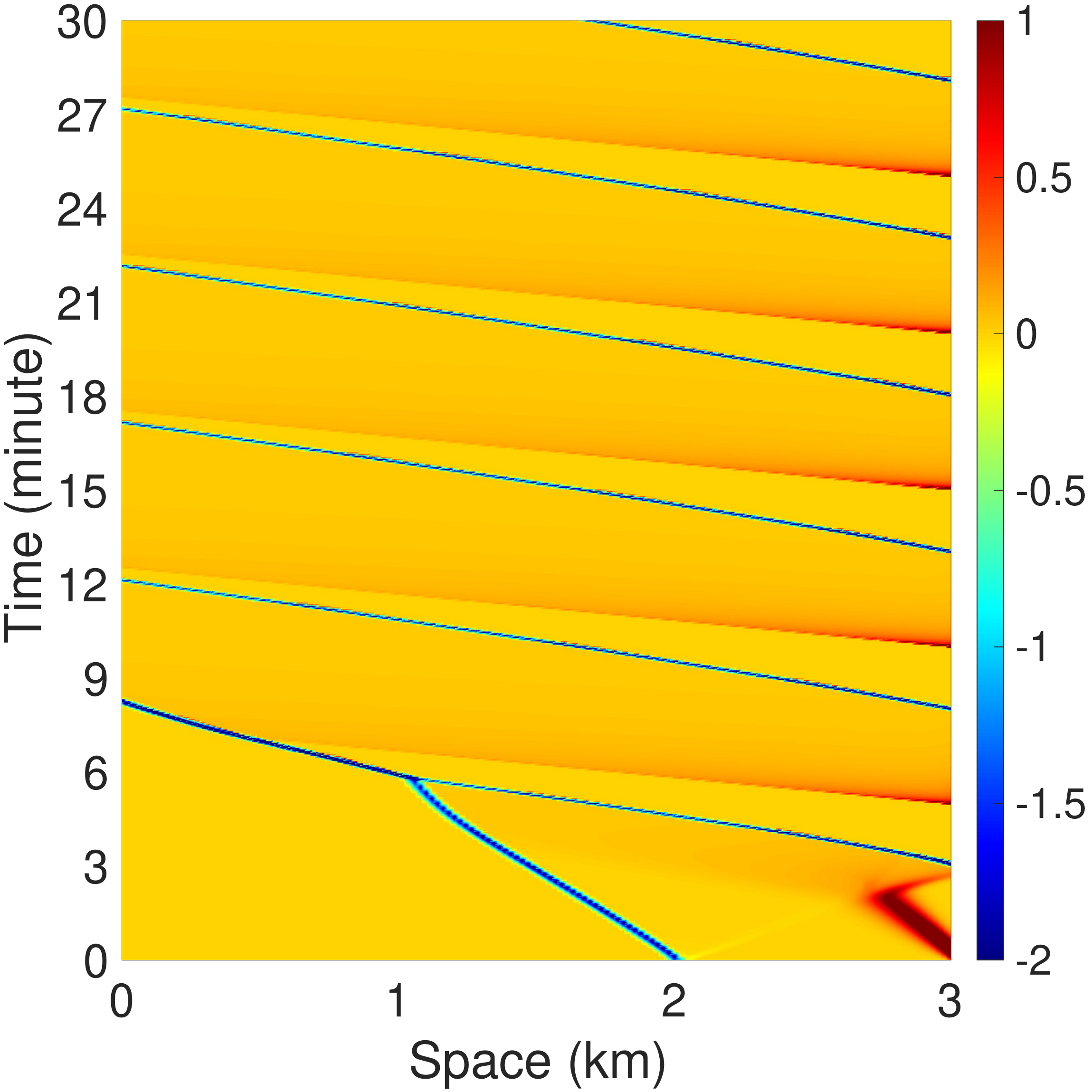}}
\quad
\subfloat[][Emission rates ($\mygram\per\myhour$).]{
\includegraphics[width=0.345\linewidth]{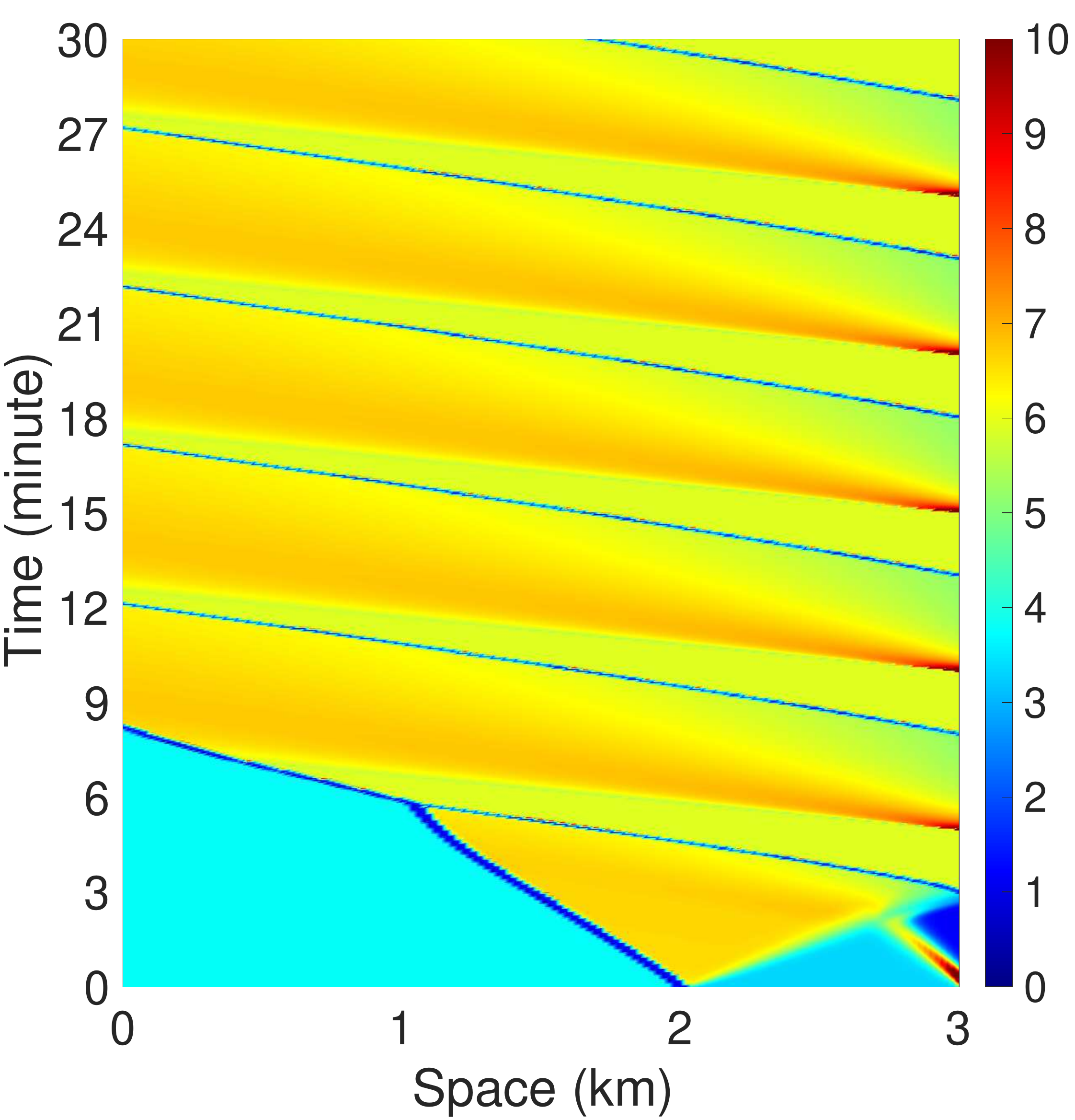}}
\caption{\ref{test2}: Variation of density (a), speed (b), analytical acceleration (c) and $\nox$ emissions (d) in space and time.}
\label{fig:test2_traffico}
\end{figure}

\begin{figure}[hpbt!]
\centering
\includegraphics[width=0.34\linewidth]{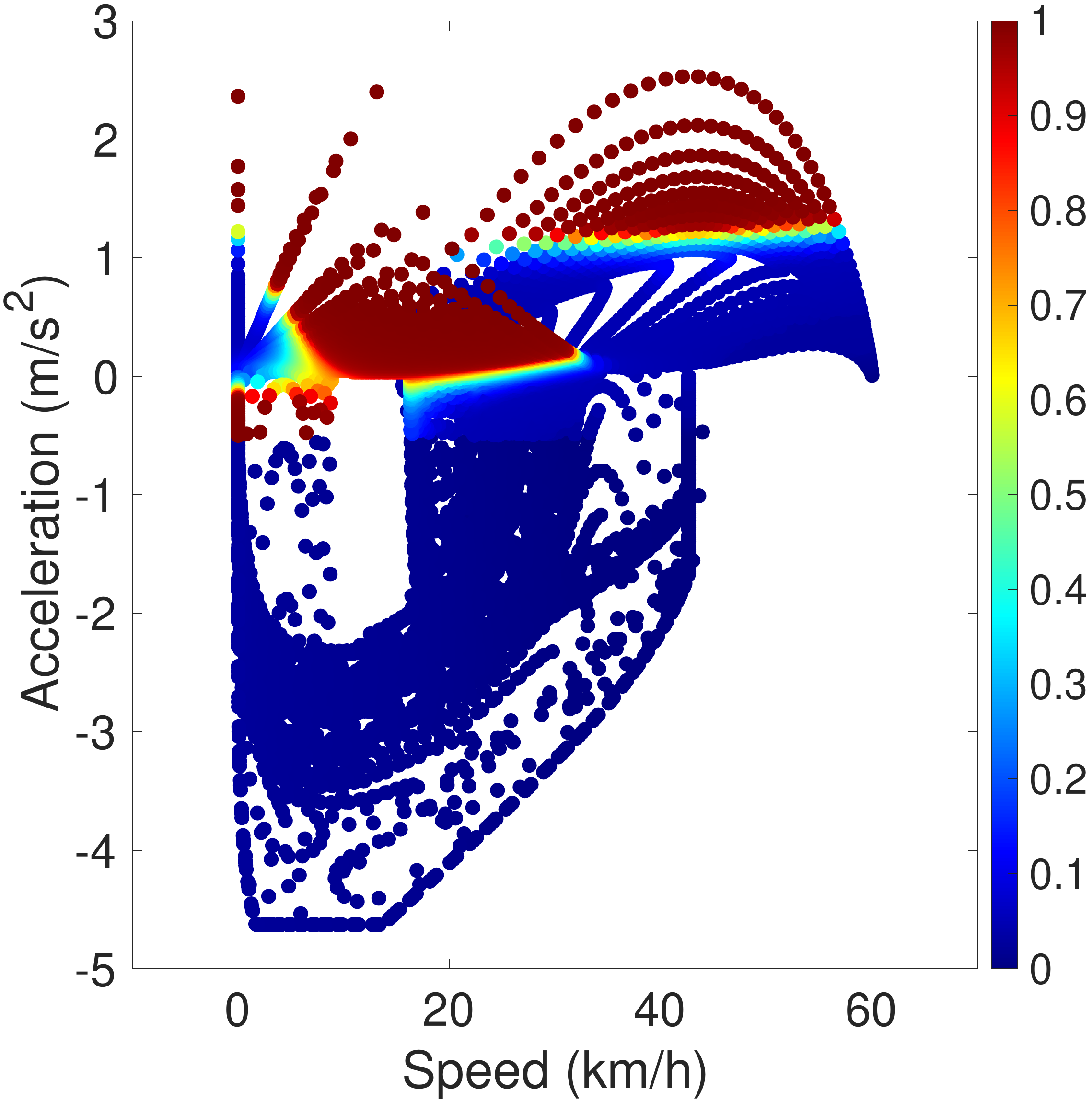}\quad
\includegraphics[width=0.35\linewidth]{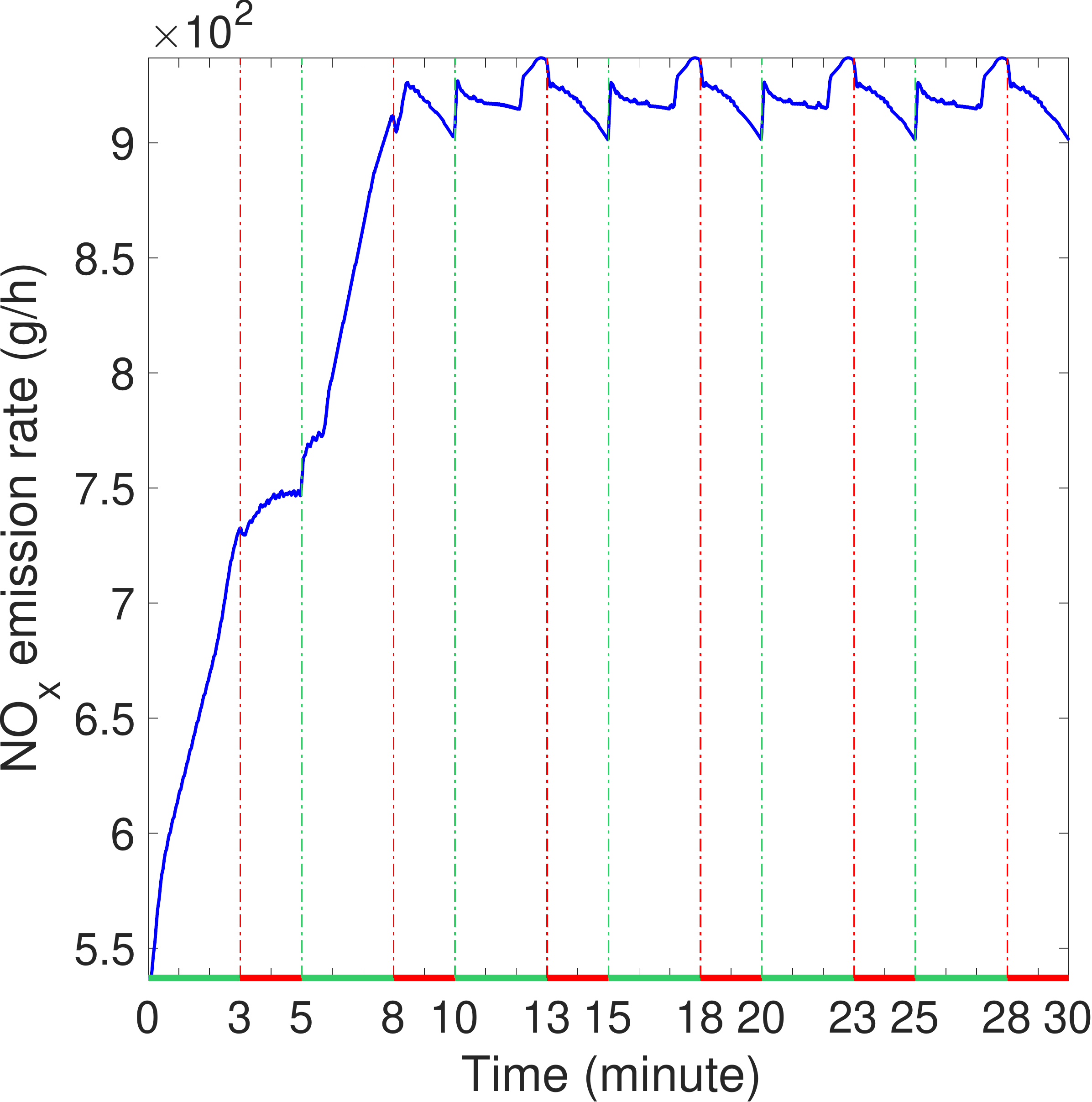}\\
\caption{\ref{test2}: $\nox$ emission rate ($\mygram\per\myhour$) as a function of speed and acceleration (left); variation in time of the total emission rate ($\mygram\per\myhour$) along the entire road (right).}
\label{fig:test2_emissioni}
\end{figure}

Let now $r=t_g/t_r$ be the ratio between the time $t_g$ of the green phase and the time $t_r$ of the red phase respectively, and let
$t_{c}$ be the time of the whole traffic light phase, i.e.\ $t_c=t_g+t_r$. We  consider two different test cases: first we fix the ratio $r$ and we vary the time $t_c$; then, conversely we fix $t_c$ and we vary $r$.

\paragraph{\labeltext{Traffic dynamic 2.1}{test2p1}: Emissions when the ratio $r$ is constant}
In Figure \ref{fig:semafori_rfisso} we show the $\nox$ emissions obtained with $r=3/2$ and three different values of traffic light duration in minutes: on the left we set
$t_c=\myunit{7.5}$ and $t_r=\myunit{3}$, in the center $t_c=\myunit{5}$ and $t_r=\myunit{2}$ and on the right $t_c=\myunit{2.5}$ and $t_r=\myunit{1}$. We observe that the maximum value of the $\nox$ emission rate increases when the frequency of vehicles restarts augments, namely when the time $t_r$ of the red phase is lower.
\begin{figure}[h!]
\centering
\includegraphics[width=0.3\linewidth]{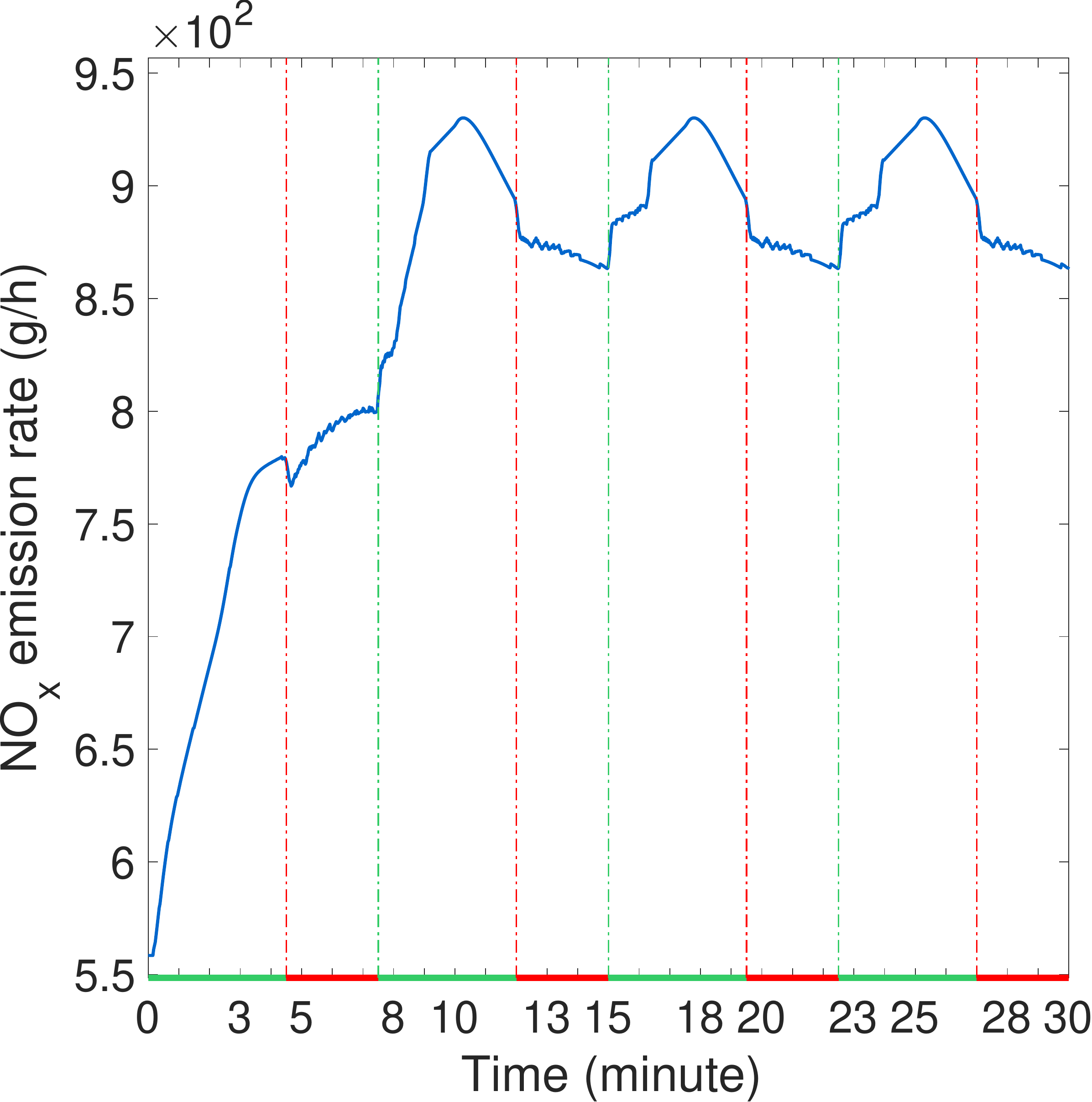}
\includegraphics[width=0.3\linewidth]{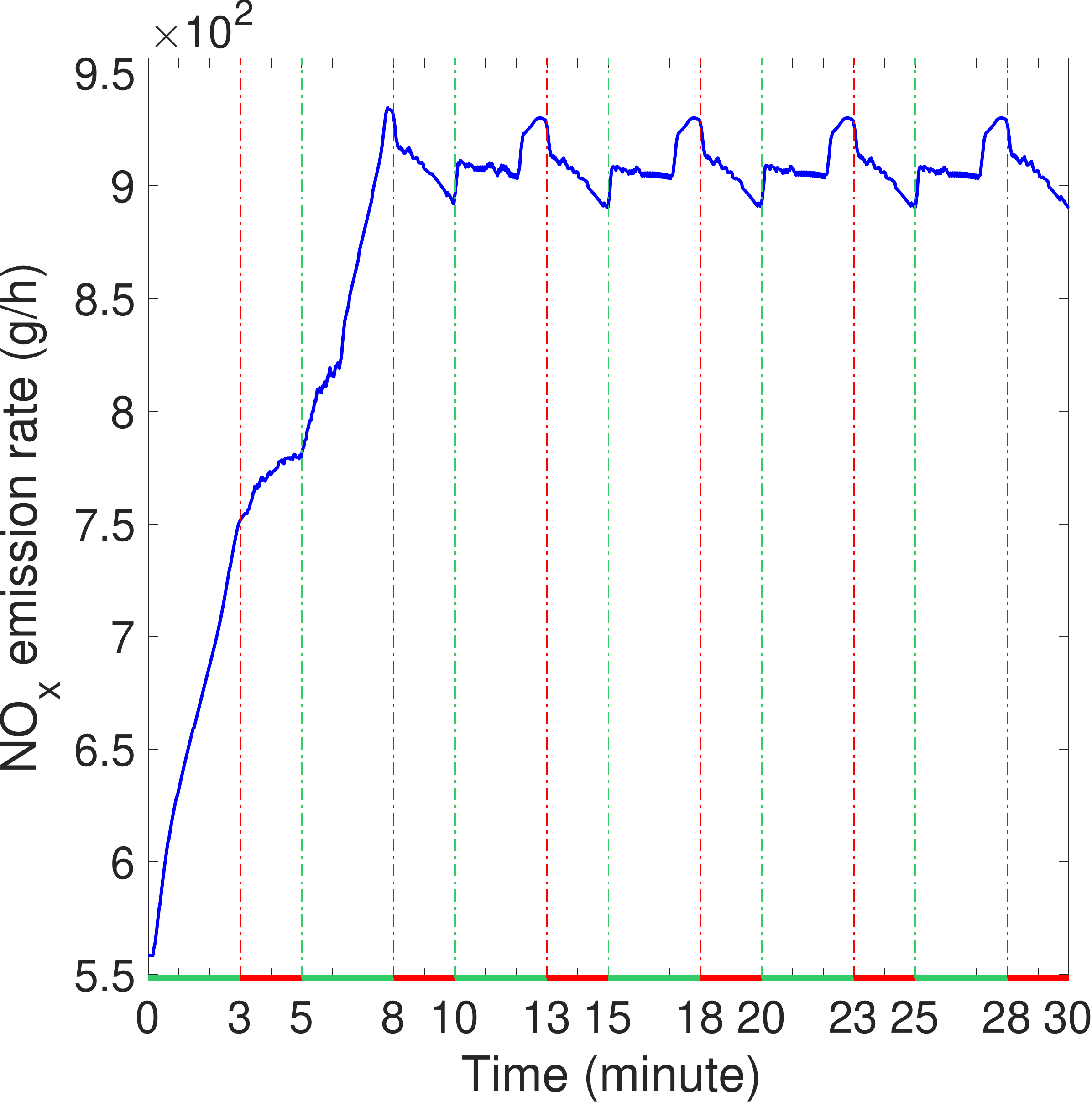}
\includegraphics[width=0.3\linewidth]{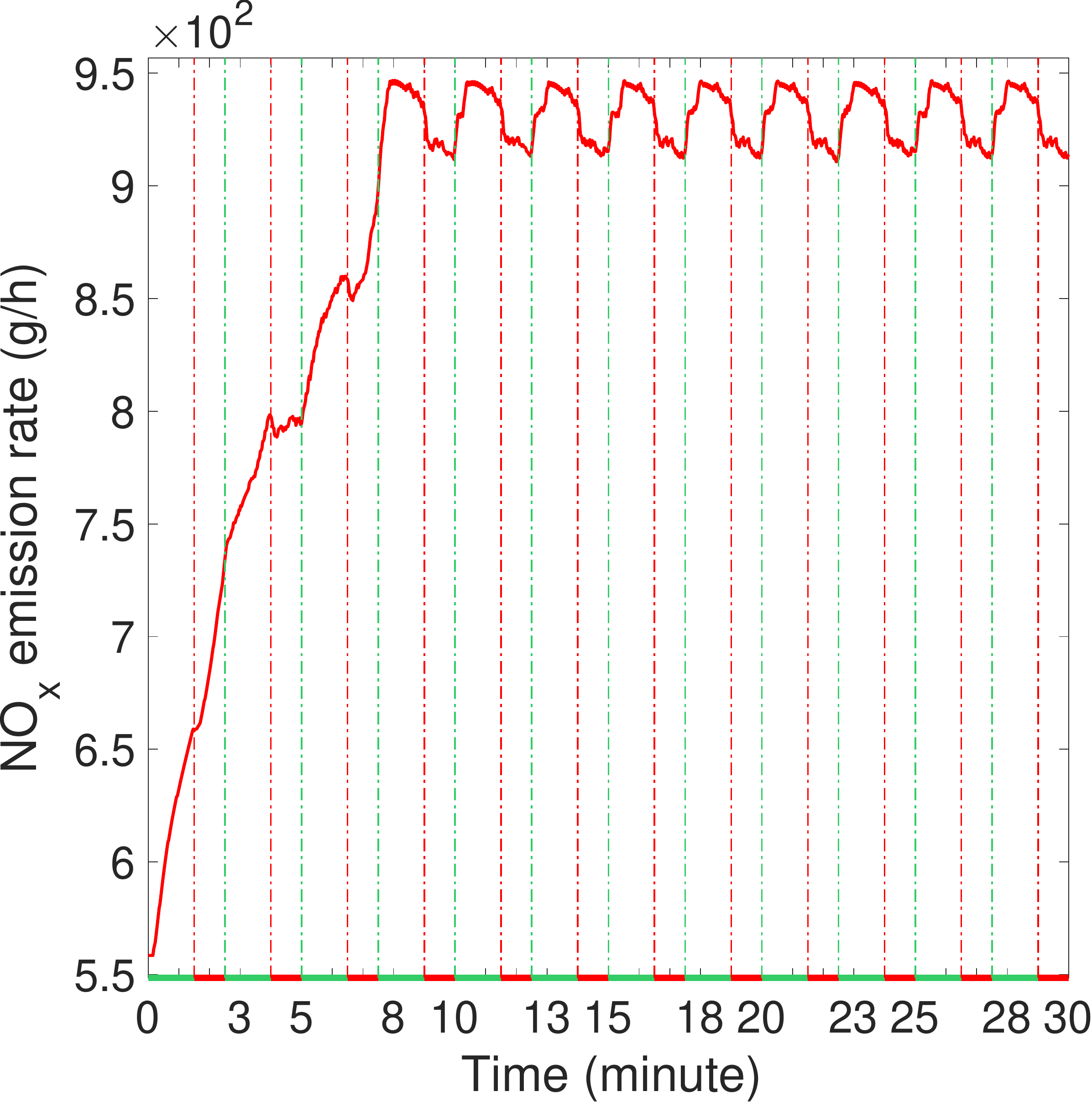}
\caption{\ref{test2p1}: Variation in time of the total $\nox$ emission rate ($\mygram\per\myhour$) along the entire road with $r=3/2$ and varying the traffic light duration $t_c$ in minutes:  $t_c=7.5$ with $t_r=3$ (left); $t_c=5$ with $t_r=2$ (center); $t_c=2.5$ with $t_r=1$ (right).}
\label{fig:semafori_rfisso}
\end{figure}
%
\paragraph{\labeltext{Traffic dynamic 2.2}{test2p2}: Emissions when the traffic light duration $t_c$ is constant}
In Figure \ref{fig:tc_fisso} we show how the 
$\nox$ emissions vary with different ratio $r$. 
Specifically, we plot $\nox$ total emissions (defined as the sum on the cells of the emission rates, at any time) for one hour of simulation with $r=\{4,\ 3/2,\ 2/3\}$ which is equivalent to assume $(t_g,t_r)=(\myunit{4},\myunit{1})$, $(t_g,t_r)=(\myunit{3},\myunit{2})$, $(t_g,t_r)=(\myunit{2},\myunit{3})$ in minutes, respectively.
We observe that until $t_r\leq t_g$ (solid line and line with circle) the maximum of the emission rate increases when $t_r$ grows, since the are more vehicle restarts; while it decreases with $t_r>t_g$ (line with stars) when there are less vehicles restarts and more phases of stopped traffic.  

\begin{figure}[h!]
\centering
\includegraphics[width=0.3\linewidth]{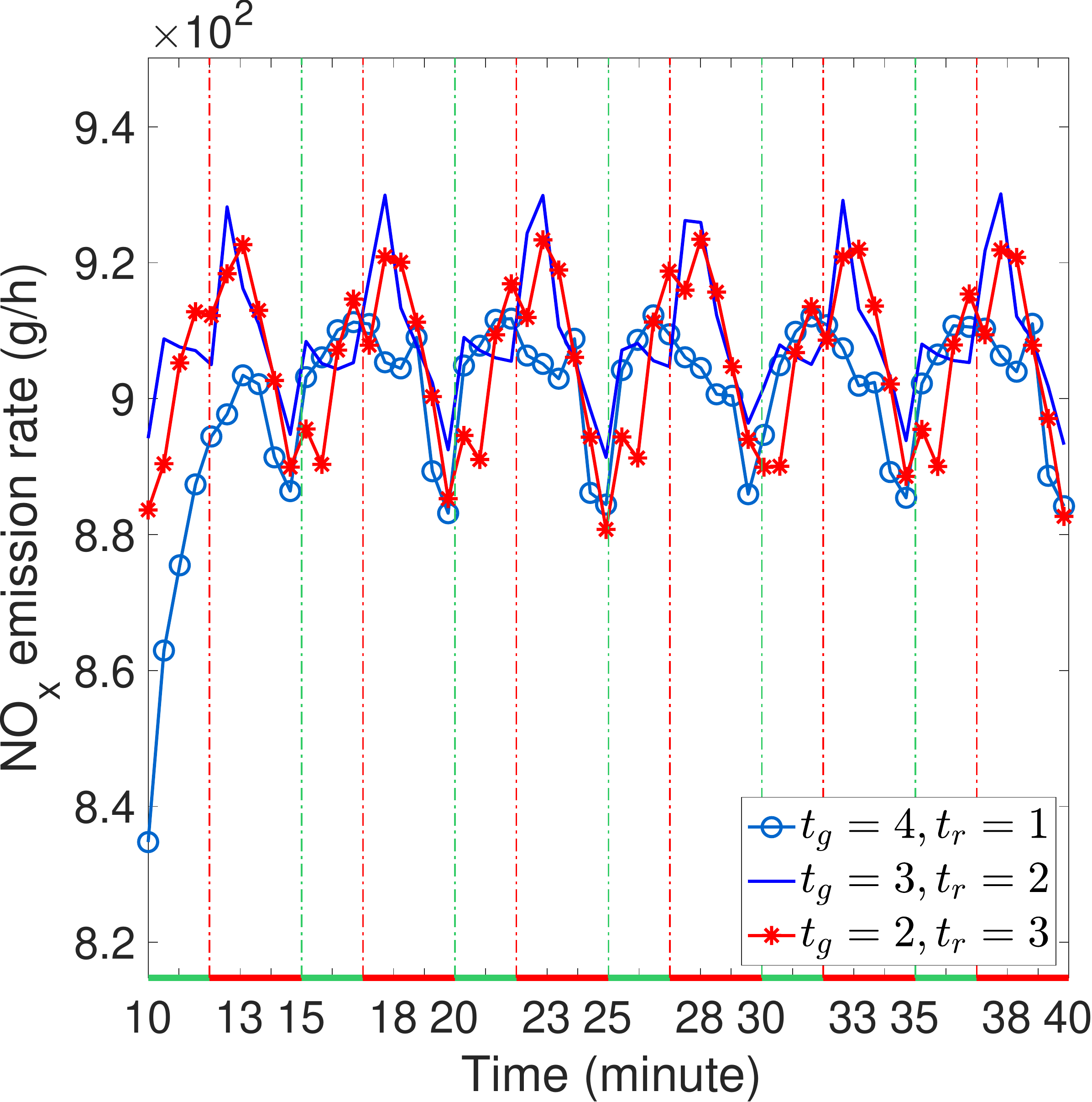}\quad
\caption{\ref{test2p2}: Variation in time of the total emission rate ($\mygram\per\myhour$) along the entire road varying the ratio $r$.}
\label{fig:tc_fisso}
\end{figure}

To sum up, the two last examples developed in \ref{test2p1} and \ref{test2p2}, suggest that the emissions grow with the increase of vehicles restarts. 
In particular, we observe from Figure \ref{fig:semafori_rfisso} that the length of the traffic light cycle $t_c$ has an highly influence on emissions, while Figure \ref{fig:tc_fisso} shows that the ratio $r$ between red-light and green-light has a less effect on the asymptotic emission values.

\subsubsection{Production of ozone and diffusion in air}\label{sec:test3}
In this section we are interested in estimating the concentration of ozone along the entire road by means of the system \eqref{eq:sistemone} and its diffusion in the atmosphere by systems \eqref{eq:diffVert} or \eqref{eq:diffOriz}. The reaction rate parameters $k_{1}$, $k_{2}$ and $k_{3}$ are listed in Table \ref{tab:parametriK}.

Let us begin with system \eqref{eq:sistemone}, which estimates the concentration of pollutants on the road. More precisely, given the numerical discretization of the road $[0,L]$ with cells of length $\dx$, we estimate the concentrations of the chemical species on each volume of dimension $\dx^{3}$ corresponding to the cells discretizing the road. For each $x_{i}$, $i=1,\ldots,N_x$, we set the initial concentrations $\CC^{0}_{i}$ as
$\cc_1(x_{i},0)=\cc_3(x_{i},0) = 0$, $\cc_2(x_{i},0)=5.02\times\,\myunit{10^{18}}\mathrm{molecule}/{\mycenti\mymeter^{3}}$
and, according to Section \ref{Air} and relation \eqref{eq:nox_source}, for $\no$ and $\nodue$ we have
\begin{equation*}
\cc_4(x_{i},0) = (1-p)\frac{E_\nox(0)}{\Delta x^3},\quad \cc_5(x_{i},0)=p\frac{E_\nox(0)}{\Delta x^3} \quad\text{ with } p=0.15.
\end{equation*}
Then, for each time step $n$ we compute the source term $s^{n}_{i}$ due to the two traffic dynamics and we apply an ODE solver for \eqref{eq:sistemone} to estimate $\CC^{n}_{i}$. To chose the right solver, we analyze the stiffness of the problem, see e.g. \cite[Chapter 6]{lambert1991JW}, without source term. Therefore, we consider the linearization of $G$ given in \eqref{eq:sistemaVett}, in a neighbourhood of the initial data $\CC^{0}=\CC(x,0)$.
The eigenvalues of the Jacobian of $G$ range in a large interval of values, due to the order of magnitude of chemical species and reaction coefficients $k_{1}$, $k_{2}$ and $k_{3}$ (see Table \ref{tab:parametriK}). In particular, we have $\lambda_{1}$ with order of magnitude $10^{7}$, $\lambda_{2}$ of $10^{1}$ and $\lambda_{3}=\lambda_{4}=\lambda_{5}=0$. A similar result is obtained adding the source term $s(x,t)\neq0$.
Hence, the problem is stiff and we need to approximate system \eqref{eq:sistemaVett} with an adaptive step size method. To this end we solve \eqref{eq:sistemaVett} using the standard Matlab tool $\mathtt{ode23s}$, which makes use of modified Rosenbrock formula of order 2 and works with an adaptive step size.

To compare the results obtained with \ref{test1} and \ref{test2} we compute the total concentration of all the chemical species along the entire road, i.e.\ for every time $t^{n}$ we sum the concentrations on all the cells.
In Figure \ref{fig:test3_reazioni} we show the variation in time of the concentration of $\otre$ and $\odue$. We observe that the ozone concentration has a huge growth (\ref{test1} - blue-solid line), which is further amplified by the presence of the traffic light (\ref{test2} - red-circles line).
On the other hand, the oxygen concentration is almost constant in both the cases,  with moderated dependence on traffic light. 

To further investigate the impact of the traffic light on all the chemical species concentration, we solve our system starting from the $\nox$ emission rates computed in \ref{test2p1} in which we fix the ratio $r$ constant. Thus, we compute the total amount of $\otre$, $\no$, $\nodue$ and $\ouno$, obtained during the whole simulation along the entire road. Then, we measure the variation of each concentration with respect to the one obtained in the test case without traffic light \ref{test1}.
The results in Table \ref{tab:r_fisso} show that all the concentrations increase coherently with the behavior of the $\nox$ source term, see Figure \ref{fig:semafori_rfisso}. So, we can conclude that the duration of traffic cycles
affects all the chemical species production more than the ratio between green and
red phase.


\begin{figure}[h!]
\centering
\includegraphics[width=0.35\linewidth]{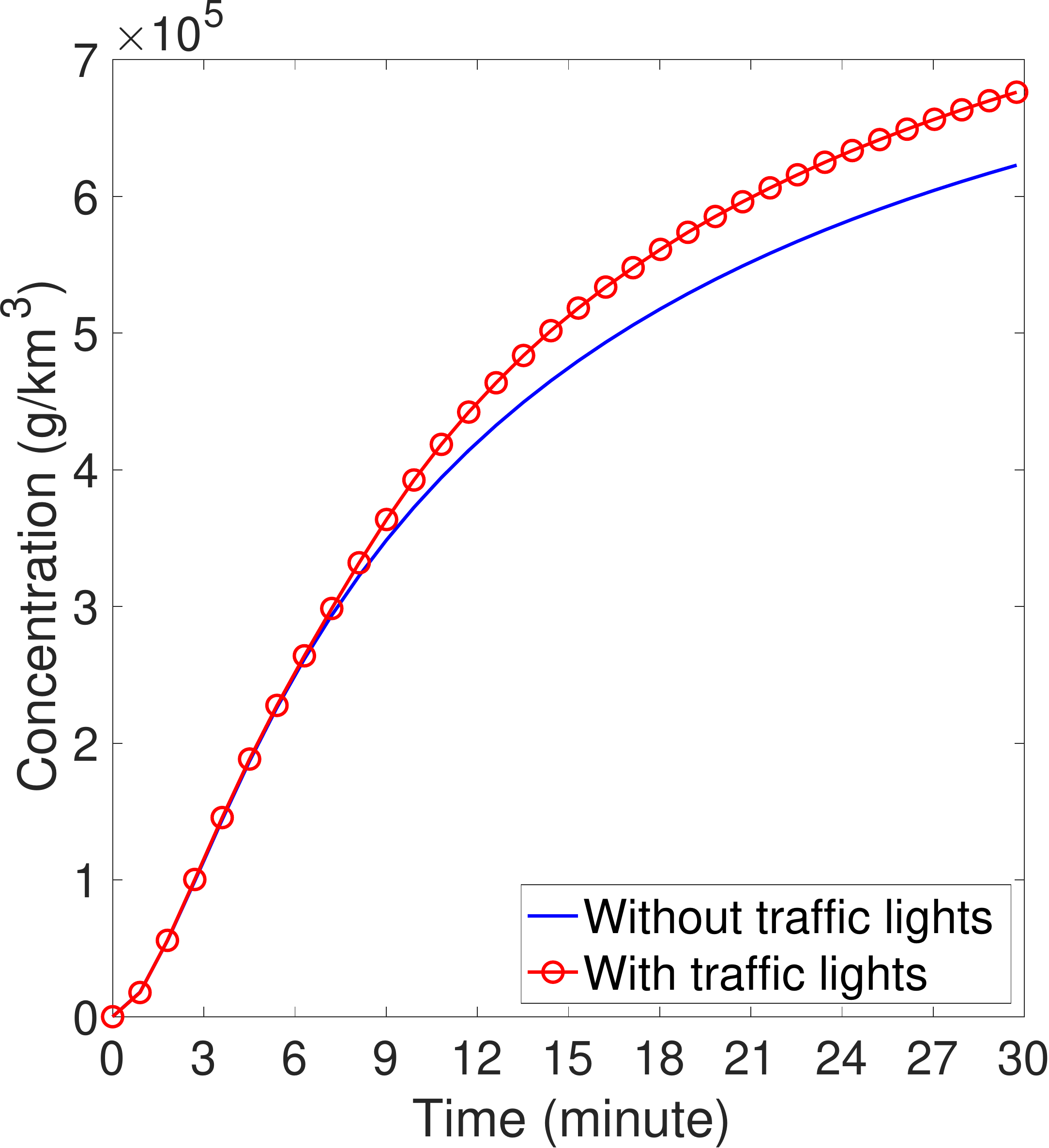}\quad
\includegraphics[width=0.36\linewidth]{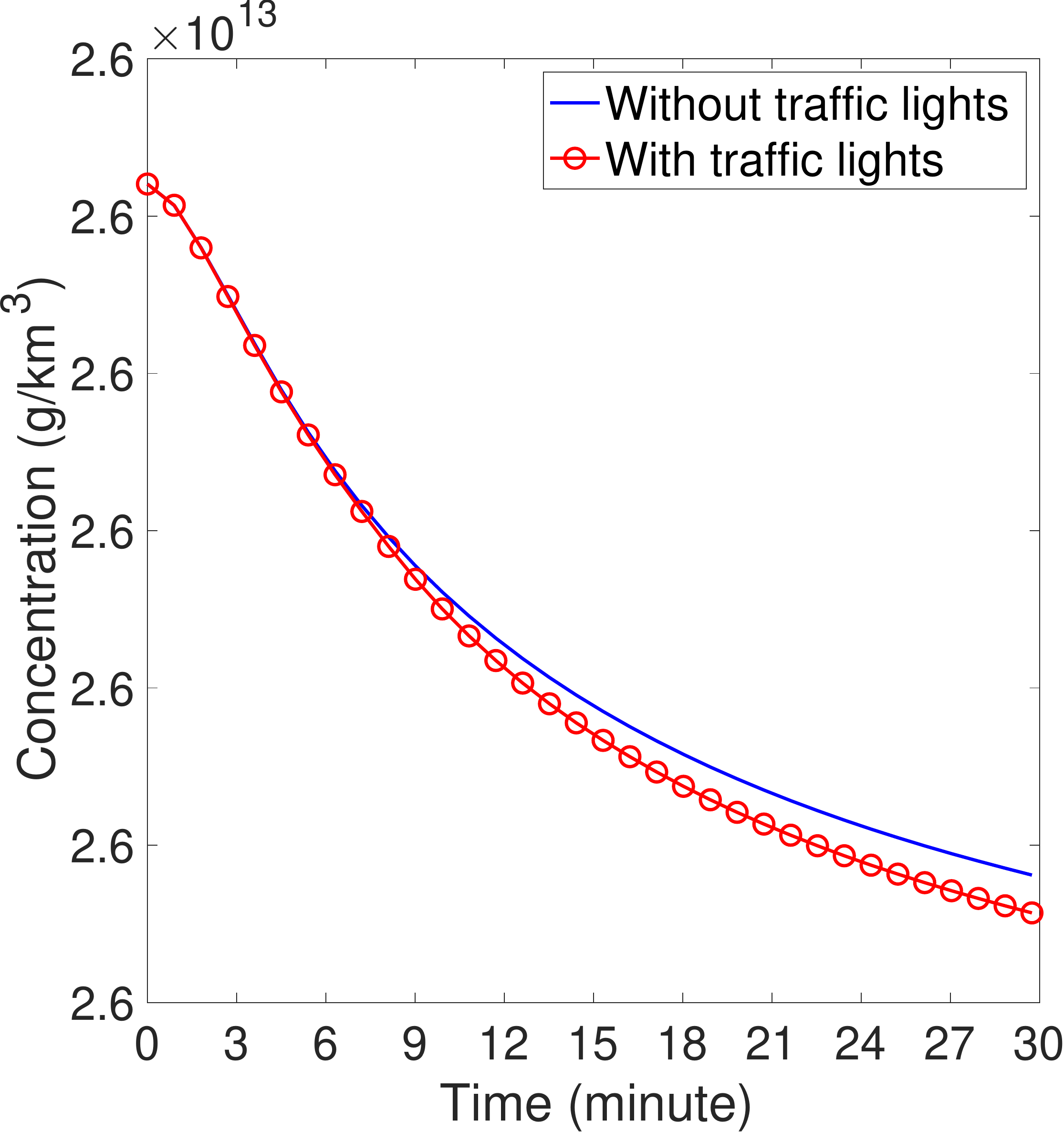}
\caption{Variation in time of the total concentration ($\mygram/\mykilo\mymeter^3$) of $\otre$ (left) and $\odue$ (right), in the case of dynamics with (red-circles) and without (blue-solid) traffic light.}
\label{fig:test3_reazioni}
\end{figure}

\begin{table}[h!]
\centering
\footnotesize
\renewcommand{\arraystretch}{1.1}
\begin{tabular}{|c|c|c|c|}\hline
$t_c=t_r+t_g$ & $\myunit{(3+4.5)}\mymin$ &  $\myunit{(2+3)}\mymin$ & $\myunit{(1+1.5)}\mymin$ \\\hline
$\otre$ & 2.95e+07 & 3.54e+07 & 3.91e+07 \\\hline
$\no$ & 1.09e+09 & 1.28e+09 & 1.43e+09 \\\hline
$\nodue$ & 1.55e+08 & 1.81e+08 & 2.02e+08 \\\hline
$\ouno$ & 7.00e+01 & 8.21e+01 & 9.13e+01 \\\hline
\end{tabular}
\caption{Variation of the total amount of $\otre$, $\no$, $\nodue$ and $\ouno$ concentration ($\mygram/\mykilo\mymeter^3$) computed with three different traffic light duration (\ref{test2p1}) with respect the total amount of concentrations without traffic light (\ref{test1}).
}
\label{tab:r_fisso}
\end{table}

\medskip
We now consider the diffusion of pollutants in the air with the two different approaches proposed in Section \ref{sec:diff}. In light of the results obtained in the previous test, we assume that the oxygen remains constant during the simulation.
We use finite differences to discretize systems \eqref{eq:diffVert} and \eqref{eq:diffOriz}. Specifically, an Euler implicit approximation in time and in space a centered difference for the diffusion term coupled with the upwinding of the first order term in \eqref{eq:diffOriz}. 

\paragraph{Vertical diffusion}
To analyze the vertical diffusion of the chemical species, we consider the domain $\Omega$ of \eqref{eq:diffVert} with $L=500\,\mymeter$ and $H=0.5\,\mymeter$, discretized via a numerical grid with $\dx=5\,\mymeter$ and $\dy=0.02\,\mymeter$. In order to define the concentration of pollutants per unit of volume, we assume that $\CC$ is constant along the third component $\dz$, which is fixed equal to $\dy$. 
We set $\CC_{0}\equiv 0$ and, following \cite{alvarez2017JCAM, Sportisse2010}, we fix $\mu=10^{-8} \,\km^{2}\per\myhour$. 
We now consider the traffic dynamics described in \ref{test1} and \ref{test2} during a time interval $[0,T]$ with $T =4\,\myhour$. We assume that the source of emissions corresponding to the vehicle exhaust pipe is placed at half a meter in height, so as to be able to analyze the concentrations of pollutants at $1\,\mymeter$ from the ground. The source of pollutants is used as Dirichlet boundary condition for $\no$ and $\nodue$ in \eqref{eq:BC1} and \eqref{eq:BC2}, i.e.\ for $n=1,\dots,N_{t}$ and $i=1,\dots,N_{x}$ we have
\begin{align*}
	(\cc_{4})^{n}_{i1} &= (1-p)s^{n}_{i}\dt\\
	(\cc_{5})^{n}_{i1} &= p\,s^{n}_{i}\dt
\end{align*}
with $s^{n}_{i}$ the source term due to traffic emissions. In Figure \ref{fig:diffVert1} we compare the concentration of ozone in $\Omega$ at different times obtained from the two different traffic dynamics. From the plots it is clear that the presence of traffic lights highly increases the diffusion of ozone in the atmosphere, resulting in a higher concentration in all the domain $\Omega$. We recall that these results are obtained starting from zero pollutant concentrations, i.e.\ $\CC_{0}=0$.

\begin{figure}[h!]
\centering
\subfloat[][No traffic light $t=1\,\myhour$.]{
\includegraphics[width=0.3\columnwidth]{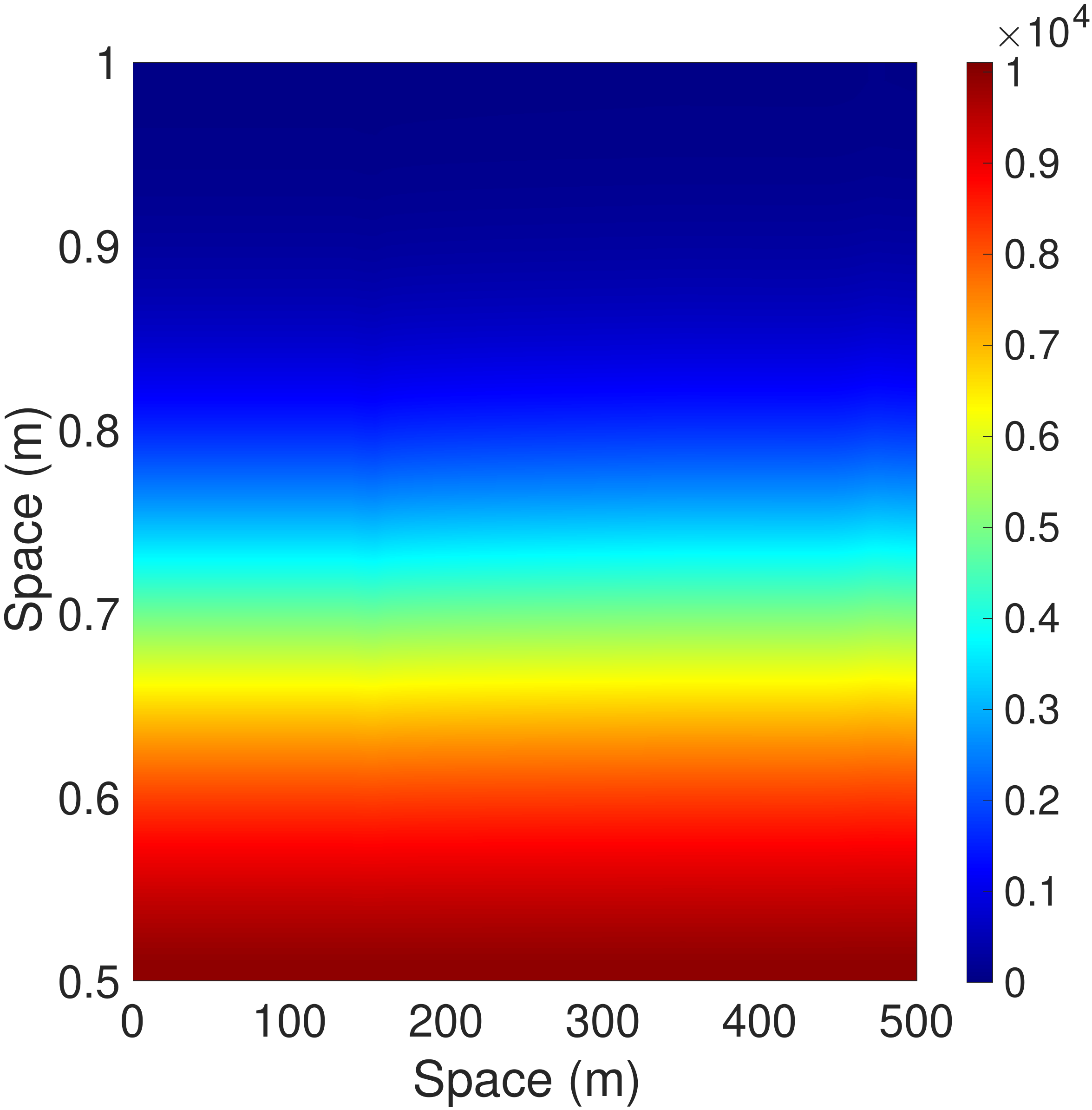}
}\quad
\subfloat[][No traffic light $t=T/2$.]{
\includegraphics[width=0.3\columnwidth]{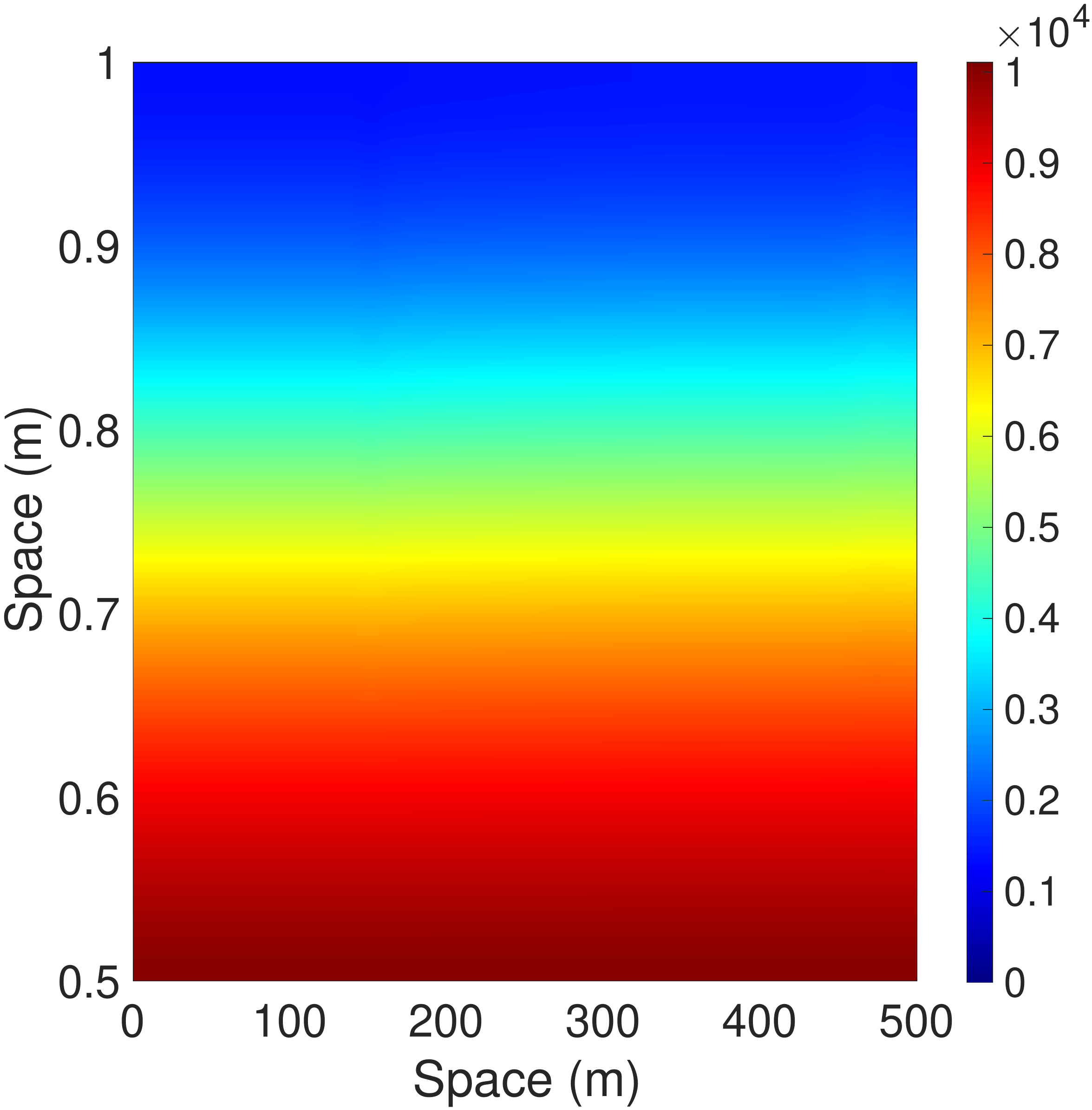}}
\quad
\subfloat[][No traffic light $t=T$.]{
\includegraphics[width=0.3\columnwidth]{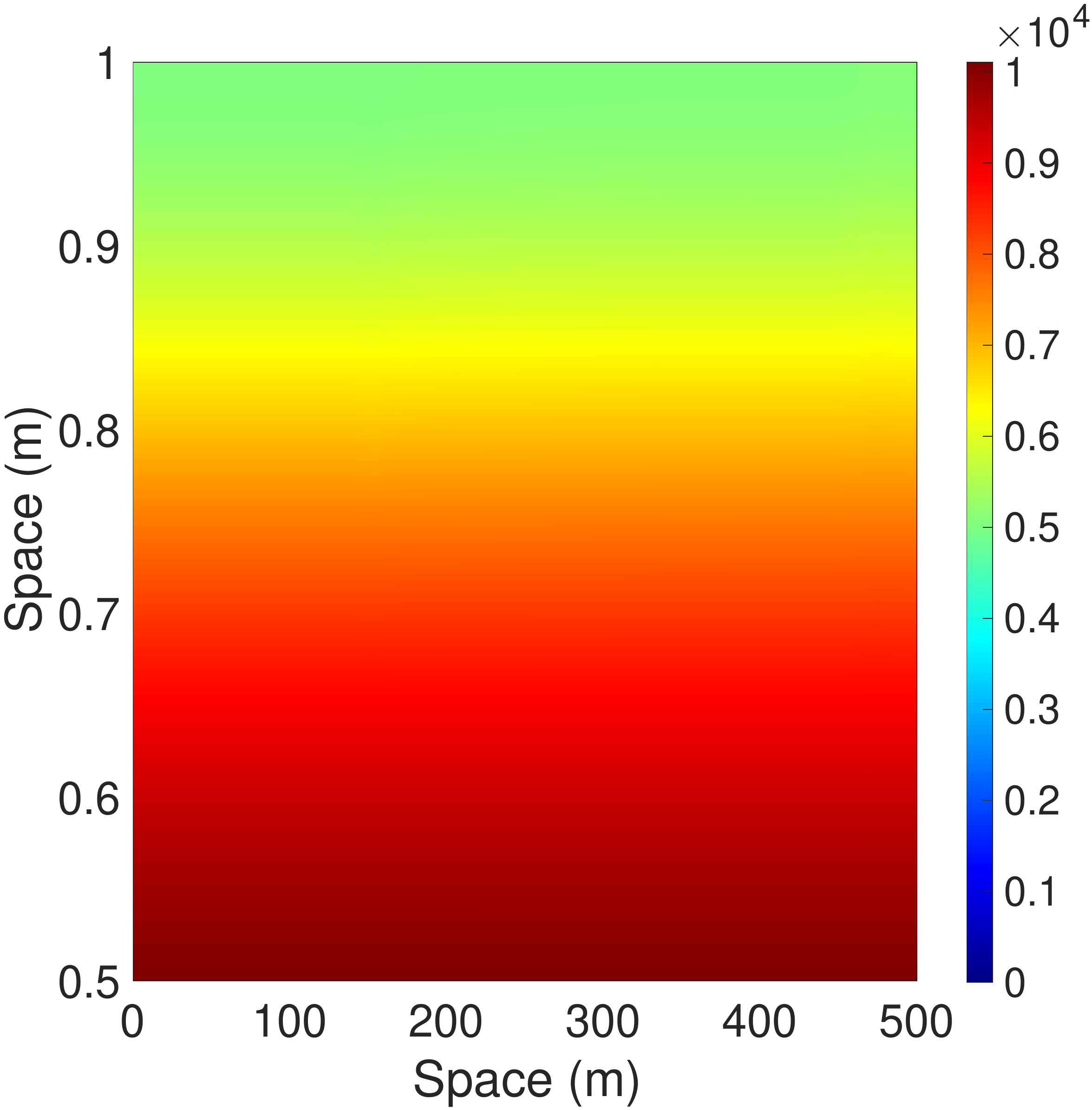}}
\\
\subfloat[][Traffic lights $t=1\,\myhour$.]{
\includegraphics[width=0.3\columnwidth]{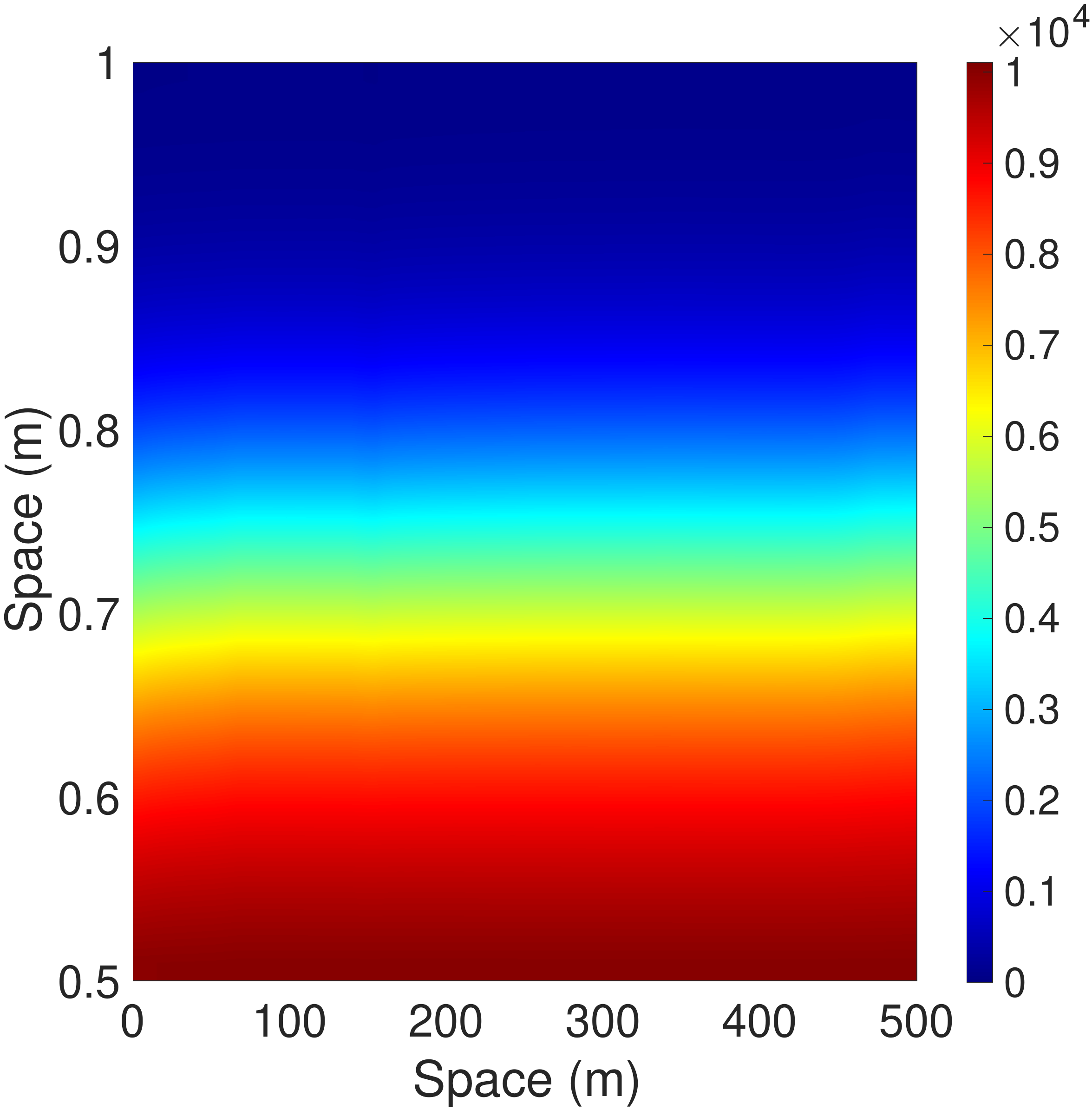}}\quad
\subfloat[][Traffic lights $t=T/2$.]{
\includegraphics[width=0.3\columnwidth]{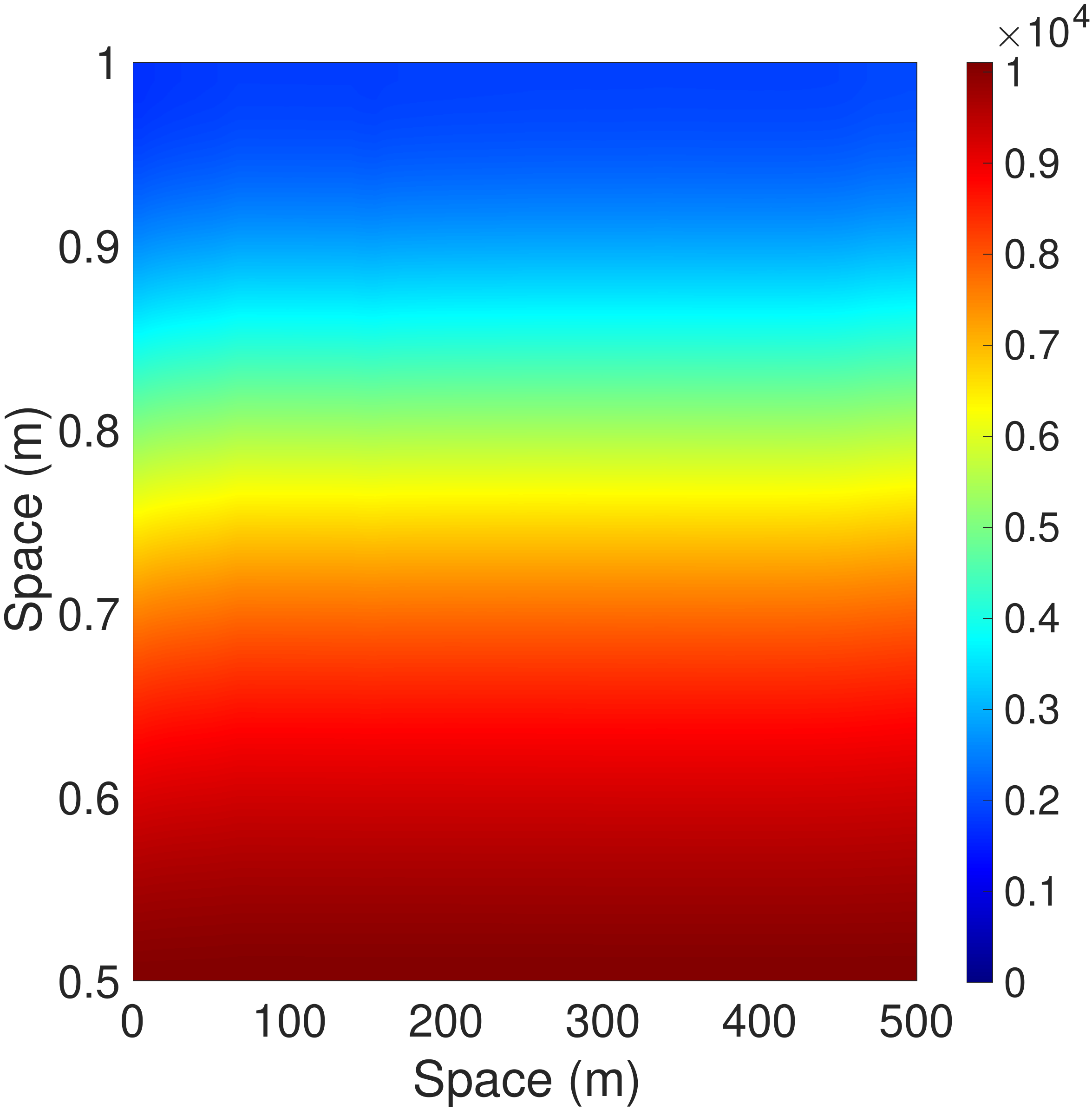}}\quad
\subfloat[][Traffic lights $t=T$.]{
\includegraphics[width=0.3\columnwidth]{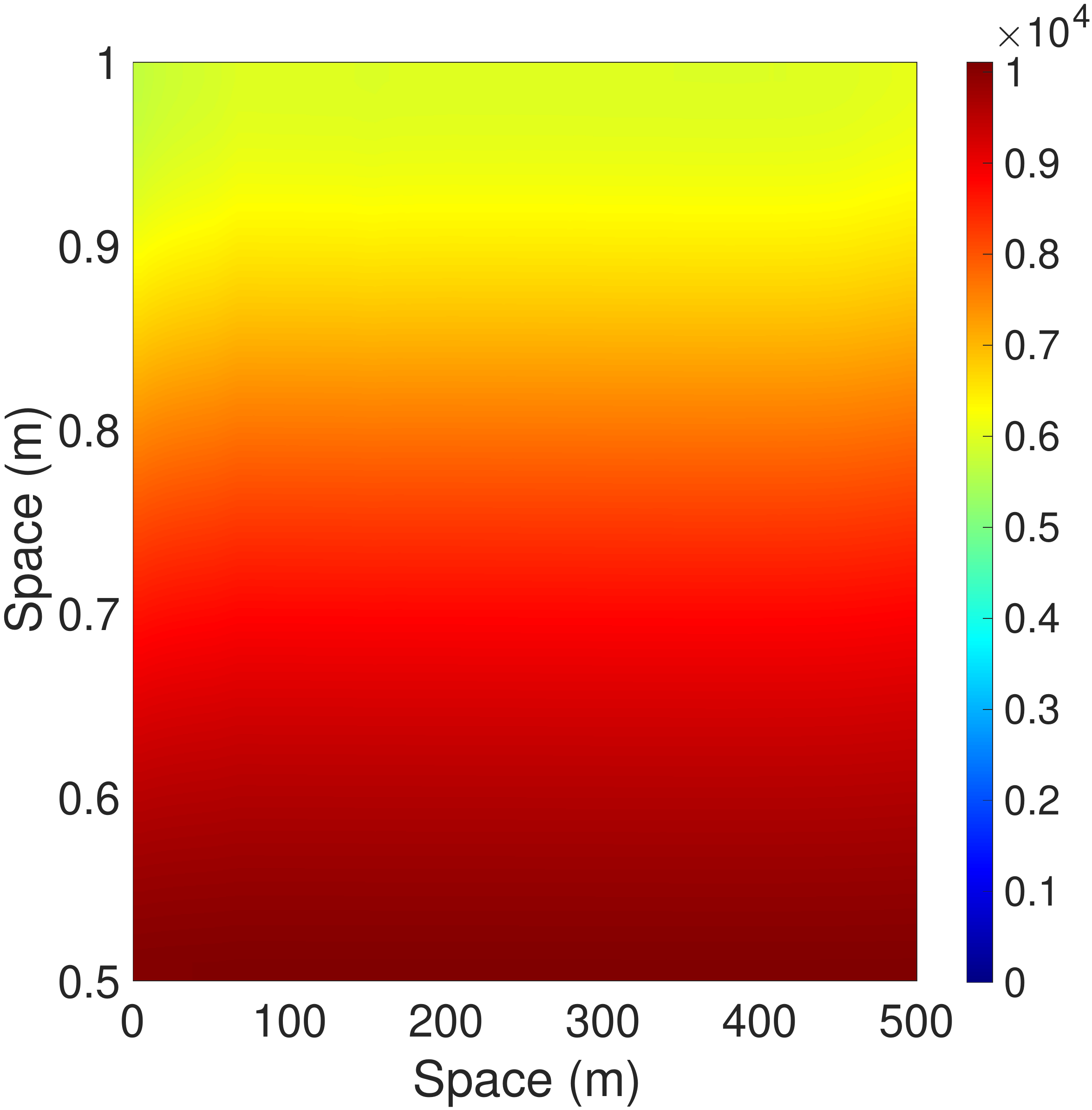}}
\caption{Vertical diffusion of ozone concentration ($\mygram\per\km^{3}$) in $\Omega$ at different times with (bottom) and without (top) traffic lights.}
\label{fig:diffVert1}
\end{figure}

In Figure \ref{fig:diffVert2} we show the variation in time of the concentration of ozone at the fixed height of $1\,\mymeter$ from the ground. Again, we see higher level of ozone in presence of traffic lights with respect to the case of no traffic lights on the road. More precisely, we are interested in estimating the increment produced with and without traffic lights. Indeed, denoting by $\cc_{3}^{1}$ and $\cc_{3}^{2}$ the concentration of ozone obtained from \ref{test1} and \ref{test2}, respectively, we compute the average concentration at the the final time at $1\,\mymeter$ from the ground, i.e.\ for $j=N_{y}$ and $n=N_{t}$ we evaluate
\begin{equation*}
	M^{1}=\frac{1}{N_{x}}\sum_{i=1}^{N_{x}}(\cc_{3}^{1})^{N_{t}}_{iN_{y}} \qquad\text{and}\qquad
	M^{2}=\frac{1}{N_{x}}\sum_{i=1}^{N_{x}}(\cc_{3}^{2})^{N_{t}}_{iN_{y}}.
\end{equation*}
From our test we obtain $M^{1}=5.06\times10^{3}\,\mygram\per\km^{3}$ and $M^{2} = 5.93\times 10^{3}\,\mygram\per\km^{3}$, therefore the presence of traffic lights causes a 18\% increase in the final average concentration of ozone at $1\,\mymeter$ from the ground compared to the absence of a traffic light. 

\begin{figure}[h!]
\centering
\includegraphics[width=0.3\columnwidth]{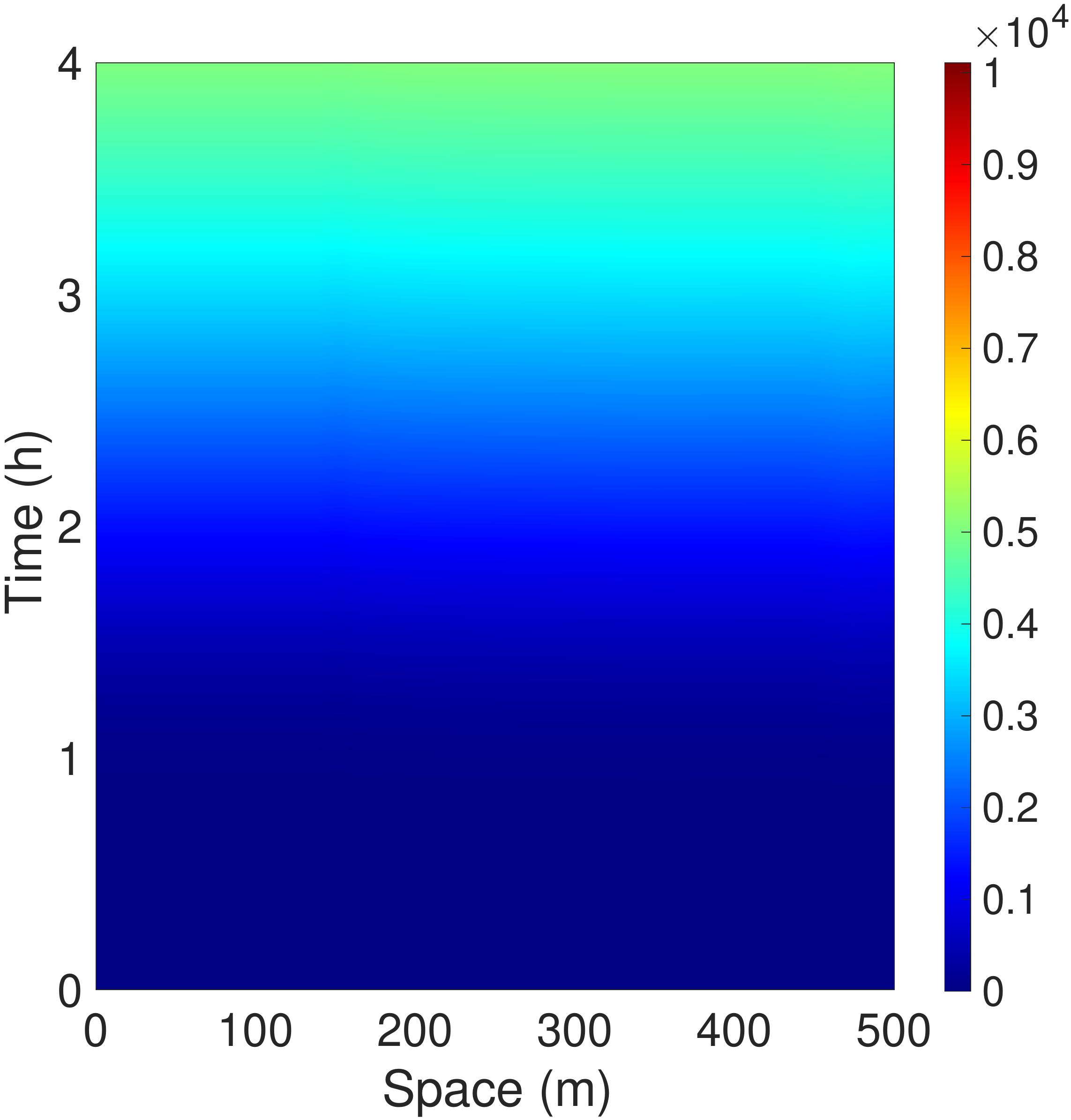}\quad
\includegraphics[width=0.3\columnwidth]{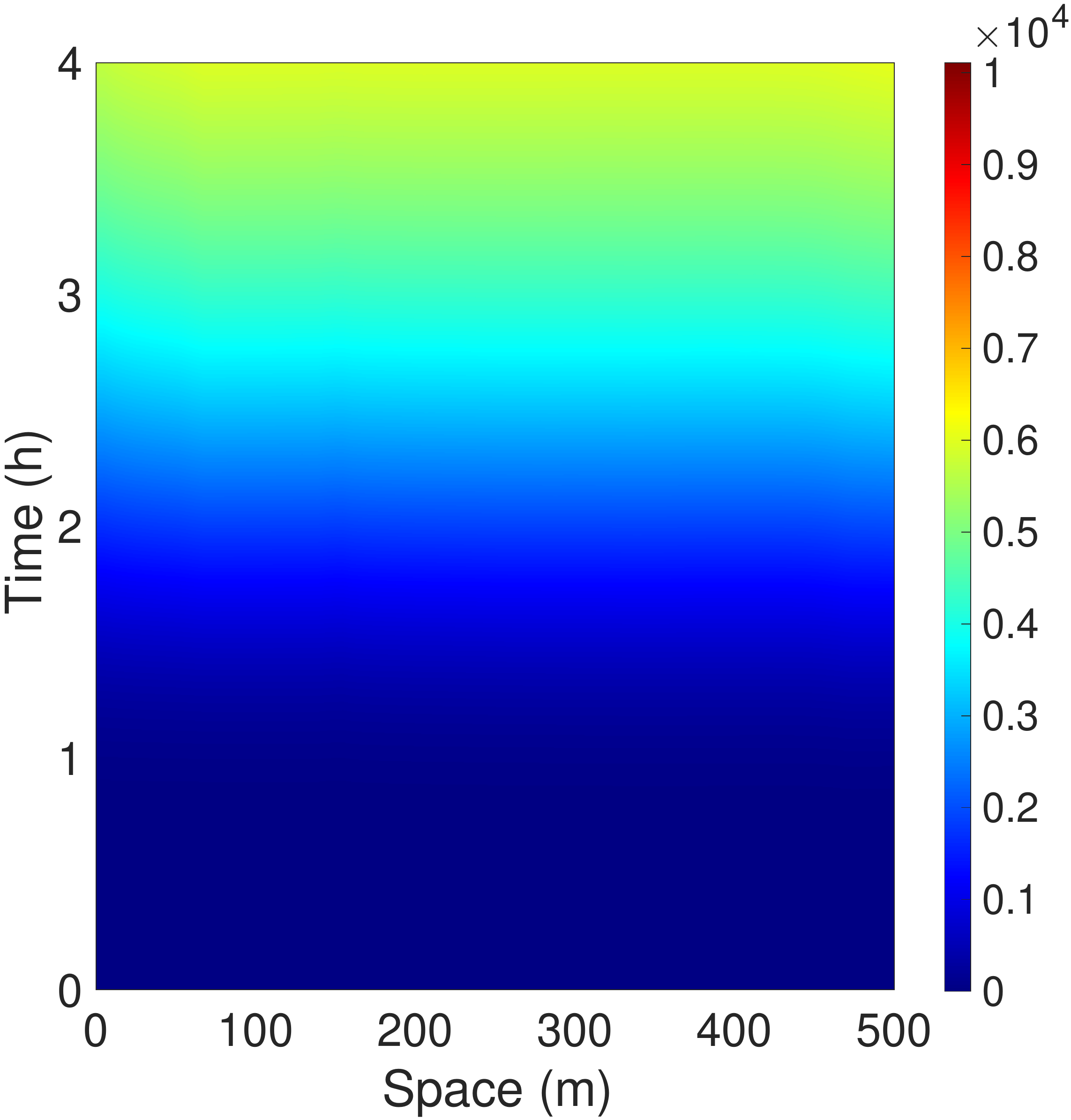}
\caption{Diffusion of ozone concentration ($\mygram\per\km^{3}$) in time at $1\,\mymeter$ height with (right) and without (left) traffic lights.}
\label{fig:diffVert2}
\end{figure}

\paragraph{Horizontal diffusion}
To analyze the horizontal diffusion of pollutants, we consider the domain $\Omega$ of \eqref{eq:diffOriz} with $\lx=\ly=500\,\mymeter$, discretized via a numerical grid with $\dx=\dy=5\,\mymeter$. We assume again that $\CC$ is constant along the third direction of length $\dz=\dy$, in order to consider the concentration per unit of volume.
We fix $\CC_{0}\equiv 0$, $\mu=10^{-8} \,\km^{2}\per\myhour$, $\ventox=-1\,\km\per\myhour^{2}$ and $\ventoy=0.2\,\km\per\myhour^{2}$.  
We consider again the traffic dynamics described in \ref{test1} and \ref{test2} during a time interval $[0,T]$ with $T =30\,\mymin$. The contribution of the traffic dynamics is used as a source term for $\no$ and $\nodue$ in the middle of the domain $\Omega$, in correspondence of the road, i.e.\ for each $n=1,\dots,N_{t}$ and $i=1,\dots,N_{x}$ we have
\begin{equation*}
\begin{array}{lr}
S^{n}_{ij} = (0,0,0,0,0) & \text{for $j\neq N_{y}/2$}\\
S^{n}_{ij} = (0,0,0,(1-p)s^{n}_{i},p\,s^{n}_{i}) & \text{for $j= N_{y}/2$}
\end{array}
\end{equation*}
with $s^{n}_{i}$ defined in \eqref{eq:nox_source}. In Figure \ref{fig:diffOriz} we compare the concentration of ozone in $\Omega$ at different times obtained from the two different traffic dynamics. The road is represented by the black line in the middle. We stress that the dynamics shown in the plots take into account zero pollutant concentrations at the beginning of the simulation process. The wind is responsible for the asymmetrical concentration of ozone in $\Omega$, which is higher in the upper part of the plots. The presence of traffic lights causes an increase of ozone production, which is further diffused by the wind. Similarly to the vertical case, denoting by $\cc_{3}^{1}$ and $\cc_{3}^{2}$ the concentration of ozone obtained from \ref{test1} and \ref{test2}, respectively, we compute the final average ozone concentration at $50\,\mymeter$ from the road, i.e.\ for $j=60$ and $n=N_{t}$ we compute
\begin{equation*}
	M^{1}=\frac{1}{N_{x}}\sum_{i=1}^{N_{x}}(\cc_{3}^{1})^{N_{t}}_{i60} \qquad\text{and}\qquad
	M^{2}=\frac{1}{N_{x}}\sum_{i=1}^{N_{x}}(\cc_{3}^{2})^{N_{t}}_{i60}.
\end{equation*}
We obtain $M^{1}=7.39\times10^{3}\,\mygram\per\km^{3}$ and $M^{2}=8.1\times 10^{3}\,\mygram\per\km^{3}$, hence the combination of wind and traffic lights causes an increase of about 11\% in the final average concentration of ozone at $50\,\mymeter$ from the road compared to the case of no traffic lights. 

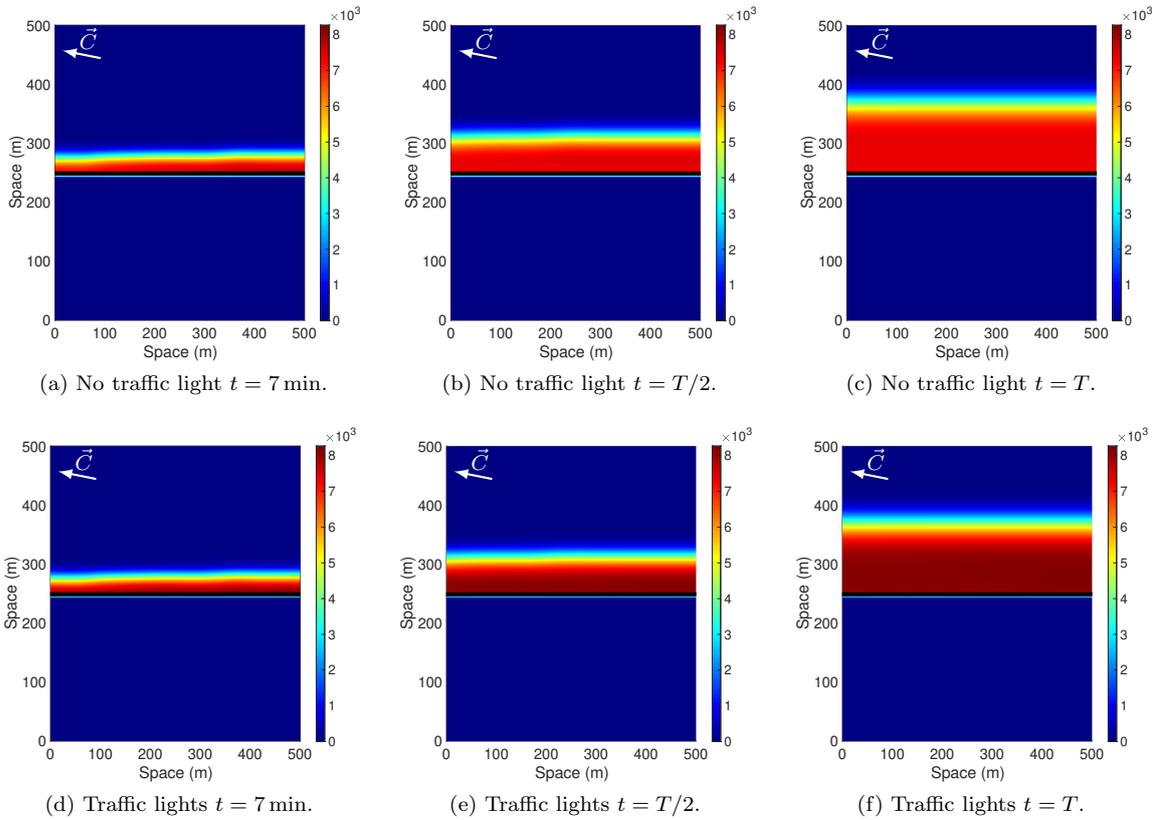
\begin{figure}[h!]
\centering
\subfloat[][No traffic light $t=7\,\mymin$.]{
\begin{overpic}[width=0.3\columnwidth]{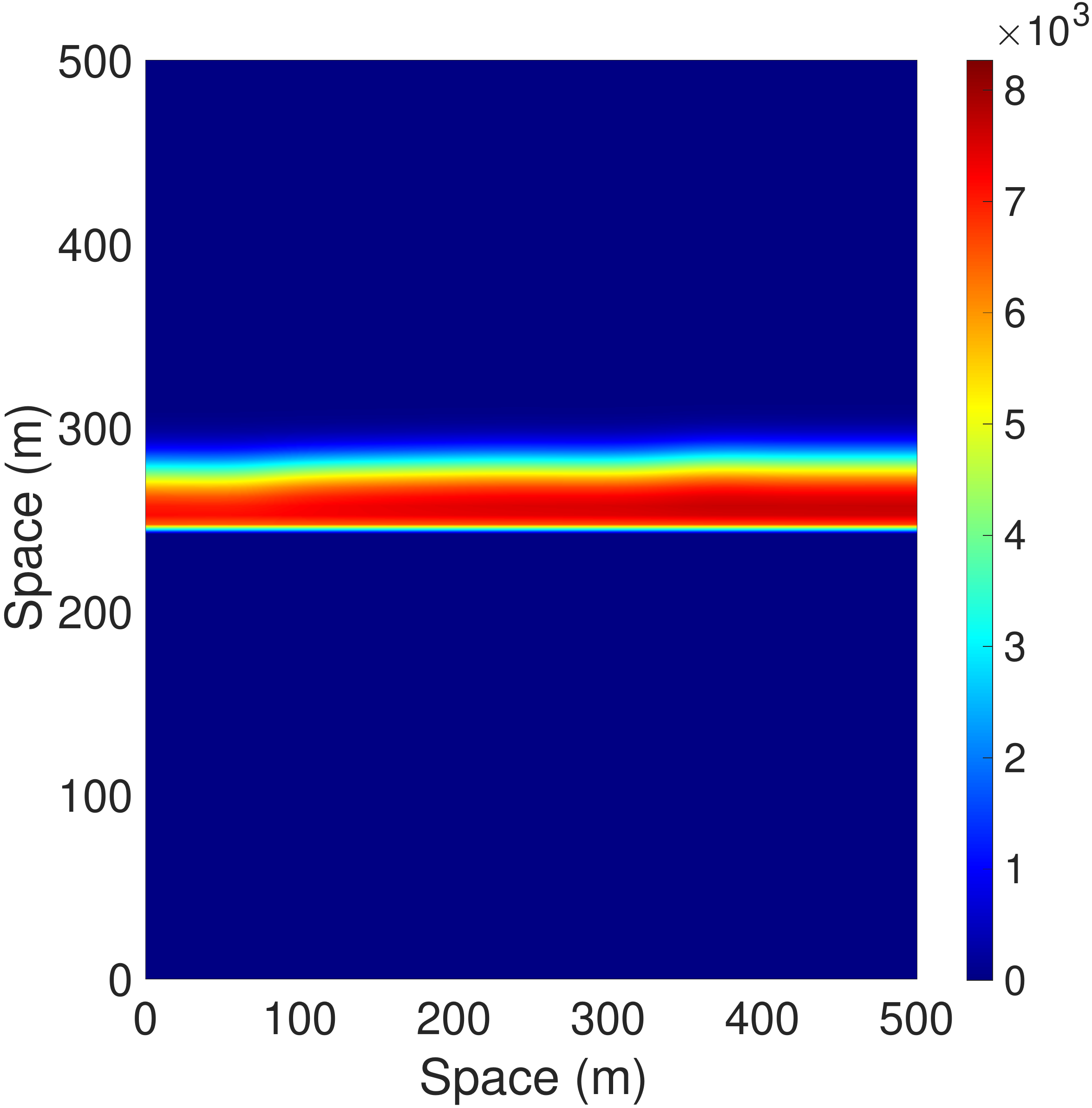}
\put(15,85){\tikz \draw[white,->,thick] (0,0)--(-0.5,0.1);}
\put(20,88){\color{white}\text{\footnotesize$\vento$}}
\put(12.5,52.3){\tikz \draw[black,ultra thick] (0,0)--(3.29,0);}
\end{overpic}
}\quad
\subfloat[][No traffic light $t=T/2$.]{
\begin{overpic}[width=0.3\columnwidth]{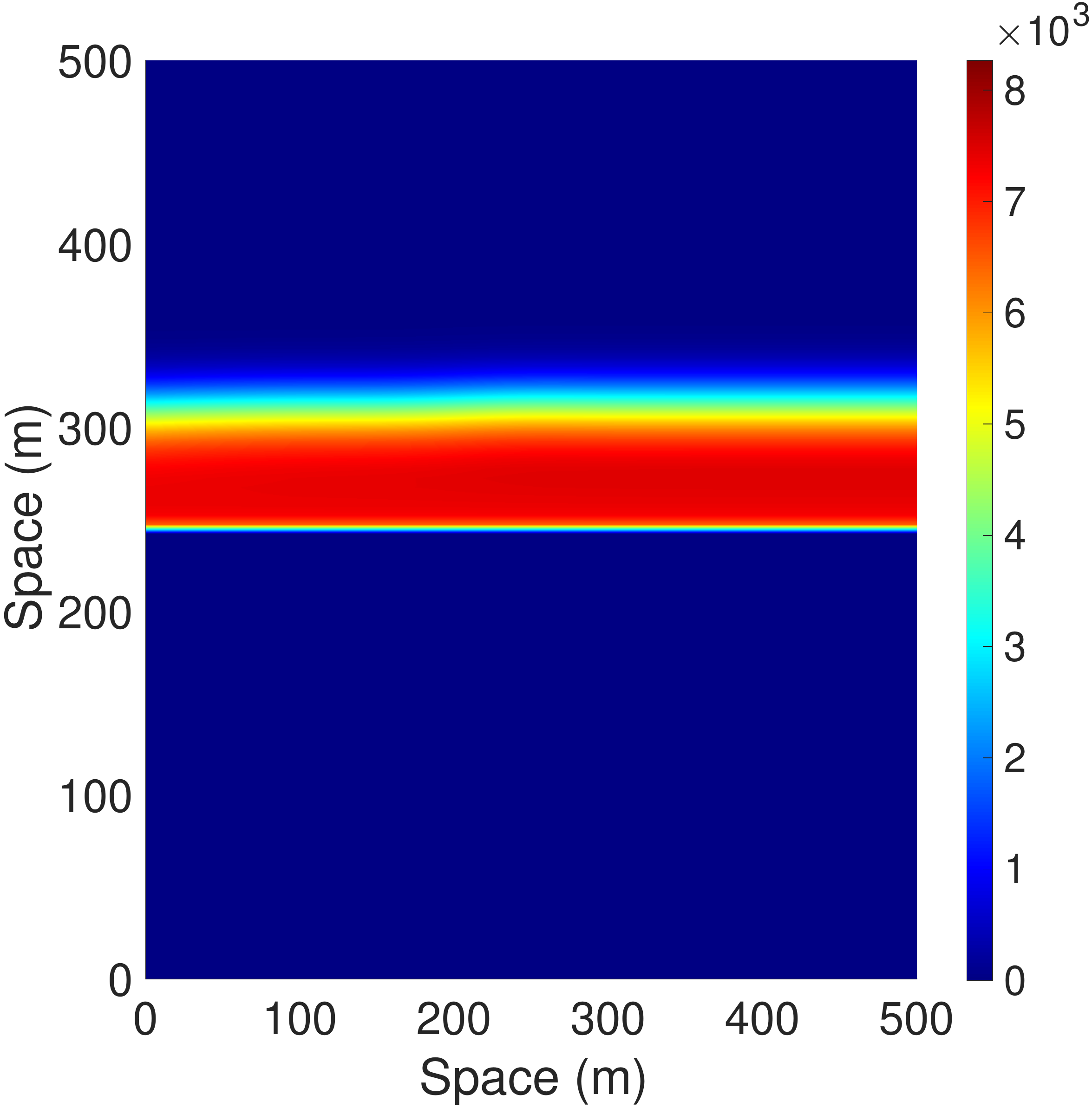}
\put(15,85){\tikz \draw[white,->,thick] (0,0)--(-0.5,0.1);}
\put(20,88){\color{white}\text{\footnotesize$\vento$}}
\put(12.5,52.3){\tikz \draw[black,ultra thick] (0,0)--(3.29,0);}
\end{overpic}
}\quad
\subfloat[][No traffic light $t=T$.]{
\begin{overpic}[width=0.3\columnwidth]{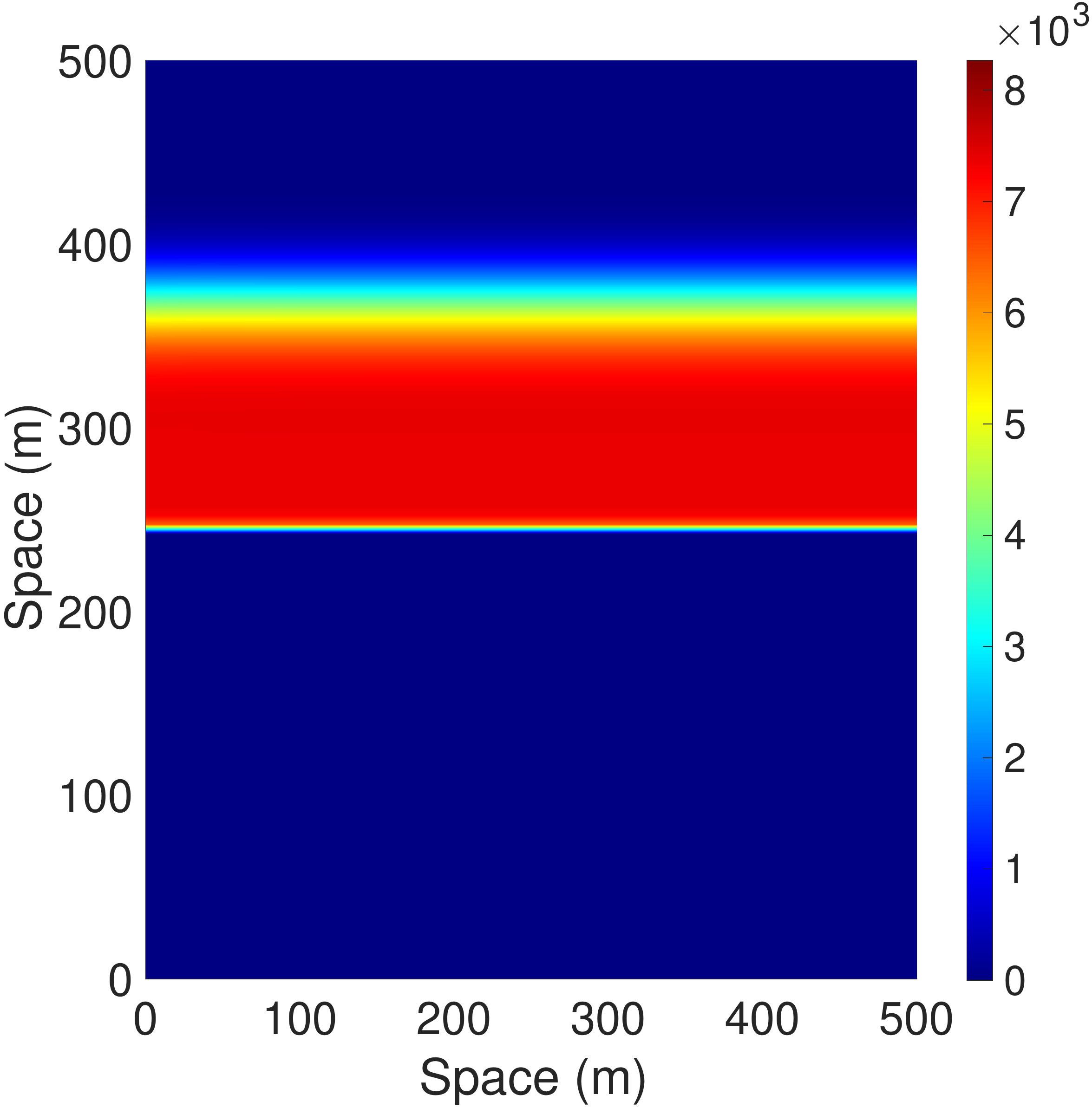}
\put(15,85){\tikz \draw[white,->,thick] (0,0)--(-0.5,0.1);}
\put(20,88){\color{white}\text{\footnotesize$\vento$}}
\put(12.5,52.3){\tikz \draw[black,ultra thick] (0,0)--(3.29,0);}
\end{overpic}}
\\
\subfloat[][Traffic lights $t=7\,\mymin$.]{
\begin{overpic}[width=0.3\columnwidth]{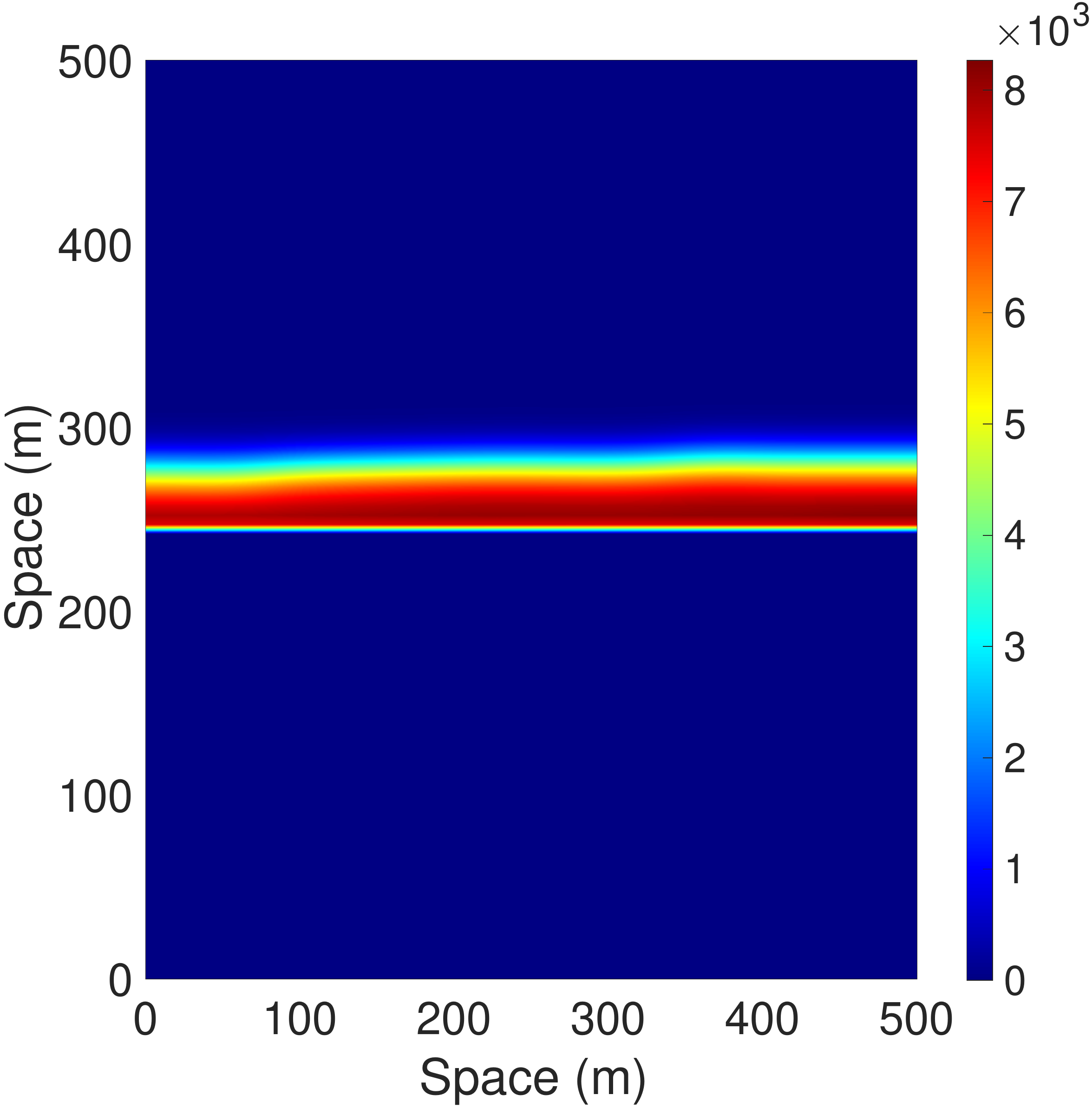}
\put(15,85){\tikz \draw[white,->,thick] (0,0)--(-0.5,0.1);}
\put(20,88){\color{white}\text{\footnotesize$\vento$}}
\put(12.5,52.3){\tikz \draw[black,ultra thick] (0,0)--(3.29,0);}
\end{overpic}
}
\quad
\subfloat[][Traffic lights $t=T/2$.]{
\begin{overpic}[width=0.3\columnwidth]{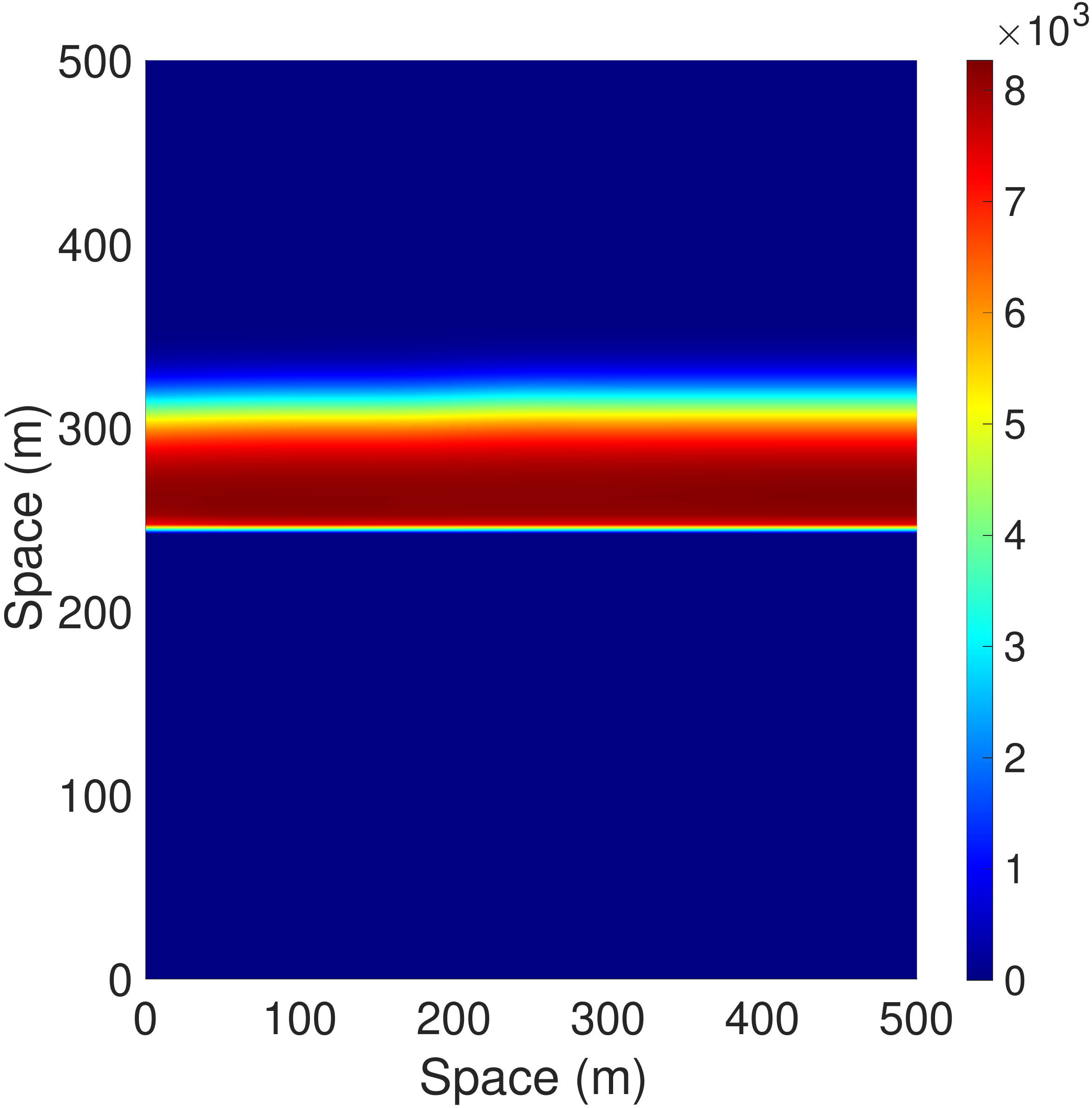}
\put(15,85){\tikz \draw[white,->,thick] (0,0)--(-0.5,0.1);}
\put(20,88){\color{white}\text{\footnotesize$\vento$}}
\put(12.5,52.3){\tikz \draw[black,ultra thick] (0,0)--(3.29,0);}
\end{overpic}
}\quad
\subfloat[][Traffic lights $t=T$.]{
\begin{overpic}[width=0.3\columnwidth]{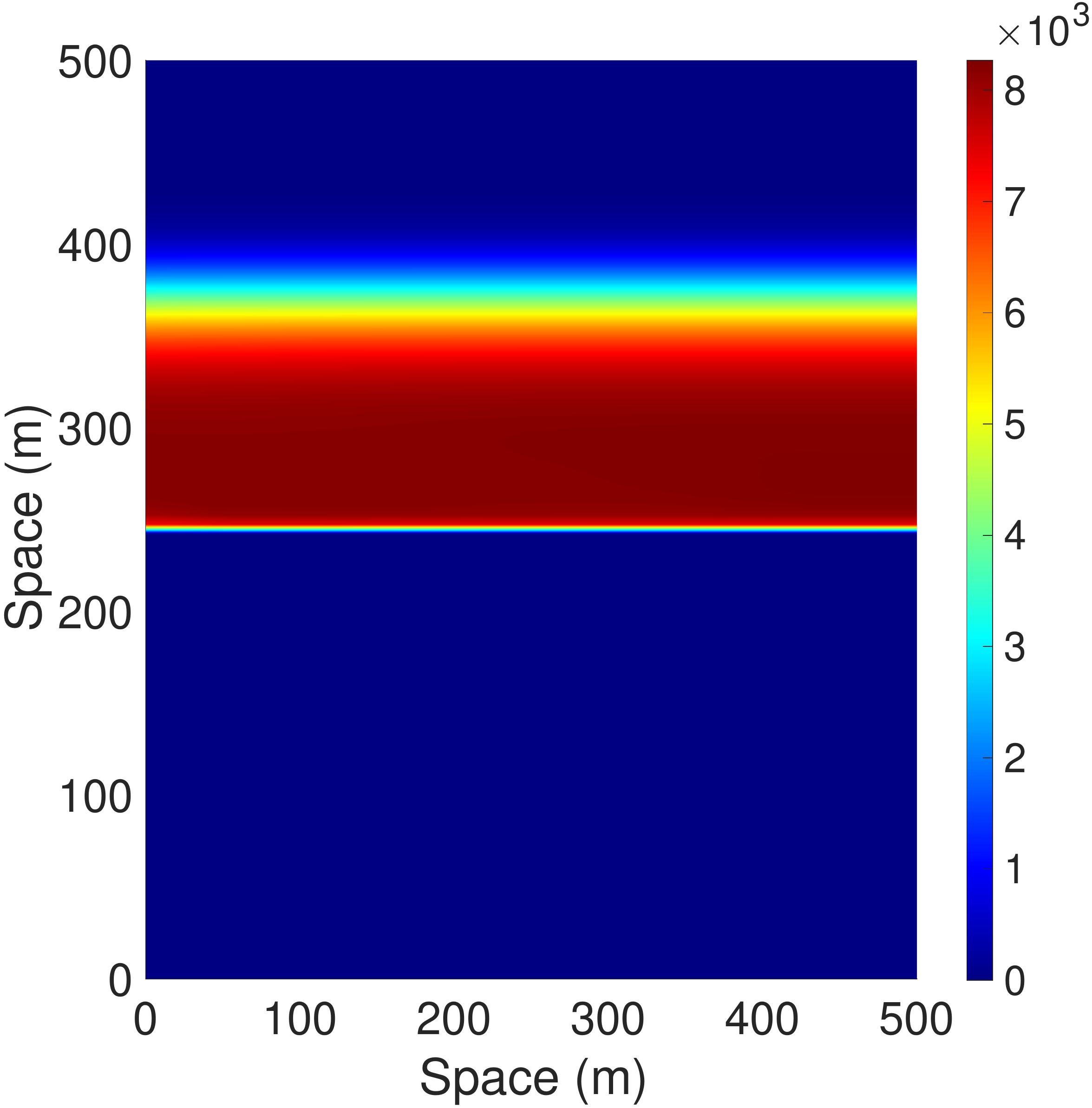}
\put(15,85){\tikz \draw[white,->,thick] (0,0)--(-0.5,0.1);}
\put(20,88){\color{white}\text{\footnotesize$\vento$}}
\put(12.5,52.3){\tikz \draw[black,ultra thick] (0,0)--(3.29,0);}
\end{overpic}
}
\label{fig:diffOriz1}
\caption{Horizontal diffusion of ozone concentration ($\mygram\per\km^{3}$) in $\Omega$ at different times with (bottom) and without (top) traffic lights.}
\label{fig:diffOriz}
\end{figure}

\section{Conclusions}\label{sec:conclusion}

In this paper we proposed to couple a second-order model for traffic with a simplified system of reactions in the atmosphere for ozone production and diffusion.
The coupling is obtained via a general emission model, with parameters specifically tuned on $\nox$ pollutants.
Via numerical simulations we tested various traffic scenarios obtaining three main results: 1)  acceleration waves are most responsible for $\nox$ emissions; 2) the length of traffic cycles impact emissions more than the ratio between green and red light; 3) ozone production and diffusion is strongly impacted by the presence of traffic light.
Future investigations may include extending the model to networks, other pollutants and chemical phenomena, and incorporating more sophisticate diffusion models.

\section*{Acknowledgements}
B.P.'s work was supported by the National Science Foundation under 
Cyber-Physical Systems Synergy Grant No. CNS-1837481. C.B. and M.B. 
would like to thank the Italian Ministry of Instruction, University and Research (MIUR) to support this research with funds coming from PRIN Project 2017 (No. 2017KKJP4X entitled  \textquotedblleft Innovative numerical methods for evolutionary partial differential equations and applications\textquotedblright).

\bibliographystyle{siam}
\bibliography{references_complete}

\end{document}